\documentclass[12pt]{amsart}
\usepackage{standard}
\usepackage{epic,eepic}
\usepackage{comment}
\usepackage{url}

\title[Microlocal analysis]{Microlocal analysis and evolution
  equations: Lecture Notes from 2008 CMI/ETH Summer School\\ \today}

\author{Jared Wunsch}
\address{Department of Mathematics\\Northwestern University\\Evanston
  IL 60208}

\newcommand{\SC}{\text{sc}}
\newcommand{\Psisc}{\Psi_{\SC}}
\newcommand{\WFsc}{\WF_{\SC}}
\newcommand{\Hsc}{H_{\SC}}

\newcommand{\tpsi}{\tilde{\psi}}

\renewcommand{\Im}{\operatorname{Im}}
\renewcommand{\Re}{\operatorname{Re}}

\numberwithin{equation}{section}
\numberwithin{theorem}{section}
\theoremstyle{remark}
\newtheorem{exercise}{Exercise}
\newtheorem{Exercise}[exercise]{Exercise*}
\numberwithin{exercise}{section}
\DeclareMathOperator{\Diff}{Diff}
\newcommand{\fcal}{\mathcal{F}}
\newcommand{\lag}{\mathcal{L}}
\newcommand{\hamvf}{\mathsf{H}}
\DeclareMathOperator{\Vol}{Vol}
\DeclareMathOperator{\Id}{Id}
\DeclareMathOperator{\singsupp}{singsupp}
\DeclareMathOperator{\Tr}{Tr}
\DeclareMathOperator{\liptic}{ell}
\newcommand{\nabb}{\mbox{$\nabla \mkern-13mu /$\,}}
\newcommand{\schwartz}{\mathcal{S}}
\newcommand{\loc}{\text{loc}}
\newcommand{\tot}{\text{tot}}
\newcommand{\sigmatot}{\sigma_{\tot}}
\newcommand{\hsigma}{\hat{\sigma}}
\DeclareMathOperator{\Opl}{Op_{\ell}}
\DeclareMathOperator{\Opr}{Op_{\emph{r}}}
\newcommand{\cl}{\text{cl}}
\newcommand{\KN}{\text{KN}}
\newcommand{\F}{\mathcal{F}}
\newcommand{\Tsc}{{}^{\SC}T}
\newcommand{\Tscstar}{{}^{\SC}T^*}
\newcommand{\Tbarstar}{\overline{T}^*}
\newcommand{\Tbarscstar}{{}^{\SC}\overline{T}^*}
\newcommand{\Vsc}{\mathcal{V}_{\SC}}
\DeclareMathOperator{\bo}{b}
\newcommand{\Vb}{\mathcal{V}_{\bo}}
\DeclareMathOperator{\esssupp}{ess-supp}

\begin{document}
\maketitle
\tableofcontents
\section{Introduction}
The point of these notes, and the lectures from which they came, is
not to provide a rigorous and complete introduction to microlocal
analysis---many good ones now exist---but rather to give a quick and
impressionistic feel for how the subject is used in practice.  In
particular, the philosophy is to crudely axiomatize the
machinery of pseudodifferential and Fourier integral operators, and
then to see what problems this enables us to solve.  The primary 
emphasis is on application of commutator methods to yield
microlocal energy estimates, and on simple parametrix constructions in
the framework of the calculus of Fourier integral operators; the
rigorous justification of the computations is kept as much as possible
inside a black box.  By contrast, the author has found that lecture
courses focusing on a careful development of the inner workings of this
black box can (at least when he is the lecturer) too easily bog down
in technicality, leaving the students with no notion of why one
might suffer through such agonies.  The ideal education, of course,
includes both approaches\dots

A wide range of more comprehensive and careful treatments of this
subject are now available.  Among those that the reader might want to
consult for supplementary reading are \cite{Melrose:notes},
\cite{GrigisSjostrand}, \cite{Shubin}, \cite{Taylor:pseudors},
\cite{Taylor:PDE2}, \cite{DimassiSjostrand}, \cite{Zworski},
\cite{Martinez} (with the last three focusing on the
``semi-classical'' point of view, which is not covered here).
H\"ormander's treatise \cite{Hormander:v1}, \cite{Hormander:v2},
\cite{Hormander:v3}, \cite{Hormander:v4} remains the definitive
reference on many aspects of the subject.

Some familiarity with the theory of distributions (or a willingness to
pick it up) is a prerequisite for reading these notes, and fine
treatments of this material include \cite{Hormander:v1} and
\cite{FriedlanderJoshi}.  (Additionally, an appendix sets out the
notation and most basic concepts in Fourier analysis and distribution
theory.)

Much of the hard technical work in what follows has been shifted
onto the reader, in the form of exercises.  Doing at least some of
them is essential to following the exposition.  The exercises that
are marked with a ``star'' are in general harder or longer than those
without, in some cases requiring ideas not developed here.

The author has many debts to acknowledge in the preparation of these notes.
The students at the CMI/ETH summer school were the ideal audience, and
provided helpful suggestions on the exposition, as well as turning up
numerous errors and inconsistencies in the notes (although many more surely
remain).  Discussions with Andrew Hassell, Michael Taylor, Andr\'as Vasy,
and Maciej Zworski were very valuable in the preparation of these lectures
and notes.  Rohan Kadakia kindly corrected a number of errrors in the final
version of the manuscript.  An error in the statement of
Theorem~\ref{theorem:DG} was kindly pointed out by Amir Vig
(subsequent to the publication of the original version of these notes)
and has been corrected in this version.
Finally, the author wishes to gratefully
acknowledge Richard Melrose, who taught him most of what he knows of this
subject: a strong influence of Melrose's own excellent lecture notes
\cite{Melrose:notes} can surely be detected here.

The author would like to thank the Clay Mathematics Institute and ETH
for their sponsorship of the summer school, and MSRI for its
hospitality in Fall 2008, while the notes were being revised.  The
author also acknowledges partial support from NSF grant DMS-0700318.

\section{Prequel: energy methods and commutators}\label{section:prequel}
This section is supposed to be like the part of an action movie before the
opening credits: a few explosions and a car chase to get you in the right
frame of mind, to be followed by a more careful exposition of plot.

\subsection{The Schr\"odinger equation on $\RR^n$}\label{subsec:introschrodinger}
Let us consider a solution $\psi$ to the \emph{Schr\"odinger equation}
on $\RR \times \RR^n:$
\begin{equation}\label{scheqn}
i^{-1}\pa_t \psi  -\nabla^2 \psi=0.
\end{equation}
The complex-valued ``wavefunction'' $\psi$ is supposed to describe the time-evolution
of a free quantum particle (in rather unphysical units).  We'll use
the notation $\Lap=-\nabla^2$ (note the sign: it makes the operator
positive, but is a bit non-standard).

Consider, for any self-adjoint operator $A,$ the quantity
$$
\ang{A\psi,\psi}
$$ where $\ang{\cdot,\cdot}$ is the sesquilinear $L^2$-inner product on
$\RR^n.$ In the usual interpretation of QM, this is the expectation value
of the ``observable'' $A.$ Since $\pa_t \psi=i\nabla^2 \psi=-i\Lap \psi,$ we can easily
find the time-evolution of the expectation of $A:$
$$
\pa_t \ang{A \psi,\psi}= \ang{\pa_t(A)\psi,\psi}+ \ang{A(-i\Lap) \psi, \psi}+ \ang{A\psi, (-i\Lap) \psi}.
$$
Now, using the self-adjointness of $\Lap$ and the sesquilinearity, we
may rewrite this as
\begin{equation}\label{Ehrenfest}
\pa_t \ang{A \psi,\psi}=\ang{\pa_t(A)\psi,\psi}+ i\ang{[\Lap,A] \psi, \psi}
\end{equation}
where $[S,T]$ denotes the \emph{commutator} $ST-TS$ of two operators
(and $\pa_t(A)$ represents the derivative of the operator itself,
which may have time-dependence).  Note that this computation is a bit
bogus in that it's a formal manipulation that we've done without regard to
whether the quantities involved make sense, or whether the formal
integration by parts (i.e.\ the use of the self-adjointness of $\Lap$)
was justified.  For now, let's just keep in mind that this makes sense
for sufficiently ``nice'' solutions, and postpone the technicalities.

If you want to learn things about $\psi(t,x),$ you might try to use
\eqref{Ehrenfest} with a judicious choice of $A.$ For instance,
setting $A=\Id$ shows that the $L^2$-norm of $\psi(t,\cdot)$ is
conserved.  Additionally, choosing $A=\Lap^k$ shows that the $H^{k}$
norm is conserved (see the appendix for a definition of this norm).
In both these examples, we are using the fact that $[\Lap,A]=0.$

A more interesting example might be the following: set $A=\pa_r,$ the
radial derivative.  We may write the Laplace operator on $\RR^n$ in polar
coordinates as
$$
\Lap=-\pa_r^2-\frac{n-1}r\pa_r+\frac{\Lap_\theta}{r^2}
$$
where $\Lap_\theta$ is the Laplacian on $S^{n-1};$ thus we compute
$$
[\Lap,\pa_r]=2\frac{\Lap_\theta}{r^3}-\frac{(n-1)}{r^2}\pa_r.
$$
\begin{exercise}
Do this computation!  (Be aware that $\pa_r$ is not a differential operator
with smooth coefficients.)
\end{exercise}
This is kind of a funny looking operator.  Note that $\Lap$ is
self-adjoint, and $\pa_r$ wants to be anti-self-adjoint, but isn't
quite.  In fact, it makes more sense to replace $\pa_r$ by
$$A=(1/2)(\pa_r-\pa_r^*)=\pa_r+\frac{n-1}{2r},$$ which corrects $\pa_r$ by a lower-order term
to be anti-self-adjoint.
\begin{exercise}
Show that $$\pa_r^*=-\pa_r-\frac{n-1}r.$$
\end{exercise}
Trying again, we get by dint of a little work:
\begin{equation}\label{morawetzcomm}
[\Lap,\pa_r+\frac{n-1}{2r}]= \frac{2 \Lap_\theta}{r^3}+
\frac{(n-1)(n-3)}{2 r^3},
\end{equation}
\emph{provided $n,$ the dimension, is at least $4.$}
\begin{exercise}
  Derive \eqref{morawetzcomm}, where you should think of both sides as
  operators from Schwartz functions to tempered distributions (see the
  appendix for definitions).  What happens if $n=3$?  If $n=2$?  Be
  very careful about differentiating negative powers of $r$ in the
  context of distribution theory\dots
\end{exercise}

Why do we like \eqref{morawetzcomm}?  Well, it has the very lovely
feature that both summands on the RHS are \emph{positive operators}.
Let's plug this into \eqref{Ehrenfest} and integrate on a finite time
interval:
\begin{align*}
i^{-1}\ang{A\psi,\psi}\big\rvert_0^T &= \int_0^T \ang{\frac{2 \Lap_\theta}{r^3}\psi,\psi}+
\ang{\frac{(n-1)(n-3)}{2 r^3}\psi,\psi}\, dt\\ &=  \int_0^T
2\norm{r^{-1/2}\nabb\psi}^2 \, dt + \frac{(n-1)(n-3)}{2} \norm{r^{-3/2}\psi}^2\, dt,
\end{align*}
where $\nabb$ represents the (correctly scaled) angular gradient:
$\nabb = r^{-1}\nabla_\theta$, where $\nabla_\theta$ denotes the gradient
on $S^{n-1}.$

Now, we're going to turn the way we use this estimate on its head,
relative to what we did with conservation of $L^2$ and $H^k$ norms:
the \emph{left}-hand-side can be estimated by a constant times the
$H^{1/2}$ norm of the initial data.  This should be at least plausible
for the derivative term, since morally, half a derivative can be
dumped on each copy of $u,$ but is complicated by the fact that
$\pa_r$ is not a differential operator on $\RR^n$ with smooth
coefficients.  The following (somewhat lengthy) pair of exercises goes
somewhat far afield from the main thrust of these notes, but is
necessary to justify our $H^{1/2}$ estimate.

In the sequel, we employ the useful notation $f \lesssim g$ to
indicate that $f\leq C g$ for some $C \in \RR^+;$ when $f$ and $g$
are Banach norms of some function, $C$ is always supposed to be
independent of the function.
\begin{Exercise}\label{exercise:interpolation}\
\begin{enumerate}\item
Verify that for $u \in \schwartz(\RR^n)$ with $n \geq 3,$ $\abs{\ang{\pa_r u,u}} \lesssim
\norm{u}_{H^{1/2}}^2.$

\textsc{Hint:} Use the fact that
$$
\pa_r  = \sum \abs{x}^{-1} x^j \pa_{x^j}.
$$
Check that $x/\abs{x}$ is a bounded multiplier on both $L^2$ and
$H^1,$ and hence, by interpolation and duality, on $H^{-1/2}.$  An
efficient treatment of the interpolation methods you will need can be
found in \cite{Taylor:PDE1}.  You
will probably also need to use \emph{Hardy's inequality} (see Exercise~\ref{exercise:hardy}).
\item
Likewise, show that the $\ang{r^{-1}u,u}$ term is bounded by a multiple of
$\norm{u}_{H^{1/2}}^2$ (again, use Exercise~\ref{exercise:hardy}).
\end{enumerate}
\end{Exercise}

\begin{exercise}\label{exercise:hardy}
Prove \emph{Hardy's inequality}: if $u \in
H^1(\RR^n)$ with $n \geq 3,$ then
$$
\frac{(n-2)^2}{4} \int \frac{\abs{u}^2}{r^2} \, dx \leq  \int
\abs{\nabla u }^2 \, dx.
$$

\textsc{Hint:} In polar coordinates, we have for $u \in \schwartz(\RR^n)$
$$
\int \frac{\abs{u}^2}{r^2} \, dx =\int_{S^{n-1}} \int_0^\infty
\abs{u}^2 r^{n-3} \, dr \, d\theta.
$$
Integrate by parts in the $r$ integral, and apply Cauchy-Schwarz.
\end{exercise}

So we obtain, finally, the \emph{Morawetz inequality}: if $\psi_0 \in
H^{1/2}(\RR^n),$ with $n\geq 4$ then
\begin{equation}\label{Morawetz}
2\int_0^T
\norm{r^{-1/2}\nabb\psi}^2 \, dt + \frac{(n-1)(n-3)}{2} \int_0^T
\norm{r^{-3/2}\psi}^2\, dt \lesssim \norm{\psi_0}_{H^{1/2}}^2.
\end{equation}
Now remember that we've been working rather formally, and there's no
guarantee that either of the terms on the LHS is finite a priori.  But
the RHS is finite, so \emph{since both terms on the LHS are positive,
  both must be finite, provided $\psi_0 \in H^{1/2}$.}  (This is a dangerously sloppy way of
reasoning---see the exercises below.)  So we get, at one stroke two
nice pieces of information: if $\psi_0 \in H^{1/2},$ we obtain the
finiteness of both terms on the left.

Let's try and understand these.  The term
$$
\int_0^T\norm{r^{-3/2}\psi}^2\, dt
$$
gives us a weighted estimate, which we can write as
\begin{equation}\label{weightedest}
\psi \in r^{3/2} L^2([0,T]; L^2(\RR^{n}))
\end{equation}
for any $T,$ or, more briefly, as
\begin{equation}\label{condensed}
\psi \in r^{3/2}L^2_\loc L^2.
\end{equation}
(The right side of \eqref{weightedest} denotes the Hilbert space of functions that are of the form $r^{3/2}$
times an element of the space of $L^2$ functions on $[0,T]$ with
values in the Hilbert space $L^2(\RR^n);$ note that whenever we use
the condensed notation \eqref{condensed}, the Hilbert space for the
time variables will precede that for the spatial variables.)  So
$\psi$ can't ``bunch up'' too much at the origin.  Incidentally, our
whole setup was translation invariant, so in fact we can conclude
$$
\psi \in \abs{x-x_0}^{3/2}L^2_\loc L^2
$$
for \emph{any} $x_0 \in \RR^n,$ and $\psi$ can't bunch up too much anywhere
at all.

How about the other term?  One interesting thing we can do is the
following: Choose $x_0,$ $x_1$ in $\RR^n,$ and let $\mathsf{X}$ be a smooth
vector field with support disjoint from the line $\overline{x_0x_1}.$
Then we may write $\mathsf{X}$ in the form
$$
\mathsf{X}=\mathsf{X}_0+\mathsf{X}_1
$$
with $\mathsf{X}_i$ smooth, and $\mathsf{X}_i\perp (x-x_i)$ for
$i=0,1;$ in other words, we split $\mathsf{X}$ into angular vector
fields with respect to the origin of coordinates placed at $x_0$ and
$x_1$ respectively.  Moreover, we can arrange that the coefficients of
$\mathsf{X}_i$ be bounded in terms of the coefficients of $\mathsf{X}$
(provided we bound the support uniformly away from
$\overline{x_0x_1}$).  Thus, we can estimate for any such vector field
$\mathsf{X}$ and any $u \in \mathcal{C}_c^\infty(\RR^n)$
$$
\int \abs{\mathsf{X}u}^2 \, dx \lesssim \int \abs{\abs{x-x_0}^{-1/2}\nabb_0 u}^2 \, dx+\int \abs{\abs{x-x_1}^{-1/2}\nabb_1 u}^2 \, dx
$$
where $\nabb_i$ is the angular gradient with respect to the origin of
coordinates at $x_i.$ Since for a solution of the Schr\"odinger
equation, \eqref{Morawetz} tells us that the time integral of each of
these latter terms is bounded by the squared $H^{1/2}$ norm of the
initial data, we can assemble these estimates with the choices
$\mathsf{X}=\chi \pa_{x^j}$ for any $\chi \in \CcI(\RR^n)$ to obtain
$$
\int_0^T \norm{\chi\nabla \psi}^2 \, dt \lesssim \norm{\psi_0}_{H^{1/2}}^2.
$$
In more compact notation, we have shown that
$$
\psi_0 \in H^{1/2} \Longrightarrow \psi \in L^2_\loc H^1_\loc.
$$
This is called the \emph{local smoothing estimate}.  It says that on
average in time, the solution is locally half a derivative smoother than the
initial data was; one consequence is that in fact, with initial data
in $H^{1/2},$ the solution is in $H^1$ in space at \emph{almost every
  time}.

\begin{exercise}
Work out the Morawetz estimate in dimension $3.$  (This is in many
ways the nicest case.)  Note that our techniques yield no estimate in
dimension $2,$ however.
\end{exercise}

In fact, if all we care about is the local smoothing
estimate (and this is frequently the case) there is an easier
commutator argument that we can employ to get just that estimate.  Let
$f(r)$ be a function on $\RR^+$ that equals $0$ for $r<1,$ is
increasing, and equals $1$ for $r\geq 2.$
Set $A=f(r) \pa_r$ and employ \eqref{Ehrenfest} just as we did before.
The commutant $f(r)\pa_r$ (as opposed to just $\pa_r$) has the virtue
of actually being a smooth vector field on $\RR^n.$  So we can write
$$
[\Lap, f(r)\pa_r] = -2f'(r) \pa_r^2 + 2r^{-3}f(r) \Lap_\theta+ R
$$
where $R$ is a first order operator with coefficients in $\CcI(\RR^n).$
As we didn't bother to make our commutant anti-self-adjoint, we might like
to fix things up now by rewriting
$$
[\Lap, f(r)\pa_r] = -2\pa_r^* f'(r) \pa_r + 2r^{-3} f(r) \Lap_\theta+R'
$$ where $R'$ is of the same type as $R.$ Note that both main terms on
the right are now nonnegative operators, and also that the term
containing $\pa_r^*$ is not, appearances to the contrary, singular at
the origin, owing to the vanishing of $f'$ there.  Thus we obtain, by
another use of \eqref{Ehrenfest},
\begin{multline}
\int_0^T \norm{\sqrt{f'(r)} \pa_r \psi}^2 \, dt+\int_0^T
\norm{\sqrt{f(r)} r^{-1/2}\nabb \psi}^2 \, dt \\ \lesssim \int_0^T
\abs{\ang{R'\psi,\psi}}\, dt + \abs{\ang{f(r) \pa_r \psi,\psi}}|_0^T. 
\end{multline} Now the first term on the RHS is bounded by a multiple of
$\norm{\psi_0}_{H^{1/2}}^2$ (as $R'$ is first order with coefficients
in $\CcI(\RR^n)$); the
second term is likewise (since $f$ is bounded with compactly supported
derivative, and zero near the origin).  This gives us an
estimate of the desired form, valid on any compact subset of $\supp f
\cap \supp f',$ which can be translated to contain any point.

\begin{exercise}\label{exercise:schrodinger}
This exercise is on giving some rigorous underpinnings to some of the
formal estimates above.  It also gets you thinking about the
alternative, Fourier-theoretic, picture of how might think about
solutions to the Schr\"odinger equation.\footnote{If you want to work
  hard, you might try to derive the local smoothing estimate from the
  explicit form of the Schr\"odinger kernel derived below.  It's not
  so easy!}
\begin{enumerate}
\item
Using the Fourier transform,\footnote{See the appendix for a very brief review of the
  Fourier transform acting on tempered distributions and $L^2$-based
  Sobolev spaces.} show that if $\psi_0 \in L^2(\RR^n),$
there exists a unique solution $\psi(t,x)$ to \eqref{scheqn} with
$\psi(0,x)=\psi_0.$
\item
As long as you're at it, use the Fourier transform to derive the
explicit form of the solution: show that 
$$
\psi(t,x) =\psi_0* K_t
$$
where $K_t$ is the ``Schr\"odinger kernel;'' give an explicit formula
for $K_t.$
\item
Use your explicit formula for $K_t$ to show that if $\psi_0 \in L^1$
then $\psi(T,x) \in L^\infty (\RR^n)$ for any $T \neq 0.$
\item Show using the first part, i.e.\ by thinking about the solution
  operator as a Fourier multiplier, that if $\psi_0 \in H^s$ then
  $\psi(t,x) \in L^\infty(\RR_t; H^s),$ hence give another proof that $H^s$ regularity is
  conserved.
\item Likewise, show that the Schr\"odinger evolution in $\RR^n$ takes
  Schwartz functions to Schwartz functions.
\item \emph{Rigorously} justify the Morawetz inequality if $\psi_0 \in
  \schwartz(\RR^n).$  Then use a density argument to rigorously
  justify it for $\psi_0 \in H^{1/2}(\RR^n).$
\end{enumerate}
\end{exercise}

\subsection{The Schr\"odinger equation with a metric}

Now let us change our problem a bit.  Say we are on an $n$-dimensional
manifold, or even just on $\RR^n$ endowed with a complete
non-Euclidean Riemannian metric $g.$ There is a canonical choice for
the Laplace operator in this setting:
$$
\Lap=d^*d
$$
where $d$ takes functions to one-forms, and the adjoint is with
respect to $L^2$ inner products on both (which of course also involve
the volume form associated to the Riemannian metric).  This yields, in coordinates,
\begin{equation}\label{Laplacian}
\Lap=-\frac{1}{\sqrt{g}} \pa_{x^i} g^{ij} \sqrt{g} \pa_{x^j},
\end{equation}
where $\sum_{i,j=1}^n g^{ij} \pa_{x^i} \otimes \pa_{x^j}$ is the dual metric on forms
(hence $g^{ij}$ is the inverse matrix to $g_{ij}$) and $g$ denotes
$\det(g_{ij}).$
\begin{exercise}
Check this computation!
\end{exercise}
\begin{exercise}
Write the Euclidean metric on $\RR^3$ in spherical coordinates, and use
\eqref{Laplacian} to compute the Laplacian in spherical coordinates.
\end{exercise}

We can now consider the Schr\"odinger equation with the Euclidean
Laplacian replaced by this new ``Laplace-Beltrami'' operator.  By
standard results in the spectral theory of self-adjoint
operators,\footnote{The operator $\Lap$ is manifestly formally
  self-adjoint, but in fact turns out to be essentially self-adjoint
  on $\CcI(X)$ for $X$ any complete manifold.}
there is still a solution in $L^\infty(\RR; L^2)$ given any $L^2$
initial data---this generalizes our Fourier transform computation in
Exercise~\ref{exercise:schrodinger}---but its form and its
properties are much harder to read off.

Computing commutators with this operator is a little trickier than in
the Euclidean case, but certainly feasible; you might certainly try
computing $[\Lap, \pa_r+(n-1)/(2r)]$ where $r$ is the distance from
 some fixed point.
\begin{exercise}
Write out the Laplace operator in Riemannian polar coordinates, and
compute $[\Delta, \pa_r+(n-1)/(2r)]$ near $r=0.$
\end{exercise}

But what happens when we get beyond the injectivity radius?  Of
course, the $r$ variable doesn't make any sense any more.  Moreover,
if we try to think of $\pa_r$ as the operator of differentiating
``along geodesics emanating from the origin'' then at a conjugate
point to $0,$ we have the problem that we're somehow supposed to be be
simultaneously differentiating in two different directions.  One fix
for this problem is to employ the calculus of pseudodifferential operators,
which permits us to construct operators that behave differently
depending on what direction we're looking in: we can make operators
that separate out the different geodesics passing through the
conjugate point, and do different things along them.

\subsection{The wave equation}\label{section:wave}

Let
$$
\Box u\equiv (\pa_t^2+\Lap)u =0
$$ denote the wave equation on $\RR\times \RR^n$ (recall that
$\Lap=-\sum \pa_{x^i}^2$). For simplicity of
notation, \emph{let us consider only real-valued solutions in this section}.

The usual route to thinking about the energy of a solution to the wave
equation is as follows.  We consider the integral
\begin{equation}\label{energyint}
0=\int_0^T \ang{\Box u, \pa_t u} \, dt
\end{equation}
where $\ang{\cdot,\cdot}$ is the inner product on $L^2(\RR^n).$  Then integrating
by parts in $t$ and in $x$ gives the conservation of
$$
\norm{\pa_t u}^2+\norm{\nabla u}^2.
$$

We can recast this formally as a commutator argument, if we like, by
considering the commutator with the indicator function of an interval:
$$
0=\int_{\RR} \ang{[\Box,1_{[0,T]}(t)\pa_t] u, u} \, dt.
$$
The integral vanishes, at least formally, by self-adjointness of
$\Box$---it is in fact a better idea to think of this whole thing as
an inner product on $\RR^{n+1}:$
$$
\ang{[\Box,1_{[0,T]}(t)] \pa_tu, u}_{\RR^{n+1}}.
$$
Having gone this far, we might like to replace the indicator function
with something smooth, to give a better justification for this formal integration by
parts; let $\chi(t)$ be a smooth approximator to the indicator
function with $\chi'=\phi_1-\phi_2$ with $\phi_1$ and $\phi_2$
nonnegative bump functions supported respectively in $(-\ep,\ep)$ and
$(T-\ep,T+\ep),$ with $\phi_2(\cdot)=\phi_1(\cdot-T)$  Let $A=\chi(t)
\pa_t+\pa_t \chi(t).$  Then we have $$[\Box,A]=2 \pa_t
\chi'\pa_t+ \pa_t^2 \chi'+\chi'\pa_t^2,$$ and by
(formal) anti-self-adjointness of $\pa_t$ (and the fact that $u$ is
assumed real),
\begin{align*}
0=\ang{[\Box,A]u,u}_{\RR^{n+1}} 
&= -2\ang{\chi' \pa_t u, \pa_t u}_{\RR^{n+1}}  +2\ang{\chi'u, \pa_t^2 u}_{\RR^{n+1}} \\
&= -2\ang{\chi' \pa_t u, \pa_t u}_{\RR^{n+1}}  +2\ang{\chi'u, \nabla^2 u}_{\RR^{n+1}} \\
&= -2\ang{\chi' \pa_t u, \pa_t u}_{\RR^{n+1}}  -2\ang{\chi'\nabla, \nabla u}_{\RR^{n+1}} \\
&= - 2\int_{\RR^{n+1}} \phi_1(t)
  (\abs{u_t}^2+\abs{\nabla u}^2)\, dt\, dx\\
&+ 2\int_{\RR^{n+1}} \phi_2(t)
  (\abs{u_t}^2+\abs{\nabla u}^2)\, dt\, dx.
\end{align*}
Thus, the energy on the time interval $[T-\ep,T+\ep]$ (modulated by
the cutoff $\phi_2$) is the same as
that in the time interval $[-\ep,\ep]$ (modulated by $\phi_1$).

We can get fancier, of course.  Finite propagation speed is usually
proved by considering the variant of \eqref{energyint}
$$
\int_{-T_2}^{-T_1}\int_{\abs{x}^2\leq t^2} \Box u\, \pa_t u \, dx \, dt,
$$
with $0<T_1<T_2.$
Integrating by parts gives negative boundary terms, and we find that
the energy in 
$$
\{t=-T_1,\abs{x}^2 \leq T_1^2\}
$$
is bounded by that in
$$
\{t=-T_2,\abs{x}^2 \leq T_2^2\}.
$$
Hence if the solution has zero Cauchy data (i.e.\ value,
time-derivative) on the latter surface, it
also has zero Cauchy data on the former.
\begin{exercise}
Go through this argument to show finite propagation speed.
\end{exercise}

Making this argument into a commutator argument is messier, but still possible:
\begin{Exercise}
Write a positive commutator version of the proof of finite propagation
speed, using smooth cutoffs instead of integrations by parts.  (An
account of energy estimates with smooth temporal cutoffs, in the general setting of
Lorentzian manifolds, can be found in \cite[Section 3]{Vasy:AdS}.)
\end{Exercise}

%% We can again replicate this argument by considering the commutator
%% $$
%% 0=\ang{[\Box, A] u,u}_{\RR^{n+1}}
%% $$
%% where $A=\chi(t) \phi(t^2-x^2);$ here $\chi$ is a cutoff in
%% $[-T_2-\ep,-T_1+\ep]$ analogous to that used above, while $\phi$ is a
%% cutoff of the form
%% $$
%% \phi(s) = \begin{cases} e^{-1/s} & s>0\\ 0 & s\leq 0.
%% \end{cases}
%% $$
%% But this gets a bit messy.

There is of course also a Morawetz estimate for the wave equation!
(Indeed, this was what Morawetz originally proved.)
\begin{Exercise}
Derive (part of) the Morawetz estimate: Let $u$ solve
$$
\Box u=0, (u, \pa_t u)|_{t=0}= (f,g)
$$
on $\RR^n,$ with $n\geq 4.$ Show that
$$
\norm{r^{-3/2} u}_{L^2_\loc(\RR^{n+1})} \lesssim \norm{f}_{H^1}^2+ \norm{g}_{L^2}^2;
$$
this is analogous to the weight part of the Morawetz estimate we
derived for the Schr\"odinger equation.  There is in fact no need for
the local $L^2$ norm---the global spacetime estimate works too:
prove this estimate, and use it to draw a conclusion about the
long-time decay of a solution to the wave equation with Cauchy data
in $\CcI(\RR^n)\oplus \CcI(\RR^n).$

\textsc{Hint}: consider $\ang{[\Box, \chi(t)(\pa_r + (n-1)/(2r))]u,u}_{\RR^{n+1}}.$
\end{Exercise}
\section{The pseudodifferential calculus}

Recall that we hoped to describe a class of operators enriching the
differential operators that would, among other things, enable us to
deal properly with the local smoothing estimate on manifolds, where
conjugate points caused our commutator arguments with ordinary
differential operators to break down.  One solution to this problem
turns out to lie in the calculus of pseudodifferential operators.

\subsection{Differential operators}
What kind of a creature is a pseudodifferential operator?  Well, first
let's think more seriously about \emph{differential} operators.  A linear
differential operator of order $m$ is something of the form
\begin{equation}\label{diffop}
P=\sum_{\abs{\alpha} \leq m} a_\alpha (x) D^\alpha
\end{equation}
where $D_j=i^{-1} (\pa/\pa {x^j})$ and we employ ``multiindex notation:''
$$
D^\alpha=D_{1}^{\alpha_1}\dots D_{n}^{\alpha_n},
$$
$$
\abs{\alpha}=\sum \alpha_j.
$$
We will always take our coefficients to be smooth:
$$
a_\alpha \in \CI(\RR^n).
$$
We let
$$
\Diff^m(\RR^n)
$$
denote the collection of all differential operators of order $m$ on
$\RR^n$ (and will later employ the analogous notation on a manifold).

If $P \in \Diff^m(\RR^n)$ is given by \eqref{diffop}, we can associate
with $P$ a function by formally turning differentiation in $x^j$ into
a formal variable $\xi_j$ with $(\xi_1,\dots,\xi_n) \in \RR^n:$
$$
p(x,\xi)=\sum a_\alpha(x) \xi^\alpha.
$$
This is called the ``total (left-) symbol'' of $P;$ of course, knowing
$p$ is equivalent to knowing $P.$ Note that $p(x,\xi)$ is a rather
special kind of a function on $\RR^{2n}:$ it is actually polynomial in
the $\xi$ variables with smooth coefficients.  Let us write
$p=\sigmatot(P).$

Note that $$\sigmatot:P \mapsto p$$ is \emph{not} a ring homomorpism: we have
$$
PQ=\sum_{\alpha,\beta} p_\alpha(x) D^\alpha q_\beta(x) D^\beta,
$$
and if we expand out this product to be of the form
$$
\sum_\gamma c_\gamma(x) D^\gamma,
$$
then the coefficients $c_\gamma$ will involve all kinds of derivatives
of the $q_\beta$'s.  This is a pain, but on the other hand life would
be pretty boring if the ring of differential operators were
commutative.

If we make do with less, though, composition of operators doesn't
look so bad.  We let $\sigma_m(P),$ the \emph{principal symbol} of
$P,$ just be the symbol of the top-order parts of $P:$
$$
\sigma_m(P) =\sum_{\abs{\alpha}=m} a_\alpha(x) \xi^\alpha.
$$ Note that $\sigma_m(P)$ is a \emph{homogeneous} degree-$m$ polynomial in
$\xi,$ i.e., a polynomial such that
$\sigma_m(P)(x,\lambda\xi)=\lambda^m\sigma_m(P)(x,\xi)$ for $\lambda \in
\RR.$  As a result, we can reconstruct it from its value at
$\abs{\xi}=1,$ and it makes sense for many purposes to just consider it as
a (rather special) smooth function on $\RR^n \times S^{n-1}.$
It turns out to make more invariant sense to regard the principal
symbol as a homogeneous polynomial on $T^*\RR^n,$ so that once we have
scaled away the action of $\RR^+,$ we may regard it as a function on
$S^*\RR^n,$ the \emph{unit} cotangent bundle of
$\RR^n,$ which is simply defined as
$T^*\RR^n/\RR^+$ (or identified with the bundle of unit covectors in,
say, the Euclidean metric).
To clarify when we are talking about the symbol on $S^*\RR^n,$ we
define\footnote{The reader is warned that this notation is not a
  standard one.}
$$
\hsigma_m(P) = \sigma_m(P)|_{\abs{\xi}=1} \in \CI(S^*\RR^n).
$$

Now it \emph{is} the case that the principal symbol is a homomorphism:
\begin{proposition}
For $P,Q$ differential operators of order $m$ resp.\ $m',$
$$
\sigma_{m+m'} (PQ)= \sigma_m(P) \sigma_{m'}(Q).
$$
(and likewise with $\hsigma$).
\end{proposition}
\exercise{Verify this!}

Moreover, the principal symbol has another lovely property that the
total symbol lacks: it behaves well under change of variables.  If
$y=\phi(x)$ is a change of variables, with $\phi$ a diffeomorphism,
and if $P$ is a differential operator in the $x$ variables, we can of course define
a pushforward of $P$ by
$$
(\phi_* P) f= P(\phi^* f)
$$
Then in particular,
$$
\phi_* (D_{x^j}) = \sum_k\frac{\pa y^k}{\pa x^j}D_{y^k},
$$
hence 
%%%%%%%% This bit still a mess!
$$
\phi_*(D_x^\alpha)
=D_{x^1}^{\alpha_1}\dots D_{x^n}^{\alpha_n}
 = \left(\sum_{k_1=1}^n\frac{\pa y^{k_1}}{\pa x^1}D_{y^{k_1}}\right)^{\alpha_1}\dots \left(\sum_{k_n=1}^n\frac{\pa y^{k_n}}{\pa x^n}D_{y^{k_n}}\right)^{\alpha_n};
$$
when we again try to write this in our usual form, as a sum of
coefficients times derivatives, we end up with a hideous mess
involving high derivatives of the diffeomorphism $\phi.$  \emph{But,}
if we restrict ourselves to dealing with principal symbols alone, the
expression simplifies in both form and (especially) interpretation:
\begin{proposition}
If $P$ is a differential operator given by \eqref{diffop}, and
$y=\phi(x),$ then
$$
\sigma_m(\phi_* P)(y,\eta) = \sum_{\abs{\alpha}=m}
a_\alpha(\phi^{-1}(y))
\left(\sum_{k_1=1}^n\frac{\pa y^{k_1}}{\pa
    x^1}\eta_{{k_1}}\right)^{\alpha_1}\dots
\left(\sum_{k_n=1}^n\frac{\pa y^{k_n}}{\pa
    x^n}\eta_{{k_n}}\right)^{\alpha_n} 
$$
where $\eta$ are the new variables ``dual'' to the $y$ variables.
\end{proposition}
This corresponds exactly to the behavior of a function defined on the
cotangent bundle: if $\phi$ is a diffeomorphism from $\RR^n_x$ to
$\RR^n_y,$ then it induces a map $\Phi=\phi^*: T^*\RR_y^n \to T^*\RR_x^n$, and 
$$
\sigma_m(\phi_* P) = \Phi^* (\sigma_m(P)).
$$
\exercise{Prove the proposition, and verify this interpretation of
  it.}

Notwithstanding its poor properties, it is nonetheless a useful fact
that the map
$$
\sigmatot: P \mapsto p
$$
is one-to-one and onto polynomials with smooth coefficients; it
therefore has an inverse, which we shall denote
$$
\Opl: p \mapsto P,
$$
taking functions on $T^*\RR^n$ that happen to be polynomial in the
fiber variables to differential operators on $\RR^n.$ $\Opl$ is called a
``quantization'' map.\footnote{It is far from unique, as will become
  readily apparent.}  You may wonder about the $\ell$ in the subscript:
it stands for ``left,'' and has to do with the fact that we chose to
write differential operators in the form \eqref{diffop} instead of as
$$
P=\sum_{\abs{\alpha} \leq m}  D^\alpha a_\alpha (x),
$$
with the coefficients on the right.  This would have changed the definition
of $\sigmatot$ and hence of its inverse.

Note that $\Opl(x^j)=x^j$ (i.e.\ the operation of multiplication by $x^j$)
while $\Opl(\xi_j) = D_j.$

Why not, you might ask, try to extend this quantization map to a more
general class of functions on $T^*\RR^n$? This is indeed how we obtain
the calculus of pseudodifferential operators.  The tricky point to
keep in mind, however, is that for most purposes, it is asking too
much to deal with the quantizations of all possible functions on
$T^*\RR^n,$ so we'll deal only with a class of functions that are
somewhat akin to polynomials in the fiber variables.

\subsection{Quantum mechanics}\label{subsection:QM}
One reason why you might care about the existence of a quantization map,
and give it such a suggestive name, lies in the foundations of quantum
mechanics.

It is helpful to think about $T^*\RR^n$ as being a classical
\emph{phase space,} with the $x$ variables (in the base) being
``position'' and the $\xi$ variables (the fiber variables) as
``momenta'' in the various directions.  The general notion of
\emph{classical mechanics} (in its Hamiltonian formulation) is as
follows: The state of a particle is a point in the phase space
$T^*\RR^n,$ and moves along some curve in $T^*\RR^n$ as time evolves;
an \emph{observable} $p(x,\xi)$ is a function on the phase space that
we may evaluate at the state $(x,\xi)$ of our particle to give a
number (the observation).  By contrast, a \emph{quantum} particle is
described by a complex-valued \emph{function} $\psi(x)$ on $\RR^n,$
and a quantum observable is a self-adjoint \emph{operator} $P$ acting
on functions on $\RR^n.$ Doing the same measurement repeatedly on
identically prepared quantum states is not guaranteed to produce the
same number each time, but at least we can talk about the
\emph{expected value} of the observation, and it's simply
$$
\ang{P\psi,\psi}_{L^2(\RR^n)}.
$$
In the early development of quantum mechanics, physicists sought a way
to transform the classical world into the quantum world, i.e.\ of
taking functions on $T^*\RR^n$ to operators on\footnote{Well, they are
  not necessarily going to be defined on all of $L^2;$ the technical
  subtleties of unbounded self-adjoint operators will mostly not
  concern us here, however.} $L^2(\RR^n).$  This is, loosely speaking,
the process of ``quantization.''

We now turn to the question of describing the dynamics in the quantum
and classical worlds.
To describe how the point in phase space corresponding to a classical
particle in Hamiltonian mechanics evolves in time, we use the notion
of the
``Poisson bracket'' of two observables. In
coordinates, we can explicitly define
$$
\{f,g\} \equiv \sum \frac{\pa f}{\pa \xi_j} \frac{\pa g}{\pa x^j}-\frac{\pa f}{\pa x^j} \frac{\pa g}{\pa \xi_j}
$$
(this in fact makes invariant sense on any symplectic manifold).  The
map $g \mapsto \{f,g\}$ defines a vector field\footnote{We use the
  geometers' convention of identifying a vector and the
  directional derivative along it.} (the Hamilton vector
field) associated to $f:$
$$
\hamvf_f= \sum \frac{\pa f}{\pa \xi_j} \frac{\pa }{\pa x^j}-\frac{\pa f}{\pa x^j} \frac{\pa }{\pa \xi_j}
$$
The classical time-evolution is along the flow generated by the
Hamilton vector field associated to the \emph{energy function} of our
system, i.e.\ the flow along $\hamvf_h$ for some given $h \in
\CI(T^*\RR^n)$.  By contrast, the wavefunction for a quantum particle
evolves in time according to the Schr\"odinger equation
\eqref{scheqn}, with $-\nabla^2$ in general replaced by a self-adjoint
``Hamiltonian operator'' $H$ whose principal symbol is the energy
function $h.$\footnote{For honest physical applications, one really
  ought to introduce the semi-classical point of view here, carrying
  Planck's constant along as a small parameter and using an associated
  notion of principal symbol.}  By a mild generalization of
\eqref{Ehrenfest}, the time derivative of the expectation of an
observable $A$ is related to the commutator
$$
[H,A].
$$
One of the essential features of quantum mechanics is that
$$
\sigma_{m+m'}([H,A]) =i\{ \sigma_m(H),\sigma_{m'}(A)\},
$$
so that the time-evolution of the quantum observable $A$ is related to the
classical evolution of its symbol along the Hamilton flow;
this is the ``correspondence principle'' between classical and quantum
mechanics.\footnote{In the semi-classical setting, the correspondence
  principle tells that we can in a sense recover CM from QM in the
  limit when Planck's constant tends to zero.  What we have in this
  setting is a correspondence principle that works at high energies,
  i.e.\ in doing computations with high-frequency waves.}

\subsection{Quantization}

How might we construct a quantization map extending the usual
quantization on fiber-polynomials?

Let $\F$ denote the Fourier transform (see Appendix for details).
Then we may write, on $\RR^n,$
\begin{align*} (D_{x^j} \psi)(x) &= \F^{-1} \xi_j \F u = (2\pi)^{-n}\int e^{ix\cdot \xi}
\xi_j \int e^{-iy\cdot\xi} \psi(y)\, dy\, d\xi\\ &= \frac{1}{2\pi} \iint \xi_j
e^{i(x-y)\cdot\xi}\psi(y) \, dy\, d\xi
\end{align*}
Likewise, since $\F^{-1} \F = I,$ we of course have
$$
(x^j\psi)(x) = (2\pi)^{-n} \iint x^j
e^{i(x-y)\cdot\xi}\psi(y) \, dy\, d\xi
$$
Going a bit further, we see that at least for a fiber polynomial $a(x,\xi) = \sum
a_\alpha(x)  \xi^\alpha$ we have
\begin{equation}\label{quantization}
(\Opl(a)\psi)(x)= \sum a_\alpha(x) D^\alpha \psi(x) =  (2\pi)^{-n}\iint a(x,\xi)
e^{i(x-y)\cdot \xi}\psi(y) \, dy\, d\xi;
\end{equation}
stripping away the function $\psi,$ we can also simply write the
Schwartz kernel (see Appendix) of the operator $\Opl(a)$ as
$$
\kappa\big(\Opl(a)\big)= (2\pi)^{-n}\int a(x,\xi)
e^{i(x-y)\cdot \xi}\, d\xi.
$$
(Making sense of the integrals written above is not entirely trivial:
Given $\psi \in \schwartz(\RR^n),$ we can make sense of
the $\xi$ integral in \eqref{quantization}, which looks (potentially) divergent, by observing that
$$(1+\abs{\xi}^2)^{-k}(1+\Lap_y)^k e^{i(x-y)\cdot \xi}=e^{i(x-y)\cdot \xi}$$ for
all $k \in \NN;$ repeatedly integrating by parts in $y$ then moves the
derivatives onto $\psi.$ This method brings down an arbitrary negative
power of $(1+\abs{\xi}^2)$ at the cost of differentiating $\psi,$ thus
making the $\xi$ integral convergent.\footnote{This kind of
  integration by parts argument is ubiquitous in the subject, and
  somewhat scanted in these notes, relative to its true importance.}
Similar arguments yield continuity of $\Opl(a)$ as a map
$\schwartz(\RR^n) \to \schwartz(\RR^n),$ hence we can extend to let $\Opl(a)$ act on $\psi
\in \schwartz'$ by duality.  For more details, cf.\
\cite{Melrose:notes}.)
\begin{Exercise}
Verify the vague assertions in the parenthetical remark above.  You
may wish to consult, for example, the beginning of \cite{Hormander:FIO1}.
\end{Exercise}

This of course suggests that we use \eqref{quantization} as the
\emph{definition} of $\Opl(a)$ for more general observables (``symbols'') $a.$
And we do.  In $\RR^n,$ we set
\begin{equation}\label{quant2}
(\Opl(a)\psi)(x) = \frac{1}{(2\pi)^n} \int a(x,\xi)
e^{i(x-y)\cdot \xi}\psi(y) \, dy\, d\xi.
\end{equation}
We can define the pseudodifferential operators on $\RR^n$ to be just the range
of this quantization map on some reasonable set of symbols $a,$ to be discussed
below.

On a Riemannian manifold, we can make similar constructions global
by cutting off near the diagonal and using the exponential map and its
inverse.  The pseudodifferential operators are those whose Schwartz
kernels\footnote{For some remarks on the Schwartz kernel theorem, see
  the Appendix.}
near the diagonal look like \eqref{quant2} in local coordinates, and that
away from the diagonal are allowed to be arbitrary functions in
$\CI(X \times X).$  If the manifold is noncompact, we will
often assume further that operators are \emph{properly supported}, i.e.\
that both left- and right-projection give proper maps from the support
of the Schwartz kernel to $X.$

\subsection{The pseudodifferential calculus}

\begin{definition} A function $a$ on $T^*\RR^n$ is a \emph{classical symbol}
  of order $m$ if
\begin{itemize}
\item $a \in \CI(T^*\RR^n)$
\item On $\abs{\xi}>1,$ we have
$$
a(x,\xi)= \abs{\xi}^m \tilde{a}(x,\hat\xi, \abs{\xi}^{-1}),
$$
where $\tilde a$ is a smooth function on $\RR^n_x \times
S^{n-1}_{\hat{\xi}} \times \RR^+,$ and
$$
\hat \xi=\frac\xi{\abs{\xi}} \in S^{n-1}.
$$
\end{itemize}
We then write $a \in S^m_\cl (T^*\RR^n).$
\end{definition}

It is convenient to introduce the notation
$$
\ang{\xi}=(1+\abs{\xi}^2)^{1/2},
$$
so that $\ang{\xi}$ behaves like $\abs{\xi}$ near infinity, but is
smooth and nonvanishing at $0.$
A fancy way of saying that $a$ is a classical symbol of order $m$ is
thus to simply say that $a$ is equal to $\ang{\xi}^m$ times
\emph{a smooth function on the fiberwise radial compactification of
  $T^*\RR^n$,} denoted $\Tbarstar \RR^n.$ This compactification is
defined as follows: We can
  diffeomorphically identify $\RR_\xi^n$ with the interior of the unit
  ball by first mapping it to the upper hemisphere of $S^n \subset \RR^{n+1}$
  by mapping \begin{equation}\label{RC}\xi \mapsto \left(\frac \xi{\ang{\xi}},
    \frac 1{\ang{\xi}}\right)\end{equation} and identifying this latter
  space with the interior of the ball.  Then $1/\ang{\xi}$ becomes a
  \emph{boundary defining function,} i.e.\ one that cuts out the
  boundary nondegenerately as its zero-set; $1/\abs{\xi}$ is also a
  valid boundary defining function near the boundary of the ball,
  i.e.\ away from its singularity.

A very important consequence is that we can write a Taylor series
for $a$ near $\abs{\xi}^{-1}=0$ (the ``sphere at infinity'') to obtain
$$
a(x,\xi) \sim \sum_{j=0}^\infty a_{m-j}(x,\hat\xi) \abs{\xi}^{m-j},\quad
\text{with } a_{m-j} \in \CI(\RR^n \times S^{n-1}),
$$
and where the tilde denotes an ``asymptotic expansion''---truncating the
expansion at the $\abs{\xi}^{m-N}$ term gives an error that is
$O(\abs{\xi}^{m-N-1}).$\footnote{This does not, of course, mean that
  the series has to converge, or, if it converges, that it has to
  converge to $a:$ we never said $a$ had to be \emph{analytic} in
  $\abs{\xi}^{-1},$ after all.}

If $X$ is a Riemannian manifold, we may define $S^m_\cl(T^*X)$ in the same
fashion, insisting that these conditions hold in local
coordinates.\footnote{One should of course check that the conditions for
  being a classical symbol are in fact coordinate invariant.}

(For later use, we will also want symbols in a more general geometric
setting: if $E$ is a vector bundle we define
$$
S^m_\cl (E)
$$
to consist of smooth functions having an asymptotic expansion, as
above, in the fiber variables.  Often, we will be concerned with
trivial examples like $E=\RR^n_x \times \RR^k_\xi,$ where we will
usually use Greek letters to distinguish the fiber variables.)

The classical symbols are the functions that we will ``quantize'' into
operators using the definition \eqref{quant2}.  As with
fiber-polynomials, the symbol that we quantize to make a given
operator will transform in a complicated manner under change of
variables, but the \emph{top order} part of the symbol,
$a_m(x,\hat\xi) \in \CI(S^*\RR^n),$ will transform invariantly.

\begin{exercise}\label{KN}
We say that a function $a \in
\CI(T^*X)$ is a \emph{Kohn-Nirenberg symbol} of order $m$ on $T^*X$ (and
write $a \in S^m_\KN(T^*X)$) if for all $\alpha,\beta,$
\begin{equation}\label{KNestimate}
\sup \ang{\xi}^{\abs{\beta}-m}\lvert\pa_x^\alpha \pa_\xi^\beta a\rvert= C_{\alpha,\beta}<\infty.
\end{equation}

Check that $S^m_{\cl,c}(T^*\RR^n) \subset S^m_\KN(T^*\RR^n),$ where
the extra subscript $c$ denotes compact support in the base variables. Find examples
of Kohn-Nirenberg symbols compactly supported in $x$ that are not classical
symbols.\footnote{Note that most authors use $S^m$ to denote $S^m_\KN.$}
\end{exercise}
In the interests of full disclosure, it should be pointed out that it
is the Kohn-Nirenberg symbols, rather than the classical ones defined above, that
are conventionally used in the definition of the pseudodifferential calculus.

\

At this point, as discussed in the previous section, we are in a
position to ``define'' the pseudodifferential calculus as sketched at
the end of the previous section: it consists of operators whose
Schwartz kernels near the diagonal look like the quantizations of
classical symbols, and away from the diagonal are smooth.  While our
quantization procedure so far has been restricted to $\RR^n,$ the
theory is in fact cleanest on compact manifolds, so \emph{we shall
  state the properties of the calculus only for $X$ a compact
  $n$-manifold.}\footnote{Some remarks about the noncompact case will
  be found in the explanatory notes that follow.}  Most of the
properties continue to hold on noncompact manifolds provided we are a
little more careful either to control the behavior of the symbols at
infinity, or if we restrict ourselves to ``properly supported''
operators, where the projections to each factor of the support of the
Schwartz kernels give proper maps.  We will therefore not shy away
from pseudodifferential operators on $\RR^n,$ for instance, even
though they are technically a bit distinct; indeed we will only use
them in situations where we could in fact localize, and work on a
large torus instead.

Instead of trying to make a definition of the calculus and read off
its properties, we shall simply try to axiomatize these objects:

\textsc{The space of pseudodifferential operators $\Psi^* (X)$ on a
  compact manifold $X$ enjoys
  the following properties.}  (Note that this enumeration is followed
by further commentary.)  \renewcommand{\theenumi}{\Roman{enumi}}
\begin{enumerate}
\item\label{property:alg} (Algebra property) $\Psi^m(X)$ is a vector
  space for each $m \in \RR.$  If $A \in \Psi^m(X)$ and
  $B \in \Psi^{m'}(X)$ then $AB \in \Psi^{m+m'}(X).$  Also, $A^* \in
  \Psi^m(X).$  Composition of operators is associative and
  distributive. The identity operator is in $\Psi^0(X).$
\item\label{property:smoothing} (Characterization of smoothing
  operators) We let $$\Psi^{-\infty}(X)= \bigcap_m \Psi^m(X);$$ the operators
  in
  $\Psi^{-\infty}(X)$ are exactly those whose Schwartz kernels are
  $\CI$ functions on $X\times X,$ and can also be characterized by the
  property that they map distributions to smooth functions on $X.$
\item\label{property:symb} (Principal symbol homomorphism) There is family of linear ``principal symbol maps'' $\hsigma_m:
  \Psi^m(X) \to \CI(S^* X)$ such that if $A \in \Psi^m(X)$ and $B \in \Psi^{m'}(X),$
$$
\hsigma_{m+m'}(AB) = \hsigma_m(A) \hsigma_{m'}(B)
$$
and
$$
\hsigma_m(A^*) = \overline{\hsigma_m(A)}
$$ We think of the principal symbol either as a function on the unit
cosphere bundle $S^*X$ or as a homogeneous function of degree $m$ on
$T^*X,$ depending on the context, and we let $\sigma_m(A)$ denote the
latter.
\item\label{property:exact} (Symbol exact sequence) There is a short exact sequence
$$
0 \to \Psi^{m-1}(X)  \to \Psi^m(X) \stackrel{\hsigma_m}{\to} \CI(S^* X) \to 0,
$$
hence the principal symbol of order $m$ is $0$ if and only if an operator is of order $m-1.$
\item\label{property:quant}
There is a linear ``quantization map'' $\Op: S^m_\cl(T^* X) \to \Psi^m(X)$ such that
if $a \sim \sum_{j=0}^\infty a_{m-j}(x,\hat\xi) \abs{\xi}^{m-j} \in
S^m_\cl(T^*X)$ then
$$
\hsigma_m(\Op(a)) = a_m(x,\hat\xi).
$$
The map $\Op$ is onto, modulo $\Psi^{-\infty}(X).$
\item\label{property:commutator} (Symbol of commutator)
If $A\in \Psi^m (X)$, $B\in \Psi^{m'} (X)$ then\footnote{That the
  order is $m+m'-1$ follows
  from Properties \eqref{property:symb}, \eqref{property:exact}.}
$[A,B] \in \Psi^{m+m'-1}(X),$ and we have
$$
\sigma_{m+m'} ([A,B]) =  i\{\sigma_m(a),\sigma_{m'}(b)\}.
$$
\item\label{property:boundedness} ($L^2$-boundedness, compactness) If
$A=\Op(a) \in \Psi^0(X)$ then $A:L^2(X)
\to L^2(X)$ is bounded, with a bound depending on finitely many
constants $C_{\alpha,\beta}$ in \eqref{KNestimate}.
 Moreoever, if $A \in \Psi^m(X),$ then
$$
A \in \mathcal{L}(H^s(X), H^{s-m}(X)) \text{ for all }s\in \RR.
$$
Note in particular that $A$ maps $\CI(X)\to\CI(X).$ As a further
consequence, note that operators of negative
order are compact operators on $L^2(X).$
\item\label{property:summation} (Asymptotic summation)
Given $A_j \in \Psi^{m-j}(X),$ with $j \in \NN,$ there exists $A \in
\Psi^m(X)$ such that
$$
A\sim \sum_j A_j,
$$
which means that
$$
A - \sum_{j=0}^N A_j \in \Psi^{m-N-1}(X)
$$
for each $N.$

\item\label{property:microsupport} (Microsupport)
Let $A=\Op(a)+R,$ $R \in \Psi^{-\infty}(X).$  The set of $(x_0,\hat\xi_0) \in S^*X$ such that
$a(x,\xi)=O(\abs{\xi}^{-\infty})$ for $x,\hat\xi$ in some neighborhood
of $(x_0,\hat\xi_0)$ is well-defined, independent of our choice of quantization
map.  Its complement is called the \emph{microsupport} of $A,$ and is
denoted $\WF'A.$  We moreover have
$$
\WF' AB \subseteq \WF'A \cap \WF'B,\quad \WF'(A+B) \subseteq \WF'A \cup \WF'B,
$$
$$
\WF' A^* =\WF'A.
$$
The condition $\WF'A=\emptyset$ is equivalent to $A \in \Psi^{-\infty}(X).$
\end{enumerate}

\noindent\textsc{Commentary:}

\begin{enumerate}
\item[\eqref{property:alg}] If we begin by defining our operators on
  $\RR^n$ by the formula \eqref{quant2}, with $a \in
  S^m_\cl(T^*\RR^n)$, it is quite nontrivial to verify that the
  composition of two such operators is of the same type; likewise for
  adjoints.  Much of the work that we are omitting in developing the
  calculus goes into verifying this property.
\item[\eqref{property:smoothing}] On a non-compact manifold, it is
  only among, say, properly supported operators that elements of
  $\Psi^{-\infty}(X)$ are characterized by mapping distributions to
  smooth functions.
\item[\eqref{property:symb}] Note that there is no sensible,
  \emph{invariant}, way to associate, to an operator $A,$ a ``total
  symbol'' $a$ such that $A=\Op(a).$ As we saw before, a putative
  ``total symbol'' even for differential operators would be
  catastrophically bad under change of variables.  Moreover, as we
  also saw for differential operators, it's a little hard to see what the total symbol of the
  composition is.  This principal symbol map is a compromise
  that turns out to be extremely useful, especially when coupled with
  the asymptotic summation property, in making iterative arguments.
\item[\eqref{property:exact}] A good way to think of this is that
  $\hsigma_m$ is just the obstruction to an operator in $\Psi^m(X)$
  being of order $m-1.$ 
\item[\eqref{property:quant}] The map $\Op$ is far from unique.  Even
  on $\RR^n,$ for
  instance, we can use $\Opl$ as defined by \eqref{quantization} but
  we could also use the ``Weyl'' quantization
$$
(\Op_W(a)\psi)(x) = (2\pi)^{-n} \iint a((x+y)/2,\xi)
e^{i(x-y)\cdot \xi}\psi(y) \, dy\, d\xi
$$
or the ``right'' quantization
$$
(\Op_r(a)\psi)(x) = (2\pi)^{-n}\iint a(y,\xi)
e^{i(x-y)\cdot \xi}\psi(y) \, dy\, d\xi
$$
or any of the obvious interpolating choices.  On a manifold the
choices to be made are even more striking.  One convenient choice that
works globally on a manifold is what might be called ``Riemann-Weyl''
quantization:  Fix a Riemannian metric $g.$  Given $a \in S^m_\cl
(T^*X),$ define the Schwartz kernel of an operator $A$ by
$$
\kappa(A)(x,y) = (2\pi)^{-n} \int \chi(x,y) a(m(x,y),\xi) e^{i
  (\exp_y^{-1}(x),\xi}) \, dg_\xi;
$$
here $\chi$ is a cutoff localizing near the diagonal and in particular,
within the injectivity radius; $m(x,y)$ denotes the midpoint of the
shortest geodesic between $x,y$, $\exp$ denotes the exponential map,
and the round brackets denote the pairing of vectors and covectors.
The ``Weyl'' in the name refers to the evaluation of $a$ at $m(x,y)$
as opposed to $x$ or $y$ (which give rise to corresponding ``left''
and ``right'' quantizations respectively---also acceptable choices).
The ``Riemann'' of course refers to our use of a choice of metric.

We will often only employ a single simple consequence of the existence of
a quantization map: given $a_m \in \CI(S^*X)$ and $m \in \RR$, there exists
$A \in \Psi^m(X)$ with principal symbol $a_m$ and with $\WF' A =\supp a_m.$

\item[\eqref{property:commutator}]
A priori of course $AB-BA \in \Psi^{m+m'}(X);$ however the principal
symbol vanishes, by the commutativity of $\CI (S^*X).$  Hence the need
for a lower-order term, which is subtler, and noncommutative.  That
the Poisson bracket is well-defined independent of coordinates
reflects the fact that $T^*X$ is naturally a symplectic manifold, and
the Poisson bracket is well-defined on such a manifold (see \S\ref{subsec:ham} below).
\begin{exercise}
  Check (by actually performing a change of coordinates) that if $f,g
  \in \CI(T^*X),$ then $\{f,g\}$ is well-defined, independent of
  coordinates.
\end{exercise}

This property is the one which ties classical dynamics to
quantum evolution, as the discussion in \S\ref{subsection:QM} shows.

\

\item[\eqref{property:boundedness}] Remarkably, the mapping property is one
  that can be derived from the other properties of the calculus purely
  algebraically, with the only analytic input being boundedness of
  operators in $\Psi^{-\infty}(X).$ This is the famous H\"ormander
  ``square-root'' argument---see \cite{Hormander:FIO1}, as well as
  Exercise~\ref{exercise:squareroot} below.

  On noncompact manifolds, restricting our attention to properly
  supported operators gives boundedness $L^2 \to L^2_{\loc}.$

  The compactness of negative order operators of course follows from
  boundedness, together with Rellich's lemma, but is worth
  emphasizing; we can regard $\hsigma_0$ as the ``obstruction to
  compactness'' in general.  On noncompact manifolds, this compactness
  property
  fails quite badly, resulting in much interesting mathematics.
\item[\eqref{property:summation}] This follows from our ability to do
  the corresponding ``asymptotic summation'' of total symbols, which
  in turn is precisely ``Borel's Lemma,'' which tells us that any
  sequence of coefficients are the Taylor coefficients of a $\CI$
  function; here we are applying the result to smooth functions on the
  radial compactification of $T^*X,$ and the Taylor series is in the
  variable $\sigma=\abs{\xi}^{-1},$ at $\sigma=0.$
\item[\eqref{property:microsupport}]
Since the total symbol is not well-defined, it is not so obvious that
the microsupport is well-defined; verifying this requires checking
how the total symbol transforms under change of coordinates; likewise,
we may verify that the (highly non-invariant) formula for the total
symbol of the composition respects microsupports to give information
about $\WF' AB.$
\end{enumerate}
\renewcommand{\theenumi}{\arabic{enumi}}

\subsection{Some consequences}
If you believe that there exists a calculus of operators with the
properties enumerated above, well, then you believe quite a lot!  For
instance:
\begin{theorem}
Let $P\in \Psi^m(X)$ with
$\hsigma_m(P)$ nowhere vanishing on $S^*X.$
Then there exists $Q \in \Psi^{-m}(X)$ such that
$$QP-I,PQ-I \in \Psi^{-\infty}(X).$$
\end{theorem}
In other words, $P$ has an approximate inverse (``parametrix'') which
succeeds in inverting it modulo smoothing operators.  

An operator $P$ with nonvanishing principal symbol is said to be
\emph{elliptic}.  Note that this theorem gives us, via the Sobolev
estimates of \eqref{property:boundedness}, the usual elliptic regularity
estimates.  In particular, we can deduce
$$
P u \in \CI(X) \Longrightarrow u \in \CI(X).
$$
\begin{exercise}
Prove this.
\end{exercise}

\begin{proof}
Let $q_{-m}=(1/\hsigma_m(P));$ let $Q_{-m}\in \Psi^{-m}(X)$ have
principal symbol $q_{-m}.$  (Such an operator exists by the exactness
of the short exact symbol sequence.)  Then by \eqref{property:symb},
$$
\hsigma_0(PQ_{-m})=1,
$$
hence by \eqref{property:exact},\footnote{The identity operator has
  principal symbol equal to $1,$ since the symbol map is a homomorphism.}
$$
PQ_{-m}-I=R_{-1} \in \Psi^{-1}(X).
$$
Now we try to correct for this ``error term:'' pick $Q_{-m-1} \in
\Psi^{-m-1}(X)$ with
$$
\hsigma_{-m-1}(Q_{-m-1})=-\hsigma_{-1}(R_{-1})/\hsigma_m(P).
$$
Then we have
$$
P(Q_{-m}+Q_{-m-1})-I= R_{-2} \in \Psi^{-2}(X).
$$
Continuing iteratively, we get a series of $Q_j\in \Psi^{-m-j}$ such that
$$
P(Q_{-m}+\dots+Q_{-m-N})-I \in \Psi^{-N-1}(X).
$$
Using \eqref{property:summation}, pick
$$
Q \sim \sum_{j=-m}^{-\infty} Q_j.
$$
This gives the desired parametrix:
\begin{exercise}\
\begin{enumerate}\item
Check that $PQ-I \in \Psi^{-\infty}(X).$
\item
Check that $QP-I \in \Psi^{-\infty}(X).$ 
(\textsc{Hint:} First check that a left parametrix exists; you may
find it helpful to take adjoints.  Then check that the left parametrix
must agree with the right parametrix.)
\end{enumerate}
\end{exercise}
\end{proof}

\begin{exercise}
Show that an elliptic pseudodifferential operator on a compact
manifold is Fredholm.  (\textsc{Hint:} You can show, for instance,
that the kernel is finite dimensional by observing that the existence
of a parametrix implies that the identity operator on the kernel is
equal to a smoothing operator, which is compact.)
\end{exercise}

\begin{Exercise}\label{exercise:inverses}\
\begin{enumerate}\item
  Let $X$ be a compact manifold.  Show that if $P \in \Psi^m(X)$ is elliptic, and has an actual
  inverse operator $P^{-1}$ as a map from smooth functions to smooth
  functions, then $P^{-1} \in \Psi^{-m}(X).$ (\textsc{Hint:} Show that
  the parametrix differs from the inverse by an operator in
  $\Psi^{-\infty}(X)$---remember that an operator is in
  $\Psi^{-\infty}(X)$ if and only if it maps distributions to smooth
  functions.)
\item More generally, show that if $P\in \Psi^m(X)$ is elliptic, then
  there exists a generalized inverse of $P,$ inverting $P$ on its
  range, mapping to the orthocomplement of the kernel, and
  annihilating the orthocomplement of the range, that lies in
  $\Psi^{-m}(X).$
\end{enumerate}
\end{Exercise}

\begin{Exercise} Let $X$ be compact, and $P$ an elliptic operator on
  $X,$ as above, with positive order.
Using the spectral theorem for compact, self-adjoint
    operators, show that if $P^*=P,$ then there is an orthornormal basis for
  $L^2(X)$ of eigenfunctions of $P,$ with eigenvalues tending to
  $+\infty.$  Show that the eigenfunctions are in $\CI(X).$
  (\textsc{Hint:} show that there exists a basis of such eigenfunctions for
  the generalized inverse $Q$ and then see what you can say about $P.$)
\end{Exercise}

\begin{exercise}\label{exercise:Lapspectrum} Let $X$ be compact.
\begin{enumerate}
\item Show that the principal symbol of $\Lap,$ the Laplace-Beltrami
  operator on a compact Riemannian manifold, is
  just $$\abs{\xi}_g^2\equiv \sum g^{ij}(x) \xi_i \xi_j,$$
  the metric induced on the cotangent bundle.
\item Using the previous exercise, conclude that there exists an
  orthonormal basis for $L^2(X)$ of eigenfunctions of $\Lap,$ with
  eigenvalues tending toward $+\infty.$
\end{enumerate}
\end{exercise}

\begin{exercise}\label{exercise:squareroot}
Work out the H\"ormander ``square root trick'' on a compact manifold
$X$ as follows.
\begin{enumerate}
\item Show that if $P \in \Psi^0(X)$ is
  self-adjoint, with positive principal symbol, then $P$ has
  an approximate square root, i.e.\ there exists $Q \in \Psi^{0}(X)$
  such that $Q^*=Q$ and $P-Q^2\in \Psi^{-\infty}(X).$ (\textsc{Hint:}
  Use an iterative construction, as in the proof of existence of
  elliptic parametrices.)
\item Show that operators in $\Psi^{-\infty}(X)$ are $L^2$-bounded.
\item Show that an operator $A\in \Psi^{0}(X)$ is $L^2$-bounded.
  (\textsc{Hint:} Take an approximate square root of $\lambda I-A^*A$
  for $\lambda \gg 0.$)
\end{enumerate}
\end{exercise}

As usual, let $\Lap$ denote the Laplacian on a compact manifold.  By
Exercise~\ref{exercise:squareroot}, there exists an operator $A\in
\Psi^1(X)$ such that $A^2=\Lap+R,$ with $R \in \Psi^{-\infty}(X).$ By
abstract methods of spectral theory, we know that $\sqrt{\Lap}$ exists
as an unbounded operator on $L^2(X).$ (This is a very simple use of the
functional calculus: merely take $\sqrt{\Lap}$ to act by
multiplication by $\lambda_j$ on each $\phi_j$, where
$(\phi_j,\lambda_j^2)$ are the eigenfunctions and eigenvalues of the
Laplacian, from Exercise~\ref{exercise:Lapspectrum}.)  In fact, we
can improve this argument to obtain:
\begin{proposition}\label{proposition:sqrtlap}
$$
\sqrt{\Lap} \in \Psi^1(X).
$$
\end{proposition}
Indeed, it follows from a theorem of Seeley that all complex powers of a self-adjoint,
elliptic pseudodifferential operator\footnote{Seeley's theorem is
  better yet: self-adjointness
is unnecessary.} on a compact manifold are
pseudodifferential operators.

All proofs of the proposition seem to introduce an auxiliary parameter in
some way, and the following (taken directly from \cite[Chapter XII,
\S1]{Taylor:pseudors}) seems one of the simplest.  An alternative approach, using the theory of elliptic boundary
problems, is sketched in \cite[pp.32-33, Exercises 4--6]{Taylor:PDE2}.
\begin{proof}
Let $A$ be the self-adjoint parametrix constructed in
Exercise~\ref{exercise:squareroot}, so that
$$
A^2-\Lap=R\in \Psi^{-\infty}(X).
$$  By taking a parametrix for the
square root of $A,$ in turn, we obtain
$$
A=B^2+R'
$$
with $B \in \Psi^{1/2}(X)$ and $R' \in \Psi^{-\infty},$ both
self-adjoint; then pairing with a test function $\phi$ shows that
$$
\ang{A\phi,\phi} \geq \ang{R'\phi,\phi} \geq -C \norm{u}^2
$$
for some $C \in \RR.$  Thus, $A$ can only have finitely many
nonpositive eigenvalues (since it has a compact generalized inverse)
hence its eigenvalues can accumulate only at $+\infty$).  So we may
alter $A$ by the smoothing operator projecting off of these
eigenspaces, and maintain
$$
A^2-\Lap=R\in\Psi^{-\infty}(X)
$$
(with a different $R$, of course) while now ensuring that $A$ is positive.

Now we may write, using the spectral theorem,
$$
(\Lap')^{-1/2} =  \frac{1}{2 \pi i} \int_\Gamma z^{-1/2} ((\Lap')-z)^{-1} \, dz
$$
where $\Gamma$ is a contour encircling the positive real axis
counterclockwise, and given by $\Im z=\Re z$ for $z$ sufficiently
large, and $\Lap'$ is given by $\Lap$ minus the projection onto
constants (hence has no zero eigenvalue).  (The integral converges
in norm, as self-adjointness of $\Lap'$ yields
$$
\norm{((\Lap')-z)^{-1}}_{L^2\to L^2} \lesssim \abs{\Im z}^{-1}.)
$$
Likewise, since $A^2=\Lap'+R$ (with $R$ yet another smoothing operator) we may write
$$
A^{-1} =  \frac{1}{2 \pi i} \int_\Gamma z^{-1/2} ((\Lap')+R-z)^{-1} \, dz
$$
Hence
\begin{align*}
(\Lap')^{-1/2}-A^{-1} &= \frac{1}{2 \pi i} \int_\Gamma z^{-1/2} \big[
((\Lap')-z)^{-1}-((\Lap')+R-z)^{-1}\big] \, dz\\ &=
\frac{1}{2 \pi i} \int_\Gamma z^{-1/2} 
((\Lap')-z)^{-1}R((\Lap')+R-z)^{-1}\, dz.
\end{align*}
Now the integrand, $z^{-1/2} 
((\Lap')-z)^{-1}R((\Lap')+R-z)^{-1},$ is for each $z$ a smoothing
operator, and decays fast enough that when applied to any $u \in \mathcal{D}'(X),$
the integral converges to an element of $\CI(X)$ (in particular, the
integral converges in $\mathcal{C}^0(X),$ even after application of
$\Lap^k$ on the left, for any $k$).   Hence
$$
(\Lap')^{-1/2}-A^{-1}=E \in \Psi^{-\infty}(X);
$$
thus we also obtain
$$
(\Lap')^{1/2}=(A^{-1}+E)^{-1} \in \Psi^1(X);
$$
as $(\Lap')^{1/2}$ differs from $\Lap^{1/2}$ by the smoothing operator
of projection onto constants, this shows that
$$
\Lap^{1/2} \in \Psi^1(X).\qed
$$\noqed
\end{proof}

\section{Wavefront set}

If $P \in \Psi^m(X)$ and $(x_0,\xi_0) \in S^*X,$ we say $P$ is elliptic at
$(x_0,\xi_0)$ if $\hsigma_m(P)(x_0,\xi_0)\neq 0.$  Of course if $P$ is
elliptic at each point in $S^*X,$ it is elliptic in the sense defined
above.  We let
$$
\liptic(P)=\{(x,\xi): P \text{ is elliptic at } (x,\xi)\},
$$
and let
$$
\Sigma_P=S^*X\backslash \liptic (X);
$$
$\Sigma_P$ is known as the \emph{characteristic set} of $P.$

\begin{exercise}\
\begin{enumerate}
\item
Show that $\liptic P \subseteq \WF'P.$
\item If $P$ is a differential operator of
  order $m$ of the form $\sum a_\alpha(x) D^\alpha$ then show that $\WF' P=\pi^*(\bigcup
  \supp a_\alpha),$ while $\liptic P$ may be smaller.
\end{enumerate}
\end{exercise}

The following ``partition of unity'' result, and variants on it, will
frequently be useful in discussing microsupports.  It yields an
operator that is microlocally the identity on a compact set, and
microsupported close to it.
\begin{lemma}\label{lemma:pof1}
Given $K\subset U \subset S^*X$ with $K$ compact, $U$ open, there
exists a self-adjoint operator $B \in \Psi^0(X)$ with
$$
\WF' (\Id-B) \cap K=\emptyset,\ \WF'B \subset U.
$$
\end{lemma}
\begin{exercise}\label{exercise:firstpof1}
Prove the lemma.  (\textsc{Hint:} You might wish to try constructing $B$ in the form
$$
\Op(\psi \sigma_{\tot}(\Id))
$$
where $\sigma_{\tot}(\Id)$ is the total symbol of the identity (which is simply $1$
for all the usual quantizations on $\RR^n$) and $\psi$ is a cutoff function
equal to $1$ on $K$ and supported in $U.$  Then make $B$ self-adjoint.)
\end{exercise}

\begin{theorem}
If $P\in \Psi^m(X)$ is elliptic at $(x_0,\xi_0),$ there exists a \emph{microlocal
  elliptic parametrix} $Q \in \Psi^{-m}(X)$ such that
$$
(x_0,\xi_0) \notin \WF'(PQ-I) \cup \WF'(QP-I).
$$
\end{theorem}
In other words, you should think of $Q$ as inverting $P$ \emph{microlocally
  near $(x_0,\xi_0)$.}
\begin{exercise}
Prove the theorem.  (\textsc{Hint:} If $B$ is a microlocal partition
of unity as in Lemma~\ref{lemma:pof1}, microsupported sufficiently close to
$(x_0,\xi_0)$ and microlocally the identity in a smaller neighborhood,
then show
$$
W=B P + \lambda\Op(\ang{\xi}^m) (\Id-B)
$$
is globally elliptic provided $\lambda \in \CC$ is chosen appropriately.
Now, using the existence of an elliptic parametrix for $W,$ prove the theorem.)
\end{exercise}

Let $u$ be a distribution on a manifold $X.$  We define the \emph{wavefront
  set} of $u$ as follows.
\begin{definition}
The wavefront set of $u,$
$$
\WF u \subseteq S^*X,
$$
is given by $$(x_0,\xi_0) \notin \WF u$$ if and only if there exists $P \in
\Psi^0(X),$ elliptic at $(x_0,\xi_0),$ such that $$Pu \in \CI.$$
\end{definition}

\begin{exercise}
Show that the choice of $\Psi^0(X)$ in this definition is immaterial, and
that we get the same definition of $\WF u$ if we require $P \in \Psi^m(X)$ instead.
\end{exercise}

Note that the wavefront set is, from its definition, a closed set.
Instead of viewing $\WF u$ as a subset of $S^*X,$ we also, on
occasion, think of $\WF u$ as a \emph{conic} subset of $T^*
X\backslash o,$ with $o$ denoting the zero section; a conic set in a
vector bundle is just one that is invariant under the $\RR^+$ action
on the fibers.

An important variant is as follows: we say that 
$$
(x_0, \xi_0) \notin \WF^m u  
$$
if and only if there exists $P \in \Psi^m(X),$ elliptic at
$(x_0,\xi_0)$ such that
$$
P u \in L^2(X).
$$

\begin{proposition}
$\WF u =\emptyset$ if and only if $u \in \CI(X);$ $\WF^m u =\emptyset$
if and only if $u \in H^m_\loc(X).$
\end{proposition}

The wavefront set serves the purpose of measuring not just where, but
also in what (co-)direction, a distribution fails to be in $\CI(X)$
(or $H^m$ in the case of the indexed version).  It is instructive to
think about testing for such regularity, at least on $\RR^n,$ by
localizing and Fourier transforming.  Given $(x_0, \hat\xi_0) \in
S^*\RR^n,$ let $\phi\in \CcI(\RR^n)$ be
nonzero at $x_0;$ let $\gamma \in \CI(\RR^n)$ be given by
$$
\gamma(\xi) = \psi\big(\big\lvert\frac{\xi}{\abs{\xi}}-\hat\xi_0\big\rvert\big)\chi(\abs{\xi})
$$ where $\psi$ is a cutoff function supported near $x=0$ and $\chi(t)
\in \CI(\RR)$ is equal to $0$ for $t<1$ and $1$ for $t>2.$ Think of
$\gamma$ as a cutoff in a cone of directions near $\xi_0,$ but
modified to be smooth at the origin.  (We will use such a construction
frequently, and refer in future to a function such as $\gamma$ as a
``conic cutoff near direction $\hat\xi_0.$''.)

Now note that
$\phi(x)\gamma(\xi)$ is a symbol of order zero, and
\begin{equation}\label{explicitoperator}
\Opl(\phi(x) \gamma(\xi))^*=\Opr(\phi(x) \gamma(\xi)) u = (2\pi)^{-n}\F^{-1}\gamma(\xi)\F(\phi u).
\end{equation}
By definition, if $\Opl(\phi(x)\gamma(\xi))^*u\in \CI,$ then $(x_0,\xi_0)
\notin \WF u.$  Note that since $\phi u$ has compact support, we
automatically have $\F(\phi u) \in \CI,$ hence $\F^{-1} \gamma \F
(\phi u)$ is rapidly decreasing.
Since $\F$ is an isomorphism from $\schwartz(\RR^n)$
to itself, we see that it in fact suffices to have
$$
\gamma \F(\phi u) \in \schwartz(\RR^n)
$$
to be able to conclude that $(x_0,\xi_0) \notin \WF u.$ 
Conversely, one can check that the class of operators of the form
$$
\Opl(\phi(x) \gamma(\xi))^* 
$$
is rich enough that this in fact amounts to a \emph{characterization}
of wavefront set:

\begin{proposition}\label{proposition:wfcharacterization}
We have $(x_0, \xi_0) \notin \WF u$ if and only if there
exist $\phi,$ $\gamma$ as above with
$$
\gamma \F(\phi u) \in \schwartz(\RR^n).
$$
\end{proposition}
\begin{exercise}
Prove the Proposition.  (\textsc{Hint:} If $A\in \Psi^0(\RR^n)$ is elliptic at $(x_0,\xi_0)$ and
$Au \in \CI(\RR^n),$ construct $B=\Opl(\phi(x) \gamma(\xi))^*$ as above so
that $\WF' B$ is contained in the set where $A$ is elliptic.  Hence
there is a microlocal parametrix $Q$ such that $B(QA-I) \in\Psi^{-\infty}(X).$)
\end{exercise}

Note that if $u$ is smooth near $x_0,$ then we have $\phi u \in
\CcI(\RR^n)$ for appropriately chosen $\phi,$ hence there is no
wavefront set in the fiber over $x_0.$

If, by contrast, $u$ is not smooth in any neighborhood of $x_0,$ then
we of course do not have $\F(\phi u) \in \schwartz,$ although it is in
$\CI;$ the wavefront set includes the directions in which it fails to
be rapidly decaying.

Thus, we can easily see that in fact the projection to the base
variables of $\WF u$ is the \emph{singular support} of $u,$ i.e.\ the
points which have no neighborhood in which the distribution $u$ is a
$\CI$ function.

\begin{exercise}
Let $\Omega \subset \RR^n$ be a domain with smooth boundary.  Show
that $\WF 1_\Omega=SN^* (\pa\Omega),$ the spherical normal bundle of
the boundary.  (\textsc{Hint:} You may want to use the fact that the
definition of $\WF u$ is coordinate-invariant.)
\end{exercise}

We have a result constraining the wavefront set of a solution to a PDE
or, more generally, a pseudodifferential equation, directly following
from the definition:
\begin{theorem}\label{theorem:characteristic}
If $Pu \in \CI(X),$ then $\WF u \subseteq \Sigma_P.$
\end{theorem}
\begin{proof}
By definition, $Pu \in \CI(X)$ means that $\WF u \cap \liptic P=\emptyset.$
\end{proof}

\begin{theorem}\label{theorem:microlocality}
If $P\in \Psi^*(X)$, $\WF Pu \subseteq \WF u \cap \WF' P.$
\end{theorem}
\begin{exercise}
Prove this, using microlocal elliptic parametrices for the inclusion
in $\WF u.$
\end{exercise}
The property of pseudodifferential operators that $\WF Pu \subseteq \WF
u$ is called ``microlocality:'' the operators are not ``local,'' in that they
do move \emph{supports} of distributions around, but they don't move
\emph{singularities}, even in the refined sense of wavefront set.

We shall also need related results on Sobolev based wavefront sets in
what follows:
\begin{proposition}\label{proposition:sobmicro}
If $P \in \Psi^m(X),$ $\WF^{k-m}Pu \subseteq \WF^k u \cap \WF'P$ for all
$k \in \RR.$
\end{proposition}
\begin{corollary}\label{corollary:microbounded}
Let $P \in \Psi^m(X).$  If
$$
\WF' P \cap \WF^m u =\emptyset
$$
then
$$
P u \in L^2(X).
$$
\end{corollary}
\begin{exercise}
Prove the proposition (again using a microlocal elliptic parametrix)
and the corollary.
\end{exercise}

We will have occasion to use the following relationship between
ordinary and Sobolev-based wavefront sets:
\begin{proposition}\label{proposition:sobwf}
$$\WF u = \overline{\bigcup_k \WF^k u}.$$
\end{proposition}
\begin{exercise}
Prove the proposition.
\end{exercise}

\begin{exercise}
Let $\Box$ denote the wave operator,
$$
\Box u = D_t^2 u -\Lap u
$$
on $M=\RR\times X$ with $X$ a Riemannian manifold.  Show if $\Box u=0$ then the
wavefront set of $u$ is a subset of the ``wave cone''
$\{\tau^2=\abs{\xi}^2_g\}$ where $\tau$ is the dual variable to $t$
and $\xi$ to $x$ in $T^* (M).$
\end{exercise}
\begin{exercise}\label{exercise:WFrestriction}\
\begin{enumerate}
\item
Let $k<n,$ and let $\iota:\RR^k \to \RR^n$
denote the inclusion map.

Show that there is a continuous \emph{restriction map} on
compactly supported distributions with no wavefront set conormal to $\RR^k:$
$$
\iota^*: \{u \in \mathcal{E}'(\RR^n): \WF u \cap SN^*(\RR^k)=\emptyset\}
\to \mathcal{E}'(\RR^k).
$$
\textsc{Hint:} Show that it suffices to consider $u$ supported in a
small neighborhood of a single point in $\RR^k.$  Then take the
Fourier transform of $u$ and try to integrate in the conormal
variables to obtain the Fourier transform of the restriction.
\item
Show that, with the notation of the previous part,
$$
\WF \iota^* u \subseteq \iota_*(\WF u)
$$
where $\iota_*: T^*_{\RR^k} \RR^n \to T^*\RR^k$ is the naturally defined
projection map.
\item
Show that both the previous parts make sense, and are
valid, for restriction to an embedded submanifold $Y$ of a manifold $X.$
\item
Show that if $u$ is a distribution on $\RR^k_x$ and $v$ is a distribution on
$\RR^l_y$ then $w=u(x) v(y)$ is a distribution on $\RR^{k+l}$ and 
$$
\WF w \subseteq \big[(\supp u, 0) \times \WF v\big] \cup \big[\WF u \times (\supp v,0)\big] \cup
\WF u \times \WF v.
$$
(\textsc{Hint:} Localize and Fourier transform, as in \eqref{explicitoperator}.)
\end{enumerate}
\end{exercise}

You might wonder: given $P,$ can the wavefront set of a solution to
$Pu=0$ be any closed subset of $\Sigma$?  The answer is no, there are,
in general, further constraints.  To talk about them effectively, we
should digress briefly back into geometry.

\subsection{Hamilton flows}\label{subsec:ham}

We now amplify the discussion \S\ref{subsection:QM} of Hamiltonian
mechanics and symplectic geometry, generalizing it to a broader
geometric context.

Let $N$ be a symplectic manifold, that is to say, one endowed with a
closed, nondegenerate\footnote{Nondegeneracy of $\omega$ means that
  contraction with $\omega$ is an isomorphism from $T_p N$ to $T_p^*
  N$ at each point.} two-form.  (Our prime example is $N=T^*X,$ endowed
with the form $\sum d\xi_j\wedge dx^j;$ by Darboux's theorem, every
symplectic manifold in fact locally looks like this.)

Given a real-valued function $a \in \CI(N),$ we can make a \emph{Hamilton vector field} from
$a$ as follows: by nondegeneracy, there is a unique vector field $\hamvf_a$
such that $\iota_{\hamvf_a}\omega \equiv \omega(\cdot, \hamvf_a) = da.$
\begin{exercise}
Check that in local coordinates in $T^* X,$
$$
\hamvf_a=\sum_{j=1}^n \frac{\pa a}{\pa \xi_j}\pa_{x^j}-\frac{\pa a}{\pa x^j}\pa_{\xi_j}.
$$
\end{exercise}
Thus, for any smooth function $b,$ we may define the Poisson bracket
$$
\{a,b\} = \hamvf_a (b)
$$
\begin{exercise}
Check that the Poisson bracket is antisymmetric.
\end{exercise}
It is easy to verify that the flow along $\hamvf_a$ preserves both the
symplectic form and the function $a:$ we have from Cartan's formula
(and since $\omega$ is closed):
$$
\mathcal{L}_{\hamvf_a} (\omega) = d \iota_{\hamvf_a} \omega = d(da)=0;
$$
also,
$$
\hamvf_a(a) = da(\hamvf_a)=\omega(\hamvf_a,\hamvf_a)=0.
$$

The integral curves of the vector field $\hamvf_a$ are called the
\emph{bicharacteristics} of $a$ and those lying inside
$\Sigma_a=\{a=0\}$ are called \emph{null bicharacteristics}.

\begin{Exercise}\label{Exercise:geodesicflow}\
\begin{enumerate}
\item
Show that the bicharacteristics of $\abs{\xi}_g=(\sigma_2(\Lap))^{1/2}$ project to
$X$ to be geodesics.  The flow along the Hamilton vector field of
$\abs{\xi}_g$ is known as \emph{geodesic flow}.
\item
Show that the null bicharacteristics of $\sigma_2(\Box)$ are lifts to
$T^*(\RR\times X)$ of geodesics of $X,$ traversed both forward and
backward at unit speed.
\end{enumerate}
\end{Exercise}

Recall that the setting of symplectic manifolds is exactly that of
Hamiltonian mechanics: given such a manifold, we can regard it as the
phase space for a particle; specifying a function (the ``energy'' or
``Hamiltonian'') gives a vector field, and the flow along this vector
field is supposed to describe the time-evolution of our particle in
the phase space.

\begin{exercise}
Check that the phase space evolution of the harmonic oscillator
Hamiltonian, $(1/2)(\xi^2+x^2)$ on $T^*\RR,$ agrees with what you
learned in physics class long ago.
\end{exercise}

\subsection{Propagation of singularities}
\begin{theorem}[H\"ormander]\label{theorem:hormander}
Let $Pu \in \CI(X),$ with $P\in \Psi^m(X)$ an operator with real principal symbol.
Then $\WF u$ is a union of maximally extended null bicharacteristics
of $\hsigma_m(P)$ in $S^*X.$
\end{theorem}
We should slightly clarify the usage here: to make sense of these null
bicharacteristics, we should actually take the Hamilton vector field
of the homogeneous version of the symbol, $\sigma_m(P);$ this is a
homogeneous vector field, and its integral curves thus have
well-defined projections onto $S^*X.$ If the Hamilton vector field
should be ``radial'' at some point $q \in T^*X,$ i.e.\ coincide with a
multiple of the vector field $\xi \cdot\pa_\xi$ there, then the
projection of the integral curve through $q$ is just a single point in
$S^*X,$ and the theorem gives no further information about wavefront
set at that point.

For $P=\Box,$ the theorem says that the wavefront set lies in the
``light cone,'' and propagates forward and backward at unit speed
along geodesics.  If we take the fundamental solution to the wave
equation\footnote{This is the spectral-theoretic way of writing the
  solution with initial value $0$ and initial time-derivative
  $\delta_p$.}$u= \sin (t\sqrt{\Lap}/\sqrt{\Lap})\delta_p,$ it is not
hard to compute that in fact for small, nonzero time,\footnote{Well, I
  am cheating a bit here, as we haven't stated any results
  allowing us to relate the wavefront set of Cauchy data for the
  wavefront set of the solution to the equation.  To understand how to
  do this, you should read \cite{Melrose:notes}.}
$$
\WF u \subseteq N^* \{d(\cdot, p)=\abs{t}\}\equiv \lag;
$$ This is a generalization of Huygens's Principle, which tells us
that in $\RR\times \RR^{n},$ for $n$ odd, the support of the fundamental solution
is on this expanding sphere (but which is a highly unstable property).
Note that $\lag$ is in fact the bicharacteristic flowout of all
covectors in $\Sigma$ projecting to $N^*(\{p\})$ at $t=0,$ and under this
interpretation, $\lag\subset T^*(\RR\times X)$ makes sense for
\emph{all} times, not just for short time, regardless of the metric
geometry.  We shall return to and amplify this point of view in \S\ref{section:wavetraceredux}.

\begin{exercise}\
\begin{enumerate}\item
Suppose that $\Box u=0$ on $\RR\times \RR^n$ and $u(t,x) \in \CI$ for $(t,x) \in
(-\ep,\ep)\times B(0,1)$ for some $\ep>0.$ Show, using Theorem~\ref{theorem:hormander}, that $u \in \CI$ on $\{\abs{x}<1-\abs{t}\}.$ Can you show this
more directly using the energy methods described in
\S\ref{section:wave}?
\item
Suppose that $\Box u=0$ on $\RR\times \RR^n$ and $u(t,x) \in \CI$ for $(t,x) \in
(-\ep,\ep)\times (B(0,1)\backslash B(0,1/2))$ for some $\ep>0.$ 
Show, using the theorem, that $u \in \CI$ in
$\{\abs{x}<1-\abs{t}\}\cap\{\abs{t} \in (3/4,1)\}$
\end{enumerate}
\end{exercise}

\begin{proof}\footnote{This proof is very close to those employed by
    Melrose in \cite{Melrose:notes} and \cite{Melrose:spectral}.}  
Note that we already know that $\WF u \subseteq \Sigma_P$ by
Theorem~\ref{theorem:characteristic}, hence what remains to be proved is
the flow-invariance.

Let $q \in \Sigma_P \subset S^*X.$  By homogeneity of $\sigma_m(P),$ we can write the Hamilton vector field in
  $T^*X$ in a neighborhood of $q$ as
  \begin{equation}\label{homogeneity}\hamvf_p=\abs{\xi}^{m-1} (\mathsf{V}+h
  \mathsf{R}),\end{equation} where $\mathsf{R}$ denotes the radial vector field $\xi
  \cdot \pa_\xi,$ $h$ is a function on $S^*X,$ and $\mathsf{V}$ is the pullback
  under quotient of a vector field on $S^*X$ itself, i.e.\ $\mathsf{V}$ is
  homogeneous of degree zero with no radial component, hence of the form
  $\sum_j f_j(x,\hat\xi) \pa_{\hat\xi_j}+g_j (x,\hat\xi)\pa_{x^j}.$ Note
  that if $a$ is homogeneous of degree $l$ then
  \begin{equation}\label{radial}\mathsf{R} a=la.\end{equation}
  (Exercise: Verify these consequences of homogeneity.)

By the comments above, we may take $\mathsf{V} \neq 0$ near $q;$ otherwise the
theorem is void.  Thus, without loss of generality, we may employ a
coordinate system $\alpha_1,\dots,\alpha_{2n-1}$ for $S^*X$ in which
\begin{equation}\label{nicecoords}
\mathsf{V}=\pa_{\alpha_1},
\end{equation}
hence using $\alpha,\abs{\xi}$ as coordinates in $T^*X,$
$$
\hamvf_p=\pa_{\alpha_1} + h \mathsf{R};
$$
we may shift coordinates so that $\alpha(q)=0.$  We split the $\alpha$ variables
into $\alpha_1$ and $\alpha'=(\alpha_2,\dots,\alpha_{n-1}).$  

 Since $\WF u$ is closed,
  it suffices to prove the following: if $q \notin \WF u$ then
  $\Phi_t(q)\notin \WF u$ for $t \in [-1,1],$ where $\Phi_t$
  denotes the flow generated by $\mathsf{V}.$\footnote{Of course,
we are assuming here that the interval $[-1,1]$ remains in our coordinate
neighborhood; rescale the coordinates if necessary to make this so.}
  (This will show that the intersection of $\WF u$ with the
  bicharacteristic through $q$ is both open and closed, hence is the
  whole thing.)

We can make separate arguments for $t \in [0,1]$ and $t \in
[-1,0],$ and will do so (in fact, we will leave one case to the reader).

For simplicity, let us take $Pu=0;$ we leave the case of an
inhomogeneous equation for the reader (it introduces extra terms, but
no serious changes will in fact be necessary in the proof).

Since $\WF u$ is closed, our assumption that $q \notin \WF u$ tells us
that there is in fact a $2\delta$-neighborhood of $0$ in the $\alpha$
coordinates that is disjoint from $\WF u;$ we are trying to extend
this regularity along the rest of the set $(\alpha_1,\alpha') \in
[0,1]\times 0.$  We proceed as follows: let
\begin{equation}\label{sup}
s_0=\sup\{s: \WF^s u \cap \{(\alpha_1,\alpha') \in [0,1]\times
B(0,\delta)\}=\emptyset\}.\end{equation} Pick any $s<s_0.$ We will show that in fact
\begin{equation}\label{tobeshown}
\WF^{s+1/2} u \cap \{(\alpha_1,\alpha') \in
[0,1]\times B(0,\delta)\}=\emptyset,
\end{equation} thus establishing that $s_0=\infty,$ which is
the desired result (by Proposition~\ref{proposition:sobwf}).
One can regard this strategy as iteratively obtaining more
and more regularity for $u$ along the bicharacteristic (i.e.\ the idea
is that we start by knowing \emph{some} possibly very bad regularity,
and we step by step conclude that we can improve upon this regularity,
half a derivative at a time).  More colloquially, the idea is that the
``energy,'' as measured by testing the distribution $u$ by
pseudodifferential operators, should be comparable at different points
along the bicharacteristic curve.

\

Now we prove the estimates that yield \eqref{tobeshown} via commutator methods.
Let $\phi(s)$ be a cutoff function with
\begin{equation}\label{cutoff}
\begin{aligned}
\phi(t) &> 0 \text{ on } (-1,1),\\
\supp \phi &=[-1,1]
\end{aligned}.
\end{equation}
Let $\phi_{\delta}(s)=\phi(\delta^{-1}s);$ arrange that $\sqrt{\phi} \in
\CI.$ Let $\chi$ be a cutoff function equal to $1$ on $(0, 1)$ and
with $\chi'=\psi_1-\psi_2,$ with $\psi_1$ supported on $(-\delta,
\delta)$ and $\psi_2$ on $(1-\delta,1+\delta);$ we will further assume
that $\sqrt{\chi}, \sqrt{\psi_i} \in \CI.$

\begin{exercise}
Verify that cutoffs with these properties exist.
\end{exercise}

In our coordinate system for $S^*X,$ let $$\hat
a=\phi_\delta(\abs{\alpha'}) \chi(\alpha_1)e^{-\lambda \alpha_1} \in
\CI(S^*X),$$ with $\lambda \gg 0$ to be chosen presently.  Passing to
the corresponding function on $a \in \CI(T^*X)$ that is homogeneous of
degree $2s-m+2,$ we have
\begin{multline}
\hamvf_p (a) =
\abs{\xi}^{2s+1}\big(-\lambda \phi_\delta(\abs{\alpha'})\chi(\alpha_1)e^{-\lambda \alpha_1} \\+
  \phi_\delta(\abs{\alpha'})(\psi_1-\psi_2)e^{-\lambda \alpha_1} +
  h(\alpha)(2s-m+2) a \big)
\end{multline} with $h$ given by
\eqref{homogeneity}.  Since a $2\delta$ coordinate neighborhood of the
origin was assumed absent from $\WF u,$ we have in particular ensured that
$\supp \phi_\delta(\abs{\alpha'})\psi_1(\alpha_1)$ is contained in
$(\WF u)^c.$  We also have $\supp \hat a \subset (\WF^s u)^c$ by
\eqref{sup}, since $s<s_0.$
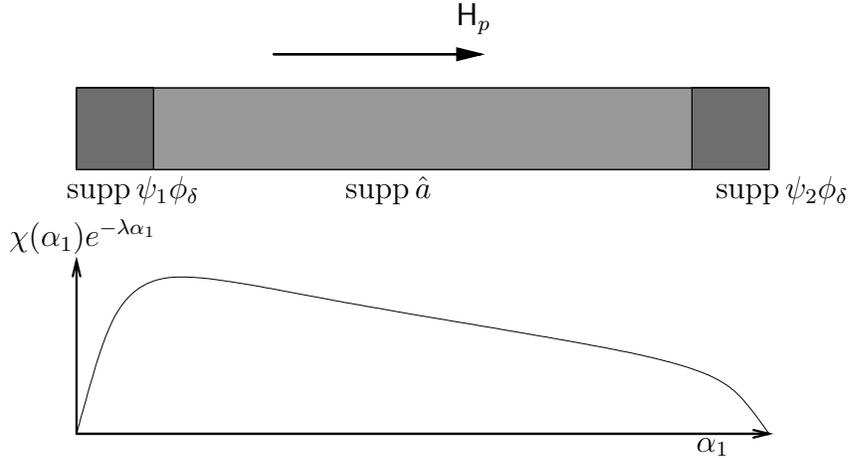
\begin{figure}[h]
\setlength{\unitlength}{0.00033333in}
\begingroup\makeatletter\ifx\SetFigFont\undefined%
\gdef\SetFigFont#1#2#3#4#5{%
  \reset@font\fontsize{#1}{#2pt}%
  \fontfamily{#3}\fontseries{#4}\fontshape{#5}%
  \selectfont}%
\fi\endgroup%
{\renewcommand{\dashlinestretch}{30}
\begin{picture}(11872,6796)(0,-10)
\thicklines
\texture{55888888 88555555 5522a222 a2555555 55888888 88555555 552a2a2a 2a555555 
	55888888 88555555 55a222a2 22555555 55888888 88555555 552a2a2a 2a555555 
	55888888 88555555 5522a222 a2555555 55888888 88555555 552a2a2a 2a555555 
	55888888 88555555 55a222a2 22555555 55888888 88555555 552a2a2a 2a555555 }
\path(1050,274)(11850,274)
\path(1050,274)(11850,274)
\path(11610.000,214.000)(11850.000,274.000)(11610.000,334.000)
\path(1050,274)(1050,2974)
\path(1110.000,2734.000)(1050.000,2974.000)(990.000,2734.000)
\thinlines
\texture{55888888 88555555 5522a222 a2555555 55888888 88555555 552a2a2a 2a555555 
	55888888 88555555 55a222a2 22555555 55888888 88555555 552a2a2a 2a555555 
	55888888 88555555 5522a222 a2555555 55888888 88555555 552a2a2a 2a555555 
	55888888 88555555 55a222a2 22555555 55888888 88555555 552a2a2a 2a555555 }
\shade\path(1050,5674)(11850,5674)(11850,4399)
	(1050,4399)(1050,5674)
\path(1050,5674)(11850,5674)(11850,4399)
	(1050,4399)(1050,5674)
\texture{33333333 fff 33333333 fff 33333333 fff 33333333 fff 33333333 fff
33333333 fff 33333333 000 33333333 000}
\shade\path(10650,5674)(11850,5674)(11850,4399)
	(10650,4399)(10650,5674)
\path(10650,5674)(11850,5674)(11850,4399)
	(10650,4399)(10650,5674)
\shade\path(1050,5674)(2250,5674)(2250,4399)
	(1050,4399)(1050,5674)
\path(1050,5674)(2250,5674)(2250,4399)
	(1050,4399)(1050,5674)
\path(1050,274)(1050,276)(1051,280)
	(1053,287)(1057,298)(1061,315)
	(1067,336)(1074,363)(1083,396)
	(1094,435)(1106,479)(1119,527)
	(1134,580)(1149,636)(1165,695)
	(1182,756)(1200,818)(1217,880)
	(1235,943)(1253,1004)(1270,1065)
	(1287,1124)(1304,1181)(1320,1236)
	(1336,1289)(1351,1340)(1366,1389)
	(1381,1435)(1395,1480)(1409,1523)
	(1422,1564)(1436,1603)(1449,1640)
	(1461,1677)(1474,1712)(1487,1746)
	(1499,1779)(1512,1811)(1525,1843)
	(1538,1874)(1552,1908)(1567,1942)
	(1582,1975)(1597,2008)(1613,2041)
	(1629,2074)(1646,2106)(1664,2138)
	(1683,2170)(1702,2201)(1722,2232)
	(1743,2263)(1764,2293)(1787,2322)
	(1810,2351)(1835,2379)(1860,2406)
	(1886,2433)(1913,2458)(1941,2483)
	(1969,2506)(1999,2528)(2029,2550)
	(2060,2569)(2091,2588)(2124,2605)
	(2157,2622)(2190,2636)(2225,2650)
	(2260,2662)(2296,2674)(2332,2683)
	(2369,2692)(2408,2700)(2447,2706)
	(2488,2711)(2516,2715)(2546,2717)
	(2576,2719)(2607,2721)(2639,2722)
	(2671,2722)(2705,2722)(2740,2722)
	(2776,2721)(2812,2719)(2850,2717)
	(2889,2715)(2929,2712)(2971,2709)
	(3013,2705)(3057,2701)(3102,2696)
	(3148,2690)(3195,2685)(3243,2678)
	(3293,2672)(3343,2665)(3395,2657)
	(3448,2649)(3501,2641)(3556,2632)
	(3611,2623)(3667,2614)(3724,2604)
	(3782,2595)(3841,2584)(3900,2574)
	(3960,2563)(4020,2552)(4081,2541)
	(4142,2530)(4204,2518)(4266,2507)
	(4328,2495)(4391,2483)(4454,2471)
	(4518,2459)(4581,2447)(4646,2435)
	(4710,2423)(4775,2411)(4840,2399)
	(4906,2386)(4972,2374)(5039,2361)
	(5107,2349)(5175,2336)(5231,2326)
	(5288,2316)(5345,2306)(5404,2295)
	(5463,2285)(5522,2274)(5583,2263)
	(5644,2253)(5706,2242)(5769,2231)
	(5833,2219)(5898,2208)(5963,2197)
	(6030,2185)(6097,2174)(6165,2162)
	(6233,2150)(6302,2138)(6373,2126)
	(6443,2114)(6515,2102)(6586,2090)
	(6659,2078)(6731,2065)(6805,2053)
	(6878,2040)(6952,2028)(7026,2015)
	(7100,2003)(7174,1990)(7248,1978)
	(7322,1965)(7396,1952)(7470,1940)
	(7543,1927)(7616,1915)(7689,1903)
	(7761,1890)(7832,1878)(7903,1866)
	(7973,1854)(8043,1842)(8112,1830)
	(8179,1818)(8246,1807)(8312,1795)
	(8378,1784)(8442,1773)(8505,1761)
	(8567,1751)(8628,1740)(8688,1729)
	(8747,1719)(8804,1708)(8861,1698)
	(8917,1688)(8971,1678)(9025,1668)
	(9077,1658)(9129,1649)(9179,1639)
	(9229,1630)(9277,1621)(9325,1611)
	(9393,1598)(9459,1585)(9524,1572)
	(9587,1559)(9650,1547)(9711,1534)
	(9770,1521)(9829,1508)(9887,1495)
	(9944,1482)(9999,1469)(10054,1456)
	(10108,1443)(10160,1430)(10212,1417)
	(10262,1404)(10312,1391)(10360,1377)
	(10407,1364)(10453,1351)(10498,1337)
	(10541,1324)(10583,1311)(10624,1297)
	(10663,1284)(10701,1271)(10738,1257)
	(10773,1244)(10808,1231)(10840,1218)
	(10872,1205)(10902,1192)(10931,1179)
	(10959,1166)(10986,1153)(11012,1140)
	(11037,1127)(11060,1114)(11083,1101)
	(11105,1089)(11127,1076)(11148,1063)
	(11168,1050)(11188,1036)(11216,1017)
	(11243,997)(11270,976)(11296,954)
	(11322,932)(11348,908)(11374,883)
	(11400,856)(11427,828)(11454,799)
	(11481,767)(11509,733)(11538,698)
	(11568,661)(11598,623)(11628,583)
	(11657,544)(11687,504)(11715,466)
	(11741,429)(11766,395)(11787,365)
	(11806,338)(11821,317)(11832,300)
	(11840,288)(11846,280)(11849,276)(11850,274)
\put(900,3954){$\supp \psi_1 \phi_\delta$}
\put(11025,3954){$\supp \psi_2 \phi_\delta$}
\put(6975,6654){$\mathsf{H}_p$}
\put(5250,3954){$\supp \hat a$}
\put(10725,0){$\alpha_1$}
\put(0,3204){$\chi(\alpha_1) e^{-\lambda \alpha_1}$}
\thicklines
\path(4125,6199)(7350,6199)
\path(4125,6199)(7350,6199)
\blacken\path(6990.000,6109.000)(7350.000,6199.000)(6990.000,6289.000)(6990.000,6109.000)
\end{picture}
}
\caption{The support of the commutant and its value along the line
  $\alpha'=0.$  The support of the term
  $\psi_1(\alpha_1)\phi_\delta(\abs{\alpha'})$ is arranged to be contained in the
  complement of $\WF u,$ while the support of the whole of $a$ is arranged
  to be in the complement of $\WF^s u.$}
\end{figure}

Let $A \in \Psi^{2s-m+2}(X)$ be given by the quantization of
$a$.\footnote{I.e., really $A$ is given by cutting off $a$ near
  $\xi=0$ to give a smooth total symbol and quantizing that.}
Since
$\sigma_m(P)$ is real by assumption, we have $P^*-P \in \Psi^{m-1}(X).$
(Exercise: Check this!)  Thus the ``commutator'' $P^*A-AP,$ which is a priori of order
$2s+2,$ has vanishing principal symbol of order $2s+2,$ hence it in fact
lies in $\Psi^{2s+1}(X),$ and we may write
$$
(P^*A-AP) =[P,A]+(P^*-P)A ,
$$
with
\begin{multline} i\sigma_{2s+1}([P,A]+(P^*-P)A) = 
\hamvf_p(a)+\sigma_{m-1}(P^*-P)a\\ =-\lambda
\phi_\delta(\abs{\alpha'})\chi(\alpha_1)e^{-\lambda \alpha_1}\abs{\xi}^{2s+1} +
\phi_\delta(\abs{\alpha'})(\psi_1-\psi_2)e^{-\lambda \alpha_1}\abs{\xi}^{2s+1}\\
+(i\sigma_{m-1}(P^*-P)+h(\alpha) (2s-m+2))a,\label{foobar}
\end{multline}
by \eqref{homogeneity},\eqref{radial}, and \eqref{nicecoords}.  If $\lambda \gg 0$ is
chosen sufficiently large, we may absorb the third term into the first, and
write the RHS of \eqref{foobar} as
$$
-f(\alpha)\phi_\delta(\abs{\alpha'})\chi(\alpha_1) +
\phi_\delta(\abs{\alpha'})(\psi_1-\psi_2)e^{-\lambda \alpha_1}
$$ with $f >0$ on the support of $\phi_\delta \chi.$

Let $B\in \Psi^{(2s+1)/2}(X)$ be obtained by quantization of
$$
\abs{\xi}^{s+1/2}({f}(\alpha){\phi_\delta}(\abs{\alpha'}){\chi}(\alpha_1))^{1/2};
$$
and let $C_i\in \Psi^{(2s+1)/2}(X)$ be obtained by quantization of
$$
\abs{\xi}^{s+1/2}({\phi_\delta}(\abs{\alpha'})\psi_i(\alpha_1))^{1/2} e^{-\lambda \alpha_1/2}.
$$
Then by the symbol calculus, i.e.\ by Properties~\ref{property:symb},
\ref{property:exact} of the calculus of pseudodifferential operators,
\begin{equation}\label{thecommutator}
i(P^*A-AP) =i(P^*-P)A+ i[P,A] = -B^*B +C_1^* C_1-C_2^*C_2+R
\end{equation}
with $R \in \Psi^{2s}(X),$ hence of lower order than the other terms;
moreover we have $\WF' R \subset \supp \hat a.$

Now we ``pair'' both sides of \eqref{thecommutator} with our solution
$u.$  We have
$$
i\ang{(P^*A-AP)u,u} = \ang{(- B^*B +C_1^* C_1-C_2^*C_2+R)u,u};
$$
as we are taking $Pu=0,$ the LHS vanishes.\footnote{In the
  case of an inhomogeneous equation, it is of course here that extra
  terms arise.}  We thus have,
rearranging this equation,
\begin{equation}\label{poscomm}
\norm{B u}^2 + \norm{C_2 u}^2= \norm{C_1 u}^2  + \ang{Ru,u}.
\end{equation}
I claim that the RHS is finite:
Recall that $R$ lies in $\Psi^{2s}(X).$
Let $\Lambda$ be an
operator of order $s,$ elliptic on $\WF'R$ and with $\WF'\Lambda$
contained in the complement of $\WF^s u.$
\begin{exercise}
Show that such a $\Lambda$ exists.
\end{exercise}
Thus, letting $\Upsilon$ be
a microlocal parametrix for $\Lambda$ on $\WF' R,$ we have
$$
\WF' R \cap \WF' (\Id-\Lambda\Upsilon)=\emptyset,
$$
hence
$$
R-\Lambda\Upsilon R =E\in \Psi^{-\infty}(X).
$$
Thus,
$$
\abs{\ang{Ru,u}} \leq \abs{\ang{\Upsilon Ru, \Lambda^* u}} +\abs{\ang{Eu,u}}<\infty
$$
by Corollary~\ref{corollary:microbounded} since $\WF' \Upsilon R\cup
\WF' \Lambda^* 
\subset (\WF^s u)^c$ (and since $E$ is smoothing).  Returning to
\eqref{poscomm}, we also note that the term $\norm{C_1 u}^2$ is finite
by our assumptions on the location of $\WF^{s+1/2} u$ (and another use
of Corollary~\eqref{corollary:microbounded}).  Thus,
$$
\norm{B u}<\infty,
$$
and consequently, $$\WF^{s+1/2} u \cap \liptic B=\emptyset,$$
which was the desired estimate.\end{proof}

\begin{exercise}
Now see how the argument should be modified to yield absence of
$\WF^{s+1/2} u$ on 
$$
\{\alpha' \in [-1,0], \alpha'=0\}.
$$
One cheap alternative to going through the whole proof might be to notice
that we also have $(-P) u \in \CI,$ and that $\hamvf_{-p}=-\hamvf_p;$ thus, the
``forward propagation'' that we have just proved should yield backward
propagation along $\hamvf_p$ as well.
\end{exercise}

\noindent\textsc{The fine print:} Now, having done all that, note that
it was a cheat.  In particular, we didn't know a priori that we could
apply any of the operators that we used to $u$ and obtain an $L^2$
function, let alone justify the formal integrations by parts used to
move adjoints across the pairings.  Therefore, to make the above
argument rigorous, we need to modify it with an \emph{approximation
argument}.  This is similar to the situation in
Exercise~\ref{exercise:schrodinger}, except in that case, we had a
natural way of obtaining smooth solutions to the equation which
approximated the desired one: we could replace our initial data
$\psi_0$ for the Schr\"odinger equation by, for instance,
$e^{-\ep\Lap} \psi_0;$ the solution at later time is then just
$e^{-\ep\Lap} \psi,$ and we can consider the limit $\ep\downarrow 0.$
In the general case to which this theorem
applies, though, we do not have any convenient families of smoothing
operators commuting with $P.$ So we instead take the tack of smoothing
our \emph{operators} rather than the solution $u.$ We should
manufacture a family of smoothing operators $G_\ep$ that strongly
approach the identity as $\ep\downarrow 0,$ and replace $A$ by
$AG_\ep$ everywhere it appears above.  If we do this sensibly, then
the analogs of the estimates proved above yield the desired estimates
in the $\ep\downarrow 0$ limit.  Of course, we need to know how
$G_\ep$ passes through commutators, etc., so the right thing to do is
to take the $G_\ep$ themselves to be pseudodifferential approximations
of the identity, something like
$$
G_\ep =\Opl(\varphi(\ep\abs{\xi}))
$$ on $\RR^n,$ with $\varphi\in \CcI(\RR)$ a cutoff equal to $1$ near $0.$  We content ourselves with referring the interested
reader to \cite{Melrose:spectral} for the analogous development in the
``scattering calculus'' including details of the approximation
argument.

\begin{exercise}\label{exercise:hormandervariants}\
\begin{enumerate}
\item
Show the following variant of Theorem~\ref{theorem:hormander}: if $P\in \Psi^m(X)$ is
an operator with real principal symbol, and $P u \in \CI(X),$ show that
$\WF^k u$ is a union of maximally extended bicharacteristics of $P$ for
each $k \in \RR.$  (Hint: the proof is a subset of the proof of
Theorem~\ref{theorem:hormander}.)
\item
Show the following inhomogeneous variant of Theorem~\ref{theorem:hormander}: if $P\in
\Psi^m(X)$ is an operator with real principal symbol, and $P u =f,$ show
that $\WF u\backslash \WF f$ is a union of maximally extended
bicharacteristics of $P.$
\end{enumerate}
\end{exercise}

\begin{exercise}\label{exercise:schrodingerinvariance}\
\begin{enumerate}
\item What does Theorem~\ref{theorem:hormander} tell us about solutions to the
  Schr\"odinger equation?  (Hint: not much.)
\item Nonetheless: let $\psi(t,x)$ be a solution to the Schr\"odinger equation on
  $\RR\times X$ with $(X,g)$ a Riemannian manifold; suppose that
  $\psi(0,x)=\psi_0 \in H^{1/2}(X).$  Define a set $S_{1}\subset S^*X$ by 
\begin{multline*}
q\notin S_{1} \Longleftrightarrow \text{there exists } A \in
\Psi^{1}(X),\ q \in \liptic(A),\\ \text{such that } \int_0^1 \norm{A\psi}^2
\,dt <\infty.
\end{multline*}
(In other words, $S_{1}$ is a kind of wavefront set measuring where in
the phase space $S^*X$ we have $\psi \in L^2([0,1]; H^{1}(X))$---cf.\
Exercise~\ref{exercise:schrodinger}.)

\emph{Show that $S_1$ is invariant under the geodesic flow on $S^*X.$}
(See Exercise~\ref{Exercise:geodesicflow} for the definition of
geodesic flow.)

(Hint: use \eqref{Ehrenfest} with $A$ an appropriately chosen
pseudodifferential operator of order zero, constructed much like the ones
used in proving Theorem~\ref{theorem:hormander}.)

Reflect on the following interpretation: ``propagation of $L^2H^{1}$
regularity for the Schr\"odinger equation occurs at infinite speed along
geodesics.''
\end{enumerate}
\end{exercise}

\section{Traces}

It turns out to be of considerable interest in spectral geometry to
consider the \emph{traces} of operators manufactured from $\Lap,$ the
Laplace-Beltrami operator on a compact\footnote{We especially emphasize that
$X$ denotes a \emph{compact} manifold throughout this section.} Riemannian manifold.  The
famous question posed by Kac \cite{Kac}, ``Can one hear the shape of a
drum,'' has a natural extension to this context:  Recall from
Exercise~\ref{exercise:Lapspectrum} that there exists an orthonormal
basis $\phi_j$ of eigenfunctions of $\Lap$ with eigenvalues
$\lambda_j^2 \to +\infty;$ what, one wonders, can one recover of the
geometry of a Riemannian manifold from the sequence of frequencies
$\lambda_j$?  Using PDE methods to understand traces of
functions of the Laplacian has led to a better
understanding of these inverse spectral problems.

Recall from Proposition~\ref{proposition:sqrtlap} that $\sqrt\Lap$ is
a first-order pseudodifferential operator on $X.$
It is a slightly inconvenient fact that while $\sqrt{\Lap} \in
\Psi^1(X),$ $\sqrt{\Lap} \notin
\Psi^1(\RR\times X):$ its Schwartz kernel is easily seen to be singular
away from the diagonal.  But this turns out be be of little practical
importance for our considerations here: it is close enough!

Let us now consider the operator
\begin{equation}\label{waveop}
U(t) =e^{-it\sqrt\Lap}
\end{equation}
which can be defined by the functional calculus to act as
the scalar operator $e^{-it\lambda_j}$ on each $\phi_j.$   $U(t)$ is
unitary, and indeed is the solution operator to the Cauchy problem for the equation
\begin{equation}\label{halfwave}
(\pa_t+i\sqrt{\Lap})u=0;
\end{equation}
that is to say, if $u=U(t) f,$ we have
$$
(\pa_t+i\sqrt{\Lap})u=0,\quad \text{and } u(0,x) = f(x).
$$
Equation \eqref{halfwave} is easily seen to be very closely related to
the wave equation: if $u$ solves \eqref{halfwave} then applying
$\pa_t -i\sqrt{\Lap},$ we see that $u$ also satisfies the wave
equation.  Of course, \eqref{halfwave} only requires a single Cauchy
datum, unlike the wave equation, so the trade-off is that the Cauchy
data of $u$ as a solution to $\Box u=0$ are constrained: we
have
$$
u(0,x) =f(x),\quad \pa_t u(0,x) = -i\sqrt{\Lap} f.
$$
The real and imaginary parts of the operator
$U(t)$ are exactly the solution operators to the (more
usual) Cauchy problem for the wave equation with $u(0,x)=f(x), \pa_t
u(0,x)=0$ and with $u(0,x)=0, \pa_t
u(0,x)=-i\sqrt{\Lap} f(x)$ respectively.

Why is the operator $U(t)$ of interest?  Well, suppose that we are
interested in the sequence of $\lambda_j$'s.  It makes sense to combine
these numbers into a generating function, and certainly one option
would be to take the exponential sum\footnote{This choice of
  generating function, corresponding to taking the wave trace, is of
  course one choice among many.  Some other approaches include taking
  the trace of the complex powers of the Laplacian or the heat trace.
  The idea of using (at least some version of) the wave trace
  originates with Levitan and Avakumovi\v{c}.}
$$
\sum_j e^{-it\lambda_j}
$$
\emph{This is, at least formally, nothing but the trace of the
  operator $U(t).$}
One of the principal virtues of this generating function is that if we
let $N(\lambda)$ denote the ``counting function''
$$
N(\lambda) =\#\{\lambda_j \leq \lambda\},
$$
then we have
$$
N'(\lambda) = \sum_j \delta(\lambda-\lambda_j),
$$
hence
$$
\sum e^{-it\lambda_j}=(2\pi)^{n/2} \mathcal{F}_{\lambda \to t} (N'(\lambda))(t).
$$

This is all a bit optimistic, as $U(t)$ is easily seen to be not of trace
class---for example at $t=0$ it is the identity.  So we should try and
think of $\Tr U(t)$ as a \emph{distribution}.  We do know
that for any test function $\varphi(t) \in \schwartz (\RR)$ and any
$f \in L^2(X),$
\begin{equation}\label{traceasdist}
\begin{aligned}
\int \varphi(t) U(t)f \, dt  &= \int (1+D_t^2)^{-k} (1+D_t^2)^k
(\varphi(t))U(t) f \, dt \\ &= \int  (1+D_t^2)^k
(\varphi(t)) (1+D_t^2)^{-k}  U(t) f\, dt \\  &= \int (1+D_t^2)^k
(\varphi(t)) (1+\Lap)^{-k}  U(t) f\, dt,
\end{aligned}
\end{equation}
since $D_t^2U=\Lap U.$
Here we can, if we like, consider $(1+\Lap)^{-k}$ to be defined by
the functional calculus; it is in fact pseudodifferential, of order $-2k.$
We easily obtain (using either point of view) the estimate:
$$
(1+\Lap)^{-k}  U(t): L^2(X) \to H^{2k}(X);
$$
hence, for $k\gg 0,$ the operator $(1+\Lap)^{-k}  U(t)$ is of trace
class.
\begin{exercise}
Prove that this operator is of trace class for $k\gg 0.$
(\textsc{Hint:} One easy route is to think about first choosing $k$
large enough that the Schwartz kernel is continuous, hence the
operator is Hilbert-Schmidt; then you can take $k$ even larger to get a
trace-class operator, by factoring into a product of two
Hilbert-Schmidt operators (see Appendix).)
\end{exercise}

Equation \eqref{traceasdist} thus establishes that
$$
\Tr U(t): \varphi \mapsto \Tr \int \varphi(t) U(t) \, dt
$$
makes sense as a distribution on $\RR.$  We can thus write
\begin{equation}\label{trformula1}
\Tr U(t) = (2\pi)^{n/2}\mathcal{F}( N')(t).
\end{equation}
where both sides are defined as distributions.  Our next goal
is to try to understand the left side of this equality through PDE
methods.

\begin{exercise}\label{exercise:computetrace}
Show that if the Schwartz kernel $K(x,y)$ of a bounded, normal
operator $T$ on $L^2(X)$
is in $\mathcal{C}^k(X)$ for sufficiently large $k,$ then $T$ is of
trace-class and
$$
\Tr T = \int K(x,x) \, dg(x).
$$
(\textsc{Hint:} Check that $K$ is trace-class as in the previous
exercise.  Then apply the spectral theorem for compact normal
operators, and use the basis of eigenfunctions of $K$ when computing
the trace.  The crucial thing to check is that if $\varphi_j$ are the
eigenfunctions, then
$$
\sum \varphi_j(x)\overline{\varphi_j(y)} =\delta_\Delta,
$$
the delta-distribution at the diagonal, since this is nothing but a
spectral resolution of the identity operator.)
\end{exercise}

As a consequence of Exercise~\ref{exercise:computetrace}, we can compute the distribution $\Tr
U(t)$ in another way if we can compute the Schwartz kernel of $U(t).$
Indeed, knowing even rather crude things about $U(t)$ can give us
some useful information here.

\begin{theorem}\label{theorem:halfwaveflow}
Let $\Phi_t$ be the geodesic flow, i.e.\
the flow generated by the Hamilton vector field of
$\abs{\xi}_g\equiv (\sum g^{ij} \xi_i\xi_j)^{1/2}.$  Then
$$
\WF U(t) f = \Phi_t( \WF f).
$$
\end{theorem}

We begin with a lemma:
\begin{lemma}
Let $(\pa_t+i \sqrt{\Lap})u=0.$ 
Then  $$(x_0,\xi_0) \in \WF u|_{t=t_0}$$ if and only if
$$
(t=t_0,\tau=-\abs{\xi_0}, x_0, \xi_0) \in \WF u.
$$
\end{lemma}
\begin{proof}\footnote{I am grateful to Andr\'as Vasy for showing me this proof.}
Suppose $q=(x_0,\xi_0) \in \WF u|_{t=t_0}.$ Since $\tilde{q}=(t=t_0,\tau=-\abs{\xi_0},
x_0, \xi_0)$ is the only vector in $\Sigma_{\pa_t+i\sqrt\Lap}$ that
projects to $(x_0,\xi_0),$ it must lie in the wavefront set of $u$ by
Exercise~\ref{exercise:WFrestriction}.

The converse is harder.  Suppose $q \notin \WF u|_{t=t_0}.$  Let $v=H(t-t_0) u,$ with $H$ denoting the
Heaviside function.  Then
$$
(\pa_t + i \sqrt{\Lap}) v= \delta(t-t_0) u(t_0,x)\equiv f.
$$
and $v$ vanishes identically for $t<t_0.$  By the last part of
Exercise~\ref{exercise:WFrestriction},
$$
\tilde{q} \notin \WF f,
$$ hence (since $\WF f$ only lies over $t=t_0$) certainly no points
along the bicharacteristic through $\tilde q$ lie in $\WF f.$
Moreover, no points along this bicharacteristic lie in $\WF v$ for
$t<t_0$ (since $v$ is in fact zero there).  Hence by the version of
the propagation of singularities in the second part of
Exercise~\ref{exercise:hormandervariants}, this bicharacteristic is
absent from $\WF u.$ In particular, $\tilde q \notin \WF u.$
\end{proof}

Theorem~\ref{theorem:halfwaveflow} now follows directly\footnote{Here is one of the places
where we should worry about the fact that $\sqrt\Lap$ is not a
pseudodifferential operator on $\RR\times X.$ This problem is seen not to
affect the proof of H\"ormander's theorem if we note that composing
$\sqrt{\Lap}$ with a pseudodifferential operator that is microsupported in
a neighborhood of the characteristic set $\{\abs{\tau}=-\abs{\xi}_g\}$
yields an operator that \emph{is} pseudodifferential, and that the symbol
calculus extends to such compositions.  (The author confesses that this is
not entirely a trivial matter.)}
from the lemma and Theorem~\ref{theorem:hormander}.

We now require a result on microlocal partitions of unity somewhat
generalizing Lemma~\ref{lemma:pof1}:
\begin{exercise}\label{exercise:pof1}
Let $\rho_j,$ $j=1,\dots, N$ be a smooth partition of unity for
$S^*X.$ Show
that there exists $A_j\in \Psi^0(X)$ with $\WF' A_j = \supp
\rho_j,$ $\hsigma_0(A_j)=\rho_j,$ $A_j^*=A_j,$ and
$$
\sum_{j=1}^N A_j^2=\Id-R,
$$
with $R \in \Psi^{-\infty}(X).$
\end{exercise}

For a distribution $u,$ let $\singsupp u$ (the ``singular support'' of
$u$) be the projection of its
wavefront set, i.e.\ the complement of the largest open set on which
it is in $\CI.$
\begin{theorem}\label{theorem:poisson}
$$\singsupp \Tr U(t) \subseteq \{0\} \cup \{\text{lengths of closed
  geodesics on } X\}.$$
\end{theorem}
This theorem is due to Chazarain and to Duistermaat-Guillemin.

We begin with the following dynamical result:
\begin{lemma}\label{lemma:covering}
Let $L$ not be the length of any closed geodesic.  Then there exists $\ep>0$ and
a cover $U_i$ of $S^*X$ by open sets such that for $t \in (L-\ep,L+\ep),$
there exists no geodesic with start- and endpoints both contained in the
same $U_i.$
\end{lemma}
\begin{exercise}\
\begin{enumerate}
\item
Prove the lemma.  (\textsc{Hint:} The cosphere bundle is compact.)
\item
As long as you're at it,
  show that $0$ is an isolated point in the set of lengths of closed
  geodesics (``length spectrum''), and that the length spectrum is a
  closed set.
\end{enumerate}
\end{exercise}

We now prove Theorem~\ref{theorem:poisson}.
\begin{proof}
Let $L$ not be the length of any closed geodesic on $X.$  Let $U_j$ be
a cover of $S^*X$ as given by Lemma~\ref{lemma:covering}.
Let $\rho_j$ be a partition of unity subordinate to $U_j$ and let
$A_j$ be a microlocal partition of unity as in
Exercise~\ref{exercise:pof1}.  Then, calculating with distributions on
$\RR^1,$ we have
\begin{align*}
\Tr U(t) &= \sum_{j} \Tr A_j^2 U(t)+ \Tr R U(t) \\
&=\sum_{j} \Tr A_j U(t) A_j + \Tr R U(t)
\end{align*}
and, more generally,
$$
D_t^{2m}\Tr U(t) = \sum_{j} \Tr A_j \Lap^{m} U(t) A_j + \Tr R \Lap^m
U(t).
$$ 
Let $u$ be a distribution on $X;$ then
$\WF A_j u\subseteq 
\WF' A_j \subset U_j.$ Thus Theorem~\ref{theorem:halfwaveflow} gives
$$
\WF \Lap^{m} U(t) A_j u \subseteq \Phi_t(U_j).
$$
But by construction, this set is disjoint from $U_j$ and hence from
$\WF' A_j.$ Hence for any $m,$\footnote{We technically have to work
  just a little to obtain the uniformity in time: observe that $A_j
  \Lap^m U(t) A_j$ are a continuous (or even smooth) \emph{family} of
  smoothing operators.  We have been avoiding the topological
  issues necessary to easily dispose of such matters, however.}
$$
A_j \Lap^{m} U(t) A_j \in L^\infty([L-\ep,L+\ep]; \Psi^{-\infty}(X));
$$
consequently,
$$
D_t^{2m}\Tr U(t) \in L^\infty([L-\ep,L+\ep]).\qed
$$
\noqed
\end{proof}

\begin{exercise}
Show that in the special case of $X=S^1,$ Theorem~\ref{theorem:poisson}
can be deduced from the Poisson summation formula.  For this reason it
is often referred to as the \emph{Poisson relation}.
\end{exercise}

One is tempted to conclude from \eqref{trformula1} and
Theorem~\ref{theorem:poisson} that one can ``hear'' the lengths of
closed geodesics on a manifold, since the right side of
\eqref{trformula1} is determined by the spectrum, and the left side
seems to be a distribution from whose singularities we can read off
the lengths of closed geodesics.  The trouble with
this approach is that we do not know with any certainty from
Theorem~\ref{theorem:poisson} that the putative singularities in $\Tr U(t)$ at
lengths of closed geodesics are actually there: perhaps the
distribution is, after all, miraculously smooth.  Thus, proving actual
inverse spectral results requires somewhat more care, as we shall
see.  To this end, we will begin studying the operator $U(t)$ more
constructively in the following section.

\section{A parametrix for the wave operator}\label{section:waveparametrix}

In order to learn more about the wave trace, we will have to bite the
bullet and construct an approximation (``parametrix'') for the
fundamental solution to the wave equation on a manifold.  The approach
will have a similar iterative flavor to the technique we used to
construct an approximate inverse for an elliptic operator, but
we have now left the comfortable world of
pseudodifferential operators: the parametrix we construct is going to
be something rather different.  Exactly what, and how to systematize
the kinds of calculation we do here, will be discussed later on.

As this construction will be local, we will work in a single
coordinate patch, which we identify with $\RR^n;$ for the sake of
exposition, we omit the coordinate maps and partitions of unity
necessary to glue this construction into a Riemannian manifold.

Consider once again the ``half-wave equation''\footnote{Remember that
  $D_t=i^{-1} \pa_t.$}
\begin{equation}\label{halfwaveblah}
(D_t+\sqrt\Lap)u=0
\end{equation}
on $\RR^n,$ where $\Lap$ is the Laplace-Beltrami operator with respect
to a metric $g.$  Our goal is to find a distribution $u$ approximately solving
\eqref{halfwaveblah} with initial data
$$
u(0,x,y)=\delta(x-y)
$$
for any $y \in \RR^n.$  Recall that if we let $U$ denote the exact
solution to \eqref{halfwaveblah} with initial data $\delta(x-y)$ then
$U$ can also be
interpreted as (the Schwartz kernel of) the ``solution operator''
mapping initial data $f$ to the solution $e^{-it\sqrt\Lap} f$ with that initial data,
evaluated at time $t;$ this is why we denote it $U,$ as we did above,
and why we will often think of our parametrix $u(t,x,y)$ as a family in $t$ of
integral kernels of operators on $\RR^n.$

We do not expect $U(t,x,y)$ or our parametrix for it to be the Schwartz kernel of a
pseudodifferential operator, as it moves wavefront set around, by
Theorem~\ref{theorem:hormander}; recall that pseudodifferential operators are
\emph{microlocal,} which is to say they don't do that.  But we will
try and construct our parametrix $u(t,x,y)$ as something of
\emph{roughly} the same form, which is to say as an oscillatory integral
$$
u(t,x,y)=\int a(t,x,\eta) e^{i\Phi}\, d\eta
$$
where the main difference is that the ``phase function''
$\Phi=\Phi(t,x,y,\eta)$ will be something a good deal more interesting
than $(x-y)\cdot \eta;$ indeed, this phase function is where all the
geometry of the problem turns out to reside.

First, let's write our initial data as an oscillatory integral:
$$
\delta(x-y) =  (2\pi)^{-n} \int e^{i(x-y)\cdot \eta} \, d\eta.
$$
Let us now try, as an Ansatz, modifying the phase as it varies in $t,x$ by setting
\begin{equation}\label{waveansatz}
u(t,x,y)=(2\pi)^{-n} \int a(t,x,\eta) e^{i (\phi(t,x,\eta)-y\cdot \eta)}\, d\eta;
\end{equation}
then if $\phi(0,x,\eta)=x\cdot\eta$ and $a(0,x,\eta)=1,$ we recover our
initial data; moreover, if $\phi$ were to remain unchanged as $t$
varied we would have nothing but a family of pseudodifferential operators.  Let
us assume that $a$ is a classical symbol of order $0$ in $\eta,$ so that we have an
asymptotic expansion
$$
a \sim a_0 + \abs{\eta}^{-1} a_{-1}+ \abs{\eta}^{-2}
a_{-2}+\dots,\quad a_j=a_j(t,x,\hat\eta).
$$
Let us further assume that $\phi$ is homogeneous in $\eta$ of degree $1,$
hence matches the homogeneity\footnote{That is is then likely to be
  singular at $\eta=0$ will not in fact concern us, as it will turn out
  that we may as well assume that $a$ vanishes near $\eta=0.$} of $x\cdot\eta.$

Now if $u$ solves the half-wave equation, it solves the wave equation,
hence we have
$$
\Box u=0;
$$
As we seek an approximate solution, we will instead accept
$$
\Box u \in \CI((-\ep,\ep)_t\times \RR^n).
$$
Our strategy is to plug  \eqref{waveansatz} into this equation and see what is forced upon
us.  
To this end, note that if we have an expression
\begin{equation}
v=(2\pi)^{-n} \int b(t,x,y,\eta) e^{i(\phi(t,x,\eta)-y\cdot \eta)}\, d\eta;
\end{equation}
where $b$ is a symbol of order $-\infty,$ then $v$ lies in $\CI,$ as the
integral converges absolutely, together with all its $t,x,y$
derivatives.  So terms of this form will be acceptable errors.

Applying $\Box$ to \eqref{waveansatz}, we group terms according to
their order in $\eta.$  The ``worst case'' terms involve factors of
$\eta^2,$ and
can only be produced by second-order terms in $\Box,$ with all
derivatives falling on
the exponential term.  Since the second-order terms in $\Lap$ are just
$$
\sum g^{ij}(x) D_i D_j,
$$
we can write the term this produces from the phase as
$\abs{d_x\phi}^2_g$ or, equivalently, $\abs{\nabla_x\phi}^2_g.$
Thus, the equation that we need to solve to make the $\eta^2$ terms
vanish is just
\begin{equation}\label{eik2}
(\pa_t\phi)^2-\abs{\nabla_x \phi}_g^2 =0.
\end{equation}
Recall that we further want our phase to agree with the standard
pseudodifferential one at time zero, i.e.\ we want
\begin{equation}\label{eik3}
\phi(0,x,\eta)=x\cdot\eta.
\end{equation}
Combining this information with \eqref{eik2} we easily see that we
in particular have
$$
(\pa_t \phi|_{t=0})^2=\abs{\eta}^2_g,
$$
and we need to make an arbitrary choice of sign in solving this to get
the initial time-derivative: we will choose\footnote{We will
  use this solution for reasons that will become apparent presently---it
  is the right one to solve \eqref{halfwave} and not merely the wave
  equation.}
\begin{equation}\label{phisignchoice}
\pa_t \phi|_{t=0}=-\abs{\eta}_g.
\end{equation}

If our metric is the Euclidean metric, we can easily solve
\eqref{eik2}, \eqref{eik3}, and \eqref{phisignchoice} by setting
$$
\phi(t,x,\eta) =x\cdot \eta -t\abs{\eta}.
$$
More generally, the construction of a phase satisfying \eqref{eik2},\eqref{eik3} and
\eqref{phisignchoice} is the classic construction of Hamilton-Jacobi
theory, and is sketched in the following exercise.
\begin{exercise}\label{exercise:hamjac}\
\begin{enumerate}\item
Show that equation \eqref{eik2} is equivalent to the statement that
for each $\eta,$ the graph of $d_{t,x} \phi(t,x,\eta)$ is contained in the set
$$\Lambda=\{\tau^2-\abs{\xi}_g^2=0\} \subset T^* (\RR_t \times
\RR^n_x)$$
(where the variables $\tau$ and $\xi$ are the canonical dual variables
to $t$ and $x$ respectively). The condition \eqref{eik3} implies
$$
d_{x} \phi(t,x,\eta)|_{t=0} = \eta \cdot dx.
$$
Equation \eqref{phisignchoice} gives further
\begin{equation}\label{signchoice2}
d_{t,x} \phi(t,x,\eta)|_{t=0} = -\abs\eta \, dt + \eta \cdot dx;
\end{equation}
accordingly, for fixed $\eta,$ let $$G_0=  \{t=0, x\in\RR^n,\tau=-\abs\eta, \xi=\eta\}\subset T^*(\RR\times \RR^n).$$
\item
Let $\hamvf$ denote the Hamilton vector field of
$\tau^2-\abs{\xi}^2_g.$  Show that flow along $\hamvf$ preserves
$\Lambda$ and that $\hamvf$ is transverse to $G_0.$
\item
Show that there is a solution to \eqref{eik2},\eqref{signchoice2} for $t \in
(-\ep,\ep)$ where the graph of $d_{t,x} \phi$ is given by
flowing out the set $G_0$ under $\hamvf.$  (Among other things, you
need to check that the resulting smooth manifold is indeed the graph
of the differential of a function.)  Show that this solution can be
integrated to give a solution to \eqref{eik2},\eqref{eik3}.
\end{enumerate}
\end{exercise}

Employing the phase $\phi$ constructed in Exercise~\ref{exercise:hamjac},
we have now solved away the homogeneous degree-two (in $\eta$) terms in the application of
$\Box$ to our parametrix.  We thus move on to the degree-one terms,
which are as follows:
\begin{equation}\label{zurich1}
2 D_t \phi D_t a_0 -  2 \ang{D_x\phi, D_x}_g a_0 +
r_1(t,x,y,\eta)
\end{equation}
where $r_1$ is a homogeneous function of degree $1$ \emph{independent
of $a_0$,} i.e.\ determined completely by $\phi.$
Given that $\phi$ solves the eikonal equation, we can rewrite
\eqref{zurich1} by factoring out $\abs{\nabla_x\phi}$ and noting that
our sign choice $\pa_t\phi=-\abs{\nabla_x\phi}$ must persist
away from $t=0$ (for a short time, anyway).  In this way we obtain
$$
  2\pa_t a_0+2\Big\langle\frac{\nabla_x\phi}{\abs{\nabla_x\phi}}_g, \pa_x\Big\rangle_g a_0
  -\tilde r_1=0,
$$
with $\tilde r_1$ homogeneous of degree $0.$
This is a \emph{transport equation} that we would like to solve, with the
initial condition $a_0(0, x,y,\eta)=1$ (the symbol of the identity operator).
We can easily see that a solution exists with the desired initial
condition $a_0(0,y,\eta)=1,$ as, letting
$$
\hamvf=2\pa_t+2\Big\langle\frac{\nabla_x\phi}{\abs{\nabla_x\phi}}_g, \pa_x\Big\rangle_g
$$
we see that $\hamvf$ is a nonvanishing vector field, transverse to $t=0,$
hence we may solve
$$
\hamvf a_0=\tilde r_1,\quad a_0|_{t=0}=1
$$
by standard ODE methods.

Now we consider degree-zero terms in $\eta.$  We find that they are of the form
$$
2 D_t \phi D_t a_{-1} -  2 \ang{D_x\phi, D_x}_g a_{-1} +
r_0(t,x,y,\eta)
$$
where $r_0$ only depends on $a_0$ and $\phi$ (i.e.\ not on $a_{-1}$).
Thus, we may use the same procedure as above to find $a_{-1}$ with
initial value zero, making the degree-zero term vanish.  (Note that
the vector field $\hamvf $ along which we need to flow remains the
same as in the previous step.)

We continue in this manner, solving successive transport equations
along the flow of $\hamvf$ so
as to drive down the order in $\eta$ of the error term.  Finally we
Borel sum the resulting symbols, obtaining a symbol
$$
a(t,x,\eta) \in S^0_\cl(\RR^{2n}_{x,y} \times \RR^n_\eta)
$$
such that
$$
a(0,x,\eta)=1,
$$
and
\begin{multline}\label{parametrixresult}
  \Box u=\Box \left((2\pi)^{-n} \int a(t,x,\eta)
    e^{i(\phi(t,x,\eta)-y\cdot \eta)}\,
    d\eta\right)\\=(2\pi)^{-n} \int b(t,x,y,\eta)
    e^{i(\phi(t,x,\eta)-y\cdot \eta)}\, d\eta \in
  \CI((-\ep,\ep)\times X),
\end{multline}
since $b \in S^{-\infty}.$

Now we need to check that \eqref{parametrixresult} implies that in
fact $u$ differs by a smooth term from the actual solution.  We will show soon (in the next section) that
our choice of the phase implies that\footnote{This can also be verified
  directly, with localization, Fourier transform, and elbow grease.}
$\WF u \subset \{\tau<0\}.$  Hence, using this fact, we have
\begin{equation}\label{factoredwave}
(\pa_t-i\sqrt\Lap)(\pa_t+i\sqrt\Lap)u =f \in \CI;
\end{equation}
Now $\pa_t-i\sqrt\Lap$ is elliptic on $\tau<0,$ so, letting $Q$ denote
a microlocal elliptic parametrix, we have
$$
Q (\pa_t-i\sqrt\Lap)=I+E
$$
with $\WF' E \cap \WF u =\emptyset.$  Thus, applying $Q$ to both sides
of \eqref{factoredwave}, we have
$$
(\pa_t+i\sqrt\Lap )  u \in \CI.
$$
Also, as we have arranged that $a(0,x,\eta)=1,$ we have got our
initial data exactly right: $u(0,x,y)=\delta(x-y).$  Letting $U$ denote the
actual solution operator to \eqref{halfwave}, we thus find
$$
(\pa_t+i\sqrt\Lap)(u-U) \in \CI,\quad u(0,x,y)-U(0,x,y)=0;
$$
hence by global energy estimates\footnote{We can either use the
  estimates developed in \S\ref{section:wave}, adapted to this variable
  coefficient setting, and with a power of the Laplacian applied to
  the solution (in order to gain derivatives); or we can apply
  Theorem~\ref{theorem:hormander}, which is overkill.}
we have
$$
u-U \in \CI((-\ep,\ep)\times \RR^n).
$$

\section{The wave trace}
Our treatment of this material (and, in part, that of the previous
section) closely follows the treatment in \cite{GrigisSjostrand}, which is in
turn based on work of H\"ormander \cite{Hormander:Spectral}.

Recall that, if $N(\lambda)=\#\{\lambda_j \leq \lambda\}$ and $U(t)$
is given by \eqref{waveop}, then
\begin{equation}\label{ftofcounting}
\Tr U(t) =(2\pi)^{n/2}\mathcal{F}(N'(\lambda)).
\end{equation}
Thus, the singularities of $\Tr U(t)$ are related to the growth of
$N(\lambda).$ We think that $\Tr U(t)$ should have singularities at
zero, together with lengths of closed geodesics; since $U(0)$ is the
identity (which has a very divergent trace), the singularity at $t=0,$
at least, seems certain to appear.  We will thus spend some time
discussing this singularity of the wave trace and its consequences for
spectral geometry.

What is the form of the singularity of $\Tr U(t)$ at $t=0$?  Our
parametrix from the previous section was
$$
u(t,x,y)=(2\pi)^{-n} \int a(t,x,\eta) e^{i(\phi(t,x,\eta) -y\cdot\eta)}\, d\eta,
$$
where $\phi(t,x,\eta) = x\cdot \eta -t \abs{\eta}_{g(x)}+O(t^2),$
and $a(t,x,\eta) = 1+O(t).$  Thus,
\begin{equation}\label{ondiag}
u(t,x,x)=(2\pi)^{-n} \int a(t,x,\eta) e^{i(-t\abs{\eta}_{g(x)} + O(t^2\abs{\eta}))}\, d\eta,
\end{equation}
where we have used the homogeneity of the phase in writing the error
term as $O(t^2\abs{\eta}).$

Formally, we would now like to conclude that the singularity at $t=0$
is approximately that of
$$
u(t,x,x)=(2\pi)^{-n} \int e^{-it\abs{\eta}_{g(x)}} \, d\eta
$$
so that integrating in $x$ would give, if all goes well,
\begin{equation}\label{heuristic}
\begin{aligned}
\Tr U(t) &\sim \int u(t,x,x) \, dx  \\ &\sim (2\pi)^{-n} \iint
e^{-it\abs{\eta}_{g}} \, d\eta\, dx\\
&= (2\pi)^{-n} \iiint_{\sigma>0,\abs\theta=1} e^{-it \sigma \abs{\theta}_g} \sigma^{n-1}\,  d \sigma
\, d\theta\, dx\\ 
& = (2\pi)^{-n/2}\iint \fcal(\sigma^{n-1} H(\sigma))(t\abs{\theta}_g)
\, d\theta \, dx,
\end{aligned}
\end{equation}
with $H$ denoting the Heaviside function.
(Recall that the notation $f\sim g$ means that $(f/g) \to 1,$ in this
case as $t \to 0.$)
If we crudely try to solve \eqref{ftofcounting} for $N'(\lambda)$ by applying an inverse
Fourier transform to $\Tr U(t)$ and pretending that the
singularity of $\Tr U(t)$ at $t=0$ is all that matters, we find, formally,
that \eqref{heuristic} yields
\begin{align*}
N'(\lambda) &\sim (2\pi)^{-n/2}\mathcal{F}_{t\to\lambda}^{-1}\Tr U(t) \\
&\sim (2\pi)^{-n} \iint_{\abs{\theta}=1} \abs{\theta}_g^{-1}
\Big(\frac{\lambda}{\abs{\theta}_g}\Big)^{n-1} \, d\theta \, dx\\
&= (2\pi)^{-n}\lambda^{n-1} \iint_{\abs{\theta}=1} \abs{\theta}_g^{-n} \, d\theta
\, dx.
\end{align*}
Integrating would formally yield
\begin{align*}
N(\lambda) &\sim (2\pi)^{-n}\frac{\lambda^n}{n}  \iint_{\abs{\theta}=1} \abs{\theta}_g^{-n} \, d\theta
\, dx\\
&= (2\pi)^{-n}\lambda^n \iiint_{\abs{\theta}=1,\rho\in (0,1)} \abs{\theta}_g^{-n}
\rho^{n-1} \, d\rho \, d\theta
\, dx\\
&=(2\pi)^{-n}\lambda^n \iiint_{\abs{\sigma\theta}_g<1} \sigma^{n-1} \, d\sigma d\theta\, dx,
\end{align*}
where we have, in the last line, set $\sigma=\rho/\abs{\theta}_g,$
with the result that definition of the region of integration now
involves the metric.  This last quantity can easily be seen to be
simply the volume in phase space of the set $\abs{\xi}_g<1,$ otherwise
known as the unit ball bundle.\footnote{Recall that on a symplectic
  manifold $(N^{2n},\omega)$ we have a naturally defined volume form
  $\omega^n,$ and it is this volume that we are integrating over the
  unit ball here.}
Thus, we obtain \emph{formally}
$$
N(\lambda) \sim (2\pi)^{-n} \lambda^{n} \Vol(B^*X)=(2\pi)^{-n} \Vol(\{\abs{\xi}_g<\lambda\}).
$$

This is all nonsense, of course, for several different reasons.
First, we were very imprecise about dropping higher order terms in $t$
in computing the asymptotics of the trace as $t\to 0.$ Furthermore, we
formally computed with $N'$ as if it were a smooth function, but of
course $N'$ is quite singular (a sum of delta distributions).
Moreover, and potentially most seriously, there are in general
infinitely many singularities in $\Tr U(t)$ that might be contributing
to the asymptotic behavior of its Fourier transform: we have been
concerning ourselves only with the one near $t=0.$ However: the above argument does
give the right leading order asymptotics, the so-called ``Weyl Law.''
What follows is (the outline of) a rigorous version of the above argument.

To begin, we need a cutoff function to localize us near the
singularity at $t=0,$ where our parametrix is valid.
\begin{exercise}
  Show that there exists $\rho \in \schwartz(\RR)$ with $\hat\rho$
  compactly supported, $\hat\rho(0)=1,$ $\hat\rho(t)=\hat\rho(-t),$
  $\rho(\lambda) >0$ for all $\lambda,$ and $\hat \rho$ supported in an
  arbitrarily small neighborhood of $0.$ (\textsc{Hint:} Start with a
  smooth, compactly supported $\hat\rho;$ convolve with its complex
  conjugate, and scale.)
\end{exercise}

We now consider
\begin{multline*}
\fcal^{-1}_{t \to\lambda}\big(\hat\rho(t) \Tr u(t)\big) \\=(2\pi)^{-n-1/2}
\iiint \hat\rho(t) a(t,x,\eta) e^{i(t(\lambda-\abs{\eta}_{g}) +
  O(t^2\abs{\eta}))}\, dx \, d\eta \, dt \\ = (2\pi)^{-n-1/2} \iiint
\hat\rho(t) a(t,x,\lambda\sigma\theta) e^{it\lambda (1-\sigma + O(t^2\sigma))}(\lambda\sigma)^{n-1}\, dx \,
d\sigma \, d\theta \, dt;
\end{multline*}
here we have used the change of variables $\eta=\lambda \sigma \theta$
with $\abs{\theta}=1.$
We now employ the \emph{method of stationary phase} to estimate the
asymptotics of the integral in $t,\sigma.$  If $\hat\rho$ is chosen
supported sufficiently close to the origin, then the unique stationary
point on the support of the amplitude is at $\sigma=1,$ $t=0;$ we thus
obtain a complete asymptotic expansion in $\lambda$ beginning with the
terms
$$
A \lambda^{n-1}+O(\lambda^{n-2})
$$
where $$A= n(2\pi)^{-n}  \Vol(B^* X).$$  
\begin{Exercise}
Do this stationary phase computation.  If you don't know about the
method of stationary phase, this is your chance to learn it, e.g.\
from \cite{Hormander:v1}.
\end{Exercise}

Thus, since $u-U \in \CI((-\ep,\ep) \times \RR^n),$ 
\eqref{ftofcounting} yields
\begin{proposition}\label{proposition:averagedweyl}
$$(\rho* N') (\lambda) \sim A \lambda^{n-1}+O(\lambda^{n-2}).$$
\end{proposition}

We now try to make a ``Tauberian'' argument to extract the desired
asymptotics of $N(\lambda)$ from this estimate.
\begin{lemma}
$$N(\lambda+1)-N(\lambda) =O(\lambda^{n-1}).$$
\end{lemma}
\begin{proof}
By Proposition~\ref{proposition:averagedweyl} and since
$N'(\lambda)=\sum \delta(\lambda-\lambda_j),$  we have
$$
\sum \rho(\lambda-\lambda_j) \sim A \lambda^{n-1}+O(\lambda^{n-2});
$$
thus, by positivity of $\rho(\lambda),$
$$
(\inf_{[-1,1]} \rho) \left(\#\{\lambda_j: \lambda-1<\lambda_j<\lambda+1\}\right)
\leq \sum \rho(\lambda-\lambda_j) =O(\lambda^{n-1}),
$$
and the estimate follows as the infimum is strictly positive.
\end{proof}
This yields at least a crude estimate:
\begin{corollary}
$$N(\lambda) = O(\lambda^n).$$
\end{corollary}
A more technically useful result is:
\begin{corollary}\label{cor:technical}
$$N(\lambda-\tau)-N(\lambda) \lesssim \ang{\tau}^n\ang{\lambda}^{n-1}.$$
\end{corollary}
\begin{exercise}
Prove the corollaries.  (For the latter, begin with the intermediate
estimate $\ang{\tau}\ang{\abs{\lambda}+\abs{\tau}}^{n-1}$.)
\end{exercise}

Now we work harder.

\begin{exercise}
Show that we can antidifferentiate the convolution to get
$$
\int_{-\infty}^\lambda (\rho*N')(\mu)\, d\mu =
(\rho*N)(\lambda).
$$
\end{exercise}

As a result, we of course have
$$
(\rho*N)(\lambda)=A\lambda^n/n+O(\lambda^{n-1}) = B \lambda^n+O(\lambda^{n-1})
$$
where $B=A/n= (2\pi)^{-n} \Vol(B^*X).$

Thus, since $\int \rho(\mu) \, d\mu=1,$
\begin{align*}
N(\lambda) &= (N*\rho)(\lambda) -\int (N(\lambda-\mu)-N(\lambda) )\rho(\mu) \, d\mu \\
&= B \lambda^n +O(\lambda^{n-1}) - \int O(\ang{\mu}^n\ang{\lambda}^{n-1}) \rho(\mu)
\, d\mu \\
&= B \lambda^n +O(\lambda^{n-1}),
\end{align*}
where we have used Corollary~\ref{cor:technical} in the penultimate
equality.  We record what we have now obtained as a theorem, better known as Weyl's law
with remainder term.  This
form of the remainder term is sharp, and not so easy to obtain by
other means.
\begin{theorem}
$$N(\lambda) = (2\pi)^{-n} \Vol(B^*X) \lambda^n + O(\lambda^{n-1}).$$
\end{theorem}
As noted above, it is perhaps suggestive to view the main term as the volume of the
sublevel set in phase space $\{ (x,\xi): \sigma(\Lap)(x,\xi) \leq
\lambda^2\}.$  Weyl's law is one of the most beautiful instances of
the quantum-classical correspondence, in which we can deduce something
about a quantum quantity (the counting function for eigenvalues, also
known as energy levels) in terms of a classical quantity, in this case
the volume of a region of phase space.

\begin{Exercise}
Show that the error term in Weyl's law is sharp on spheres.
\end{Exercise}

\section{Lagrangian distributions}

The form of the parametrix that we used for the wave equation turns
out to be a special case of a very general and powerful class of
distributions, known as \emph{Lagrangian distributions}, introduced by
H\"ormander.  Here we will give a very sketchy introduction to the
general theory of Lagrangian distributions, and see both how it
systematizes and extends our parametrix construction for the wave
equation and how (in principle, at least) it can be made to
yield the Duistermaat-Guillemin trace formula, which gives us
an explicit description of the singularities of the wave trace.

We begin with a special case of the theory.

\subsection{Conormal distributions}

Let $X$ be a smooth manifold of dimension $n$ and let $Y$ be a
submanifold of codimension $k.$  The \emph{conormal distributions}
with respect to $Y$ are a special class of distributions having wavefront
set\footnote{Recall that we have defined the wavefront set to lie in $S^*X$ but it
  is often convenient to regard it as a \emph{conic} subset of
  $T^*X\backslash o,$ with $o$ denoting the zero-section.}
in the conormal bundle of $Y,$ $N^*Y.$  Let us suppose that $Y$ is
locally cut out by defining functions $\rho_1,\dots, \rho_k\in \CI(X),$ i.e.\
that (at least locally), $\{\rho_1=\dots =\rho_k=0\}=Y,$ and
$d\rho_1,\dots, d\rho_k$ are linearly independent on $Y.$  Then we may
(locally) extend the $\rho_j$'s to a complete coordinate system
$$(x_1,\dots x_k,y_1,\dots y_{n-k})$$ with
$$
x_1=\rho_1,\dots,x_k=\rho_k,
$$
so that $Y=\{x=0\}.$  In these coordinates, how might we write
down some distributions with wavefront set lying only in $N^*Y$?
Well, we can try to make things that are singular in the $x$ variables
at $x=0,$ with the $y$'s behaving like smooth parameters.  How do we
create singularities at $x=0$?  One very nice answer is in the
following:

\begin{lemma}
Let $a(\xi) \in S^m_\cl(\RR^k_\xi)$ for some $m.$  Then $\WF
\mathcal{F}^{-1}(a) \subseteq N^*(\{0\}).$
\end{lemma}
\begin{proof}
Writing
$$
\mathcal{F}^{-1}(a)(x) = (2\pi)^{-k/2} \int a(\xi)e^{i\xi\cdot x} \,
d\xi,
$$
we first note that
$$
\mathcal{F}^{-1}(a)(x) \in H^{-m-k/2-\ep}(\RR^k)
$$
for any $a \in S^m_\cl$ and for all $\ep>0.$  Moreover for all $j,$
\begin{align*}
(x^i D_{x^j}) \mathcal{F}^{-1}(a)(x) &= (2\pi)^{-k/2} \int a(\xi)(x^i D_{x^j})e^{i\xi\cdot x} \,
dx\\  &= (2\pi)^{-k/2} \int x^i \xi_j a(\xi)e^{i\xi\cdot x} \, d\xi \\  
&= (2\pi)^{-k/2} \int  \xi_j a(\xi) D_{\xi_i} e^{i\xi\cdot x} \, d\xi\\
&= -(2\pi)^{-k/2} \int  D_{\xi_i} (\xi_j a(\xi)) e^{i\xi\cdot x} \, d\xi,
\end{align*}
where we have integrated by parts in the final line.
Note that if $a \in S^m_\cl$ then $D_{\xi_i} (\xi_j a(\xi)) \in S^m_\cl$
too (cf.\ Exercise~\ref{KN}).  Thus we also have
$$
(x^i D_{x^j}) \mathcal{F}^{-1}(a)(x) \in H^{-m-k/2-\ep}(\RR^k).
$$
Iterating this argument gives
\begin{equation}\label{itreg}
(x_{i_1} D_{x_{j_1}}) \dots (x_{i_l} D_{x_{j_l}})
    \mathcal{F}^{-1}(a)(x) \in H^{-m-k/2-\ep}(\RR^k).
\end{equation}
for all choices of indices and all $l \in \NN.$ Thus
$\mathcal{F}^{-1}a$
is smooth\footnote{We are of course proving more than the lemma
  states here: \eqref{itreg} gives a more precise ``conormality''
  estimate that is valid uniformly across the origin.}
  away from $x=0.$
\end{proof}

By the same token, we have more generally,
\begin{proposition}
Let $\rho_1,\dots, \rho_k$ be (local) defining functions for $Y
\subset X$ and let \begin{equation}\label{conormalsymbol} a \in S^{m+(n-2k)/4}_\cl(\RR^n_x \times \RR^k_\xi)\end{equation} be
compactly supported in $x.$  Then
\begin{equation}\label{conormaldist}
u(x) = (2\pi)^{-(n+2k)/4}
\int_{\RR^k} a(x,\theta) e^{i (\rho_1 \theta_1+\dots +\rho_k \theta_k)} \, d\theta
\end{equation}
has wavefront set contained in $N^* Y.$  Moreover there exists $s \in
\RR$ such that if $V_1,\dots V_l$
are vector fields tangent to $Y,$ then
$$
V_1\dots V_l u \in H^s.
$$
\end{proposition}
\begin{exercise}
Prove the proposition.  You will probably find it helpful to change to
a coordinate system $(x_1,\dots,x_k, y_1,\dots,y_{n-k})$ in which $x_1,\dots,x_k=\rho_1,\dots,\rho_k.$
Note that in this coordinate system, any vector field tangent to $Y$
can be written
$$
\sum a_{ij}(x,y) x^i \pa_{x^j}+ \sum b_j(x,y) \pa_{y^j}.
$$

What values of $s,$ the Sobolev exponent in the proposition, are allowable?
\end{exercise}

\begin{definition}
  A distribution $u \in \mathcal{D}'(X)$ is a \emph{conormal
    distribution} with respect to $Y,$ of order $m,$ if it can
  (locally) be written in the form \eqref{conormaldist} with symbol as
  in \eqref{conormalsymbol}.
\end{definition}
While it may appear that the definition of conormal distributions depends on the choice of the
defining functions $\rho_j,$ this is in fact not the case.  The rather
peculiar-looking convention on the orders of distributions is not
supposed to make much sense just yet.

Note that examples of conormal distributions include $\delta(x) \in
\RR^n$ (conormal with respect to the origin), and more generally,
delta distributions along submanifolds.  Also quite pertinent is the
example of pseudodifferential operators: if $A=\Opl(a)\in \Psi^m(X)$
then the Schwartz kernel of $A$ is a conormal distribution with
respect to the diagonal in $X\times X,$ of order $m.$ (This goes at
least some of the way to explaining the convention on orders.)
Indeed, we could (at some pedagogical cost) simply have introduced
conormal distributions and then used the notion to define the Schwartz
kernels of pseudodifferential operators in the first place.

\subsection{Lagrangian distributions}

We now introduce a powerful generalization of conormal distributions,
the class of \emph{Lagrangian distributions.}\footnote{These were
  first studied by H\"ormander \cite{Hormander:FIO1}.}  We begin by introducing
some underlying geometric notions.

An important notion from symplectic geometry is that of a
\emph{Lagrangian submanifold} $\lag$ of a symplectic manifold
$N^{2n}.$ This is a submanifold of dimension $n$ on which the symplectic
form vanishes.  We can always find local coordinates in which the
symplectic form is given by $\omega=\sum dx^i \wedge dy^i$ and
$\lag=\{y=0\},$ so there are no interesting local invariants of
Lagrangian manifolds.

A \emph{conic Lagrangian manifold} in $T^*X$ is a Lagrangian
submanifold of $T^*X\backslash o$ that is invariant under the $\RR^+$
action on the fibers.  (Here, $o$ denotes the zero-section.)

Among the most important examples of conic Lagrangians are the
following: let $Y\subset X$ be any submanifold; then $N^*Y \subset
T^*X$ is a conic Lagrangian.
\begin{exercise}
Verify this.
\end{exercise}

The trick to defining Lagrangian distributions is to figure out how to
associate a \emph{phase function} $\phi$ with a conic Lagrangian
$\lag$ in $T^*X.$
\begin{definition}
A nondegenerate phase function is a smooth function $\phi(x,\theta),$
locally defined on a coordinate neighborhood of $X\times\RR^k,$
such that $\phi$ is homogeneous of degree $1$ in $\theta$ and
such that the differentials $d(\partial{\phi}/\partial \theta_j)$ are
  linearly independent on the set
$$
C=\left\{(x,\theta):\frac{\partial{\phi}}{\partial \theta_j}=0 \text{ for all }
  j=1,\dots, k\right\}.
$$
The phase function is said to locally \emph{parametrize} the conic Lagrangian $\lag$
if
$$
C \ni (x,\theta) \mapsto (x,d_x\phi)
$$
is a local diffeomorphism from $C$ to $\lag.$
\end{definition}
\begin{exercise}\
\begin{enumerate}
\item Show that, in the notation of the definition above, $C$ is
  automatically a manifold, and the map $C \ni (x,\theta) \mapsto (x,d_x\phi)$ is
  automatically a local diffeomorphism from $C$ to its image, which is
  a conic Lagrangian.
\item
Show that if $\rho_j$ are definining functions for $Y\subset X$ then
$$
\phi=\sum \rho_j \theta_j
$$
is a nondegenerate parametrization of $N^*Y.$
\item
What Lagrangian is parametrized by the phase function used in our
parametrix for the half-wave operator in the Euclidean case, given by
$$
\phi(t,x,y,\theta) = (x-y)\cdot \theta-t\abs{\theta}?
$$
\end{enumerate}
\end{exercise}

It turns out that every conic Lagrangian manifold has a local
parametrization; the trouble is, in fact, that it has lots of them.

\begin{definition}
A \emph{Lagrangian distribution} of order $m$ with
respect to the Lagrangian $\lag$ as one that is given,
locally near any point in $X,$ by a finite sum of oscillatory
integrals of the form
$$
(2\pi)^{-(n+2k)/4} \int_{\RR^k} a(x,\theta) e^{i\phi(x,\theta)} \, d\theta
$$
where
$$
a \in S_\cl^{m+(n-2k)/4}(\RR^n_x \times\RR^k_\theta)
$$
and where $\phi$ is a nondegenerate phase function parametrizing
$\lag.$  Let $I^m(X,\lag)$ denote the space of all Lagrangian
distributions on $X$ with respect to $\lag$ of order $m.$
\end{definition}

Note that the connection between $k,$ the number of phase variables,
and the geometry of $\lag$ is not obvious; indeed, it turns out that
we have some choice in how many phase variables to use.  As there are
many different ways to parametrize a given conic Lagrangian manifold,
one tricky aspect of the theory of Lagrangian distributions
is necessarily the proof that using different parametrizations
(possibly involving different numbers of phase variables) gives
us the same class of distributions.

The analogue of the \emph{iterated regularity} property of conormal
distributions, i.e.\ our ability to repeatedly differentiate along
vector fields tangent to $Y,$ turns out to be as follows:
\begin{proposition}
Let $u \in I^m(X,\lag).$  There exists $s$ such that for any $l \in \NN$ and for any $A_1,\dots,
A_l \in \Psi^1(X)$ with $\sigma_1(A_j)|_{\lag}=0,$ we have
$$
A_1 \dots A_l u \in H^s(X).
$$
\end{proposition}
Of course, once this holds for one $s,$ it holds for all smaller
values; the precise range of possible values of $s$ is related to the
order $m$ of the Lagrangian distribution; we will not pursue this
relationship here, however.  This iterated regularity property of Lagrangian
distributions completely characterizes them if we use ``Kohn-Nirenberg''
symbols (as in Exercise~\ref{KN}) instead of ``classical'' ones 
(see \cite{Hormander:v4}).

\subsection{Fourier integral operators}

Fourier integral operators (``FIO's'') quantize classical \emph{maps}
from a phase space to itself just as pseudodifferential operators quantize
classical \emph{observables} (i.e.\ functions on the phase space).
The maps from phase space to itself that we may quantize in this
manner are the \emph{symplectomorphisms}, exactly the class of
transformations of phase space that arise in classical mechanics.  We
recall that a symplectomorphism between symplectic manifolds is a
diffeomorphism that preserves the symplectic form.  We further define
a \emph{homogeneous symplectomorphism} from $T^*X$ to $T^*X$ to be one
that is homogeneous in the fiber variables, i.e.\ commutes with the
$\RR^+$ action on the fibers.

An important class of homogeneous symplectomorphisms is those
obtained as follows:
\begin{exercise}
Show that the time-$1$ flowout of the Hamilton vector field of a
homogeneous function of degree $1$ on $T^*X$ is a homogeneous symplectomorphism.
\end{exercise}

Given a homogeneous symplectomorphism $\Phi:T^*X \to T^*X,$ consider its \emph{graph}
$\Gamma_\Phi\subset (T^*X \backslash o)\times (T^*X \backslash o).$
Since $\Phi$ is a symplectomorphism, we have
$$
\iota^* \pi_L^* \omega=\iota^* \pi_R^* \omega,
$$
where $\iota$ is inclusion of $\Gamma_\Phi$ in $(T^*X \backslash
o)\times (T^*X \backslash o),$ and $\pi_\bullet$ are the left and
right projections.  If we alter $\Gamma_\Phi$ slightly, forming
$$
\Gamma'_\Phi = \{(x_1,\xi_1, x_2,\xi_2): (x_1,\xi_1, x_2,-\xi_2) \in \Gamma_\Phi\},
$$
and let $\iota'$ denote the inclusion of this manifold, then we find
that a sign is flipped, and
$$
(\iota')^* \pi_L^* \omega+(\iota')^* \pi_R^* \omega=0;
$$
since $\Omega = ( \pi_L^* \omega+\pi_R^* \omega)$ is just
the symplectic form on 
$$
T^*(X\times X) =T^*X \times T^*X,
$$
we thus find that 
$\Gamma'_\Phi$ is \emph{Lagrangian} in $T^*(X\times X).$  In fact, it is easily to verify
that given a diffeomorphism $\Phi,$ $\Gamma'_\Phi$ is Lagrangian if and only
if $\Phi$ is a symplectomorphism.
\begin{exercise}
Check this.
\end{exercise}

Now we simply define the class of Fourier integral operators (of order
$m$) associated with
the symplectomorphism $\Phi$ of $X$ to be those operators
from smooth functions to distributions whose Schwartz kernels lie in
the Lagrangian distributions
$$
I^m(X\times X, \Gamma'_\Phi).
$$
It would be nice if this class of operators turned out to have good
properties such as behaving well under composition, as
pseudodifferential operators certainly do.  We note right off the bat
that these operators \emph{include} pseudodifferential operators, as well as
a number of other, familiar examples:

\begin{enumerate}
\item $\Psi^m(X) = I^m(X\times X, \Gamma'_{\Id}).$
\item In $\RR^n,$ fix $\alpha$ and let $Tf(x)=f(x-\alpha)$  Then $T$
  has Schwartz kernel
$$
\delta(x-x'-\alpha)
$$
which is clearly conormal of order zero at $x-x'-\alpha=0.$  Note that
this is certainly \emph{not} a pseudodifferential operator, as it moves
wavefront around; indeed, it is associated with the symplectomorphism
$\Phi(x,\xi)=(x+\alpha,\xi),$ and it it no coincidence that
$$
\WF Tf = \Phi(\WF f).
$$
\item
As a generalization of the previous example, note that if $\phi: X\to
X$ is a diffeomorphism, then we may set
$$
Tf(x)=f(\phi(x));
$$
this is a FIO associated to the homogeneous symplectomorphism
$$
\Phi(x,\xi) =(\phi^{-1}(x), \phi^*_{\phi^{-1}(x)}(\xi))
$$
induced by $\phi$ on $T^*X.$
\end{enumerate}
\begin{exercise}
Work out this last example carefully.
\end{exercise}

Now it turns out to be helpful to actually consider a broader class
of FIO's than we have described so far.  Instead of just using
Lagrangian submanifolds of $T^*(X\times X)$ given by
$\Gamma'=\Gamma'_\Phi$ where $\Phi$ is a symplectomorphism, we just
require that $\Gamma'$ be a reasonable Lagrangian (and we allow
operators between different manifolds while we are at it):
\begin{definition}
Let $X,Y$ be two manifolds (not necessarily of the same dimension).  A
\emph{homogeneous canonical relation} from $T^*Y$ to $T^*X$ is a
homogeneous submanifold $\Gamma$ of $(T^*X\backslash o)
\times(T^*Y\backslash o),$ closed in $T^*(X\times Y) \backslash o$ such that
$$
\Gamma' \equiv \{(x,\xi,y,\eta):(x,\xi,y,-\eta)\in \Gamma\}
$$
is Lagrangian in $T^*(X\times Y).$
\end{definition}
We can view $\Gamma$ as giving a multivalued generalization of a
symplectomorphism, with
$$
\Gamma(y,\eta) \equiv \{(x,\xi): (x,\xi,y,\eta) \in \Gamma\}.
$$
and, more generally, if $S \subset T^*Y$ is conic,
\begin{equation}\label{canrel}
\Gamma(S) \equiv \{(x,\xi): \text{ there exists } (y,\eta) \in S,
\text{ with }(x,\xi,y,\eta) \in \Gamma\}.
\end{equation}

\begin{definition}
A Fourier integral operator of order $m$ associated to a homogeneous
canonical relation $\Gamma$ is an
operator from $\CcI(Y)$ to $\mathcal{D}'(X)$ with Schwartz kernel in
$$
I^m(X\times Y, \Gamma').
$$
\end{definition}

\begin{exercise}
Show that a homogeneous canonical relation $\Gamma$ is associated to a
symplectomorphism if and only if its projections onto both factors
$T^*X$ and $T^*Y$ are diffeomorphisms.
\end{exercise}

\begin{exercise}\label{exercise:restrictionfio}\
\begin{enumerate}
\item
Let $Y \subset X$ be a submanifold.  Show that the operation of
restriction of a smooth function on $X$ to $Y$ is an FIO.
\item Endow $X$ with a metric, and consider the volume form $dg_Y$ on
  $Y$ arising from the restriction of this metric; show that the map
  taking a function $f$ on $Y$ to the distribution $\phi \mapsto
  \int_Y \phi\rvert_Y (y) f(y) dg_Y$ is an FIO.  (Think of it as just
  multiplying $f$ by the delta-distribution along $Y,$ which makes
  sense if we choose a metric.)  What is the relationship between the
  restriction FIO and this one, which you might think of as an
  extension map?
\end{enumerate}
\end{exercise}

In the special case that $\Gamma$ is a canonical relation that is
locally the graph of a symplectomorphism, we say it is a \emph{local
  canonical graph}.

We now briefly enumerate the properties of the FIO calculus, somewhat
in parallel with our discussion of pseudodifferential operators.
These theorems are considerably deeper, however.  In preparation for
our discussion of composition, suppose that
\begin{align*}
\Gamma_1\subset T^*X\backslash o &\times T^*Y\backslash o,\\
\Gamma_2\subset T^*Y\backslash o &\times T^*Z\backslash o
\end{align*}
are homogeneous canonical relations.  We say that $\Gamma_1$ and $\Gamma_2$ are
\emph{transverse} if the manifolds
$$
\Gamma_1\times \Gamma_2\  \text{ and } T^*X \times \Delta_{T^*Y} \times T^*Z
$$
intersect transversely in $T^*X\times T^*Y \times T^*Y \times T^*Z;$
here $\Delta_{T^*Y}$ denotes the diagonal submanifold.
\begin{exercise}
Show that if either $\Gamma_1$ or $\Gamma_2$ is the graph of
a symplectomorphism, then $\Gamma_1$ and $\Gamma_2$ are transverse.

\end{exercise}

In what follows, we will as usual assume for simplicity that all manifolds are
compact.\footnote{In the absence of this assumption, we need as usual to add
  various hypotheses of properness.}  In the following list of
properties, some are special to FIO's, that is to say, Lagrangian
distributions on product manifolds, viewed as operators; others are
more generally properties of Lagrangian distributions per se, hence
their statements do not necessarily involve products of manifolds.  In
the interests of brevity, we focus on the deeper properties, and omit
trivialities such as associativity of composition.
Note also that for brevity we will systematically confuse operators with
their Schwartz kernels.
\renewcommand{\theenumi}{\Roman{enumi}}
\begin{enumerate}
\item\label{fioproperty:alg} (Algebra property) If $S \in I^m(X\times
  Y, \Gamma'_1)$ and $T \in I^{m'} (Y \times Z, \Gamma'_2)$ and
  $\Gamma_1$ and $\Gamma_2$ are transverse, then
$$
S\circ T \in I^{m+m'} (X \times Z, (\Gamma_1 \circ \Gamma_2)'),
$$
where
\begin{multline}\label{relations}
\Gamma_1 \circ \Gamma_2 = \{(x,\xi,z,\zeta): (x,\xi,y,\eta) \in
\Gamma_1 \\ \text{ and } (y,\eta,z,\zeta) \in \Gamma_2 \text{ for some } (y,\eta)\}.
\end{multline}
Moreover,
$$
S^* \in I^m(Y \times X, (\Gamma^{-1})')
$$
where $\Gamma^{-1}$ is obtained from $\Gamma$ by switching factors.
\item\label{fioproperty:smoothing} (Characterization of smoothing
  operators) The distributions
  in $I^{-\infty}(X,\lag)$ are exactly those in $\CI(X);$ composition
  of an operator $S \in I^m(X\times Y, \Gamma')$ on either side with a
  smoothing operator (i.e.\ one with smooth Schwartz kernel) yields a smoothing operator.
\item\label{fioproperty:symb} (Principal symbol homomorphism) There is family
of linear ``principal symbol maps'' \begin{equation}\label{fiosymbol}\sigma_m:
  I^m(X, \lag) \to \frac{S_\cl^{m+(\dim X)/4}(\lag; L)}{S_\cl^{m-1+(\dim X)/4}(\lag; L)}.\end{equation}
  Here $L$ is a
certain canonically defined line bundle on $\lag$ (see the commentary below),
and $S_\cl^m(\lag; L)$ denotes $L$-valued symbols.
We may identify the quotient space in \eqref{fiosymbol} with
$$
\CI(S^* \lag; L),
$$
and we call the resulting map $\hsigma_m$ instead.
If $S,$
  $T,$ are as in \eqref{fioproperty:alg}, with canonical relations
  $\Gamma_1,\Gamma_2$ intersecting transversely,
$$
\sigma_{m+m'}(ST) = \sigma_m(S) \sigma_{m'}(T)
$$
and
$$
\sigma_m(A^*) = s^*\overline{\sigma_m(A)},
$$
where $s$ is the map interchanging the two factors.
The product of the symbols, at $(x,\xi,z,\zeta) \in \Gamma_1\circ
\Gamma_2,$ is defined as
$$
\sigma_m(S)(x,\xi,y,\eta)\cdot  \sigma_{m'}(T)(y,\eta,z,\zeta)
$$
evaluated at (the unique) $(y,\eta)$ such that
$(x,\xi,y,\eta) \in \Gamma_1,$ $(y,\eta,z,\zeta)\in \Gamma_2.$
\item\label{fioproperty:exact} (Symbol exact sequence) There is a short exact sequence
$$
0 \to I^{m-1}(X, \lag)  \to I^m(X, \lag)
\stackrel{\hsigma_m}{\to} \CI(S^* \lag; L) \to 0.
$$
hence the symbol is $0$ if and only if an operator is of lower order.
\item\label{fioproperty:quant}
Given $\lag,$ there is a linear ``quantization map'' $$\Op: S^{m+(\dim X)/4}_\cl(\lag;L) \to
I^m(X, \lag)$$ such that
if $$a \sim \sum_{j=0}^\infty a_{m+(\dim X)/4-j}(x,\hat\xi)
\abs{\xi}^{m+(\dim X)/4-j} \in
S^{m+(\dim X)/4}_\cl(\lag;L)$$ then
$$
\sigma_m(\Op(a)) = a_{m+(\dim X)/4}(x,\hat\xi).
$$  The map $\Op$ is onto, modulo $\CI(X).$
\item\label{fioproperty:commutator} (Product with vanishing principal symbol)
If $P\in \Diff^m (X)$ is self-adjoint and $u\in I^{m'} (X, \lag),$ with $\lag
\subset\Sigma_P\equiv \{\sigma_m(P)=0\},$ then
$$
P u \in I^{m+m'-1}(X,\lag)
$$
and
$$
\sigma_{m+m'-1} (Pu) = i^{-1} \hamvf_p (\sigma_{m'}(u)),
$$
with $\hamvf_p$ denoting the Hamilton vector field.
\item\label{fioproperty:boundedness} ($L^2$-boundedness, compactness)
If $T\in I^m(X\times Y, \Gamma)$ is associated to a local canonical
graph, then
$$
T \in \mathcal{L}(H^s(Y), H^{s-m}(X)) \text{ for all }s\in \RR.
$$
Negative-order operators of this type acting on $L^2(X)$ are thus compact.
\item\label{fioproperty:summation} (Asymptotic summation)
Given $u_j \in I^{m-j}(X,\lag),$ with $j \in \NN,$ there exists $u \in
I^m(X,\lag)$ such that
$$
u\sim \sum_j u_j,
$$
which means that
$$
u - \sum_{j=0}^N u_j \in I^{m-N-1}(X,\lag)
$$
for each $N.$

\item\label{fioproperty:microsupport} (Microsupport)
The microsupport of $T \in I^{m} (X\times Y, \Gamma')$ is well
defined as the largest conic subset $\tilde\Gamma\subset \Gamma$
on which the symbol is $O(\abs{\xi}^{-\infty}).$  We have
$$
\WF T u \subseteq \tilde \Gamma  (\WF u)
$$
for any distribution $u$ on $Y,$ where the action of $\tilde \Gamma$
on $\WF u$ is given by \eqref{canrel}. Furthermore,
$$
\WF' (S\circ T) \subseteq \WF'S \circ \WF' T.
$$
\end{enumerate}

\noindent\textsc{Commentary:}

\begin{enumerate}
\item[\eqref{fioproperty:alg}] This is a major result.  Since FIO's
  include pseudodifferential operators, this includes the composition
  property for pseudodifferential operators as a special case.
  Another special case, when $Z$ a point, yields the statement that an
  FIO applied to a Lagrangian distribution on the manifold $Y$ with
  respect to the Lagrangian $\lag\subset T^*Y $ is a Lagrangian
  distribution associated to $\Gamma(\lag),$ where $\Gamma$ is the
  canonical relation of the FIO and $\Gamma(\lag)$ is defined by
  \eqref{canrel}.

  One remarkable corollary of this result is as follows: As will be
  discussed below, what our parametrix construction in
  \S\ref{section:waveparametrix} really showed was that for $t$
  sufficiently small, and fixed, we have
$$
e^{-it\sqrt{\Lap}} \in I^0(X \times X, \lag_t)
$$
where $\lag_t$ is the backwards geodesic flowout, for time $t,$ in the left
factor of $N^*\Delta,$ of the conormal bundle to the diagonal in
$T^*(X\times X).$
\begin{Exercise}
Verify this assertion!  (Try this now, but fear not: we will discuss this example further in
\S\ref{section:wavetraceredux} and you can try again then.)
\end{Exercise}
Now $e^{-it\sqrt{\Lap}}$ is a one-parameter group and so the
composition property for FIO's allows us to conclude that in fact
$e^{-it\sqrt{\Lap}}$ is an FIO for \emph{all} times $t,$ associated to
the same flowout described above.  The interesting subtlety is that
while $\lag_t$ is an inward- or outward-pointing conormal bundle for
small positive resp.\ negative time (i.e.\ in the regime where our
parametrix construction worked directly), for $t$ exceeding the
injectivity radius, it ceases to be a conormal bundle, while remaining
a smooth Lagrangian manifold in $T^*(X\times X).$
\item[\eqref{fioproperty:symb}]
Modulo bundle factors, the principal symbol is defined as follows:
if $u \in I^m (X,\lag)$ is given by
$$
u = (2\pi)^{-(n+2k)/4} \int_{\RR^k} a(x,\theta)e^{i\phi(x,\theta)} \, d\theta,
$$
then $\sigma_m(u)$ is defined by first restricting $a(x,\theta)$ to the
manifold
$$
C=\{(x,\theta): d_\theta\phi=0\};
$$
as $\phi$ is a nondegenerate phase function, this manifold is locally
diffeomorphic (via a homogeneous diffeomorphism) to $\lag,$ hence we
may identify $a|_{C}$ with a function on $\lag;$ transferring this
function to $\lag$ via the local diffeomorphism and taking the
top-order homogeneous term in the asymptotic expansion gives the
principal symbol.

Much has been swept under the rug here---for a proper discussion, see,
e.g., \cite{Hormander:FIO1}.  In particular, the line bundle $L$ contains
not just the density factors that we have been studiously ignoring---the
Schwartz kernel of an operator from functions to functions on $X$ is actually
a ``right-density'' on $X \times X,$ i.e.\ a section of the pullback
of the bundle $\abs{\Omega^n(X)}$ in the right factor---but also the
celebrated ``Keller-Maslov index,'' which is related to the
indeterminacy in choosing the phase function parametrizing the
Lagrangian.  We will not enter into a serious discussion of these
issues here.  We have also omitted discussion of the geometry of
composing canonical relations, and the fact that transverse canonical
relations compose to give a new canonical relation, with a unique
point $y,\eta$ such that $(x,\xi,y,\eta) \in \Gamma_1,$
$(y,\eta,z,\zeta) \in \Gamma_2$ whenever $(x,\xi,z,\zeta) \in
\Gamma_1\circ \Gamma_2.$

\item[\eqref{fioproperty:commutator}]
There is a more general version of this statement valid for any $P \in
\Psi^{m}(X)$ characteristic on $\lag,$ but it involves the notion of
subprincipal symbol, which requires some explanation; see
\cite[\S5.2--5.3]{Duistermaat-Hormander:FIO2}.  Moreover, if we are a little more
honest about making this computation work invariantly, so that the
symbol has a density factor in it (one factor in the line bundle $L,$)
then we should really write
$$
\sigma_{m+m'-1} (Pu) = i^{-1} \mathcal{L}_{\hamvf_p} \sigma_{m'}(u),
$$
where $\mathcal{L}_Z$ denotes the Lie derivative along the vector field $Z.$
\item[\eqref{fioproperty:boundedness}]
This is fairly easy to prove, as if $T$ of order $m$ is associated to a
symplectomorphism from $Y$ to $X$, it is easy to check from the previous properties
that $T^*T$ is an FIO associated with the canonical relation given by
the identity map, and hence
$$
T^*T \in \Psi^{2m}(Y),
$$
and we may invoke boundedness results for the pseudodifferential
calculus.  In cases when $T$ is not associated to a local canonical
graph, this argument fails badly (i.e.\ interestingly), and the
optimal mapping properties are a subject of ongoing research.
\end{enumerate}
\renewcommand{\theenumi}{\arabic{enumi}}

Finally, as with the pseudodifferential calculus, we may define a
notion of ellipticity for FIO's, and the above properties imply that
(microlocal) parametrices exist for the inverses of elliptic operators
associated to symplectomorphisms.

\section{The wave trace, redux}\label{section:wavetraceredux}

Let us briefly revisit our construction of the parametrix for the half-wave
equation in the light of the FIO calculus.  Here is what we did, in
hindsight:  we sought a distribution
$$
u \in I^m(\RR\times X \times X, \lag)
$$
for some Lagrangian $\lag,$ and some order $m,$ with
$$
u(0,x,y)=\delta(x-y)
$$
such that
$$
(D_t +  \sqrt{\Lap}_x ) u \in I^{-\infty} ((-\ep,\ep) \times X \times X,\lag)= \CI((-\ep,\ep) \times X \times X).
$$

We begin by sorting out what $m,$ the order of $u,$ should be.
Since
$$
u |_{t=0}=\delta(x-y)=
(2\pi)^{-n} \int_{\RR^n} e^{i(x-y) \cdot \theta} \, d\theta,
$$
we were led us to a solution that for $t$ small was of the form
$$
\int_{\RR^n} a(t,x,y,\theta) e^{i\Phi(t,x,y,\theta)} \, d\theta
$$
with $a$ a symbol of order zero such that $a(0,x,y,\theta)=1,$ and
$\Phi$ a nondegenerate phase function such that
$\Phi(0,x,y,\theta)=(x-y)\cdot \theta.$ This was certainly the rough
form of our earlier Ansatz; it should now be regarded as a Lagrangian
distribution, of course.  Since $\dim (\RR\times X \times X) = 2n+1$
and we have $n$ phase variables $\theta_1,\dots,\theta_n,$ the
convention on orders of FIO's leads to $m=-1/4.$

Now we address the following question: what Lagrangian $\lag$ ought we
to choose?  Since $$\Box_{t,x} \in \Diff^2(\RR\times X \times
X)\subset \Psi^2(\RR\times X \times X),$$ we a priori would have
$$
\Box u  \in I^{7/4}(\RR\times X\times X,\lag);
$$
as we would like smoothness of $\Box u,$ we ought to start by making
the principal symbol of $\Box u$ vanish. The symbol of $\Box$ vanishes only on
$$\Sigma_\Box = \{\tau^2=\abs{\xi}_g^2\}$$
hence the easiest way to ensure vanishing of the principal symbol is
simply to arrange that
\begin{equation}\label{laginchar}
\lag \subset \Sigma_{\Box}.
\end{equation}
Now, recall that our initial conditions were to be
$$
u(0,x,y)=\delta(x-y),
$$
where we may view this as a Lagrangian distribution on $X \times X$ with respect to
$N^*\Delta,$ the conormal to the diagonal:
$$
N^*\Delta =\{(x,y,\xi,\eta): x=y,\xi=-\eta\}.
$$  It is not difficult to
check that the requirement that $u|_{t=0}$ gives this
lower-dimensional Lagrangian\footnote{We really ought to think a bit
  about restriction of Lagrangian distributions here: this is best
  done by regarding the restriction operator itself as an FIO (cf.\
  Exercise~\ref{exercise:restrictionfio}).  We shall omit further
  discussion of this point, but remark that it should at least seem
  plausible that the Lagrangian manifold associated to the restriction
  is the projection (i.e.\ pullback under inclusion), of the Lagrangian
  in the ambient space---cf.\ Exercise~\ref{exercise:WFrestriction}.}
together with the requirement \eqref{laginchar} that $\lag$ should lie
in the characteristic set implies that $\lag\cap\{t=0\}$ should just
consist of points in $\Sigma_{\Box}$ \emph{projecting} to points
in $N^*\Delta,$ i.e.\ that we should in fact have
$$
\lag \cap \{t=0\} = \{(t=0, \tau=-\abs{\eta}_g, x=y,
\xi=-\eta)\}\subset T^*(\RR\times X\times X).
$$
Here we have chosen the sign $\tau=-\abs{\eta}_g$ in view of our real
interest, which is in solving
$$
(D_t +\sqrt\Lap)u=0
$$
rather than the full wave equation;\footnote{We have chosen to
  emphasize this distinction only at this critical juncture only
  because as it is in some respects more pleasant to deal with $\Box$
  than with the half-wave operator when possible.} we have thus kept
$\lag$ inside the characteristic set of $D_t+\sqrt\Lap,$ which is one
of the two components of $\Sigma_\Box.$ 

Let $\lag_0$ now
denote $\lag\cap\{t=0\}.$ The set $\lag_0$ is a manifold on which the
symplectic form vanishes (an ``isotropic'' manifold), of dimension one
less than half the dimension of $T^*(\RR\times X \times X).$
(Exercise: Check this!  Most of the work is done already, as
$N^*(\Delta)$ is Lagrangian in $T^*(X\times X).$)

We now proceed as follows to find a Lagrangian (necessarily one
dimensional larger) containing $\lag_0$: let $\hamvf=\hamvf_\Box$ denote
the Hamilton vector field of the symbol of the wave operator, in the
variables $(t,x,\tau,\xi).$ (I.e., take the Hamilton vector field of
$\Box_{(t,x)}$ on the cotangent bundle of $\RR\times X \times
X$---nothing interesting happens in $y,\eta.$) By construction,
$\lag_0 \subset \Sigma_{\Box};$ we now \emph{define} $\lag$ to be the
union of integral curves of $\hamvf$ passing through points in
$\lag_0.$ More concretely, these are all backwards unit-speed
parametrized geodesics beginning at $(x=y, \xi=-\eta),$ where
$(x,\xi)$ evolves along the geodesic flow, and $(y,\eta)$ are fixed.
(Meanwhile, $t$ is evolving at unit speed, and $\tau$ is constrained
by the requirement that we are in the characteristic set so that
$\tau=-\abs{\xi}_g$.)  The manifold $\lag$ stays inside $\Sigma_\Box$
(indeed, inside the component that is $\Sigma_{D_t +\sqrt\Lap}$)
since $\hamvf$ is tangent to this manifold; moreover, $\lag$ is
automatically Lagrangian since $\omega$ vanishes on $\lag_0$ and
$\sigma_2(\Box)$ does as well, so that for $\mathsf{Y} \in T\lag_0,$ we further
have $$\omega(\mathsf{Y},\hamvf) = (d(\sigma_2(\Box)),\mathsf{Y}) =\mathsf{Y}\sigma_2(\Box) =0.$$
This gives vanishing of $\omega$ on the tangent space to $\lag$ at
points along $t=0;$ to conclude it more generally, just recall that the
flow generated by a Hamilton vector field is a family of symplectomorphisms.
\begin{exercise}
Check that $\lag$ is in fact the \emph{only} connected conic Lagrangian manifold
passing through $\lag_0$ and lying in $\Sigma_{\Box}.$  (\textsc{Hint:}
Observe that $\hamvf$ is in fact the \emph{unique} vector at each point
along $\lag_0$ that has the property $\omega(\mathsf{Y},\hamvf)=0$ for all $\mathsf{Y}
\in T\lag_0.$)
\end{exercise}

Thus, to recapitulate, if we obtain $\lag$ by flowing out $\lag_0$
(the lift of the conormal bundle of the diagonal to the characteristic
set of $D_t+\sqrt{\Lap}$) along $\hamvf,$ the Hamilton vector field of
$\Box,$ we
produce a Lagrangian on which $\Box$ is characteristic.

\begin{exercise}
Show that the phase function $\phi(t,x,\eta)-y\cdot \eta$ that we
constructed explicitly in \S\ref{section:waveparametrix} does indeed
parametrize $$
\lag =\{(t,\tau,x,\xi,y,-\eta): \tau=-\abs{\xi}_g,\
(x,\xi)=\Phi_t(y,\eta)
$$ (with $\Phi_t$ denoting geodesic flow, i.e.\ the flow generated by the Hamilton vector
field of $\abs{\xi}_g$) over $\abs{t}\ll 1.$

Compare our solution to the eikonal equation using Hamilton-Jacobi
theory in Exercise~\ref{exercise:hamjac} to what we have done here.
\end{exercise}

We now remark that while our parametrization of the Lagrangian in
\S\ref{section:waveparametrix} worked only for small $t,$ the
definition given here of $\lag \subset T^*(\RR\times X \times X)$ makes sense
\emph{globally} in $t,$ not merely for short time.  When $t$ is small
and positive and $y$ fixed, the projection of $\lag$ to $(x,\xi)$ is just the inward-pointing conormal bundle to an expanding
geodesic sphere centered at $y;$ when $t$ exceeds
the injectivity radius of $X,$ $\lag$ ceases to be a conormal bundle,
but remains a well-behaved smooth Lagrangian.

Let us now return from our lengthy digression on the construction of
$\lag$ to recall what it gets us.
Solving the eikonal equation, i.e.\ choosing $\lag,$ has reduced our
error term by one order, and we have achieved
$$
\Box u \in I^{3/4}(\RR\times X\times X,\lag);
$$
to proceed further, we invoke Property~\eqref{fioproperty:commutator}
of FIO's, to compute
$$
\sigma_{3/4}(\Box u) = i^{-1} \hamvf \sigma_{-1/4}(u);
$$
setting this equal to zero yields our first transport equation, and it is solved by simply insisting
that $\sigma_{-1/4}(u)$ be constant along the flow, hence equal to
$1,$ its value at $t=0$ (which was dictated by our $\delta$-function initial data).

Now we have achieved $\Box u =r_{-1/4} \in I^{-1/4}$
Adding an element $u_{-5/4}$ of $I^{-5/4}(\RR\times X\times X,\lag)$ to
 solve this error away and again applying \eqref{fioproperty:commutator}
 yields the transport equation
$$
i^{-1} \hamvf (\sigma_{-5/4}(u_{-5/4})) = -\sigma_{-1/4} (r_{-1/4}),
$$
which we may solve as before.
Continuing in this manner and asymptotically summing the resulting terms, we have our
parametrix $u \in I^{-1/4}(\RR\times X \times X,\lag).$

\

Now we describe, \emph{very roughly}, how to use the FIO calculus to compute
the singularities of $\Tr U(t)$ at lengths of closed geodesics.

Let $T$ denote the operator $\CI(\RR\times X \times X) \to \CI(\RR)$
given by\footnote{It is here that our omission of density factors
  becomes most serious: $T$ should really act on \emph{densities}
  defined along the diagonal, so that the integral over $X$ is
  well-defined.  Fortunately, $U$ itself should be a
  \emph{right}-density (i.e.\ a section of the density bundle lifted
  from the right factor); restricted to the diagonal, this yields a
  density of the desired type.}
$$
T: f(t,x,y) \mapsto \int_X f(t,x,x) \, dx.
$$
Thus, $\Tr U = T (U),$ and we seek to identify this composition as a
Lagrangian distribution on $\RR^1;$ such a distribution is thus
conormal to some set of points; as we saw above (and will see
again below) these points may only be the lengths of closed geodesics,
together with $0.$

The Schwartz kernel of $T$ is the distribution
$$
\delta(t-t') \delta (x-y)
$$
on $\RR\times \RR \times X \times X;$ it is thus conormal to
$t=t',x=y,$ i.e.\ is a Lagrangian distribution with respect to the Lagrangian
$$
\{t=t',x=y,\tau=-\tau',\xi=-\eta\}
$$
Noting that if we reshuffle the factors into $(\RR\times X) \times
(\RR\times X),$ the distribution $\delta(t-t')\delta(x-y)$ becomes the kernel of the
identity operator, we can easily see that the order of this Lagrangian
distribution is $0.$ Thus,
$$
T \in I^0(\RR\times \RR \times X \times X,\Gamma')
$$
where the relation $\Gamma: T^*(\RR\times X\times X) \to T^*\RR$ maps
as follows:
$$
\Gamma(t,\tau,x,\xi,y,\eta) =\begin{cases}
\emptyset,\ \text{if } (x,\xi) \neq (y,-\eta)\\
(t,\tau),\  \text{if } (x,\xi)=(y,-\eta).
\end{cases}
$$

Let $\lag$ be the Lagrangian for our parametrix $u$ constructed above.
If an interval about $L\in \RR$ contains no lengths of closed geodesics, then
we see that no points in  $\lag$ lie over $\{(x,\xi)=(y,-\eta)\}$ for
$t$ near $L,$ hence $\Gamma(\lag)$ has no points over this
interval, i.e.\ the composition $Tu$ is smooth in this interval.
\emph{This gives another proof of the Poisson relation,} Theorem~\ref{theorem:poisson}.

If, by contrast, there is a closed geodesic of length $L,$ then
$$
\{(L,\tau): \tau<0\} \in \Gamma( \lag).
$$
Note that in effect we get a contribution from every $(x,\xi)$ lying
along the geodesic, and that in particular, the fiber over $(L,\tau)$ of the
projection on the left factor
$$
\left(T^*\RR \times \Delta_{T^*(\RR \times X\times X) \times
    T^*(\RR\times X\times X)}\right) \cap\left(\Gamma \times \lag\right) \to
 T^*\RR
$$
(giving the composition $\Gamma(\lag)$)
consists of at least a whole geodesic of length $L$, rather than a single point.
Thus, the composition of these canonical relations \emph{is not
  transverse} and the machinery described thus far does not apply.  In
\cite{Duistermaat-Guillemin}, Duistermaat-Guillemin remedied this
deficiency by constructing a
theory of composition of FIO's with canonical relations intersecting
\emph{cleanly}. 
\begin{definition}
Two manifolds $X,Y$ intersect \emph{cleanly} if $X\cap Y$ is a
manifold with $T(X\cap Y) =TX \cap
TY$ at points of intersection.
\end{definition}
For instance, pairs of coordinate axes intersect cleanly but not
transversely in $\RR^n.$ In general, in the notation of
Property~\eqref{fioproperty:alg}, if the intersection of the
product of canonical relations $\Gamma_1 \times \Gamma_2$ with the
partial diagonal
$T^*X \times \Delta \times T^*Z$ is clean, we define the
\emph{excess}, $e,$ to be the dimension of the fiber of the projection
from this intersection to $T^*X \times T^*Z;$ this is zero in the case
of transversality.  Duistermaat-Guillemin show:
$$
S\circ T \in I^{m+m'+e/2}(X \times Z, (\Gamma_1\circ \Gamma_2)')
$$
i.e.\ composition goes as before, but with a change in order.  In
addition the symbol of the product is obtained by \emph{integrating} the
product of the symbols over the $e$-dimensional fiber of the
projection in what turns out to be an invariant way.

Let us now assume that there are finitely many closed geodesics of
length $L,$ and that they are \emph{nondegenerate} in the following
sense.  For each closed bicharacteristic (i.e.\ lift to $S^*X$ of a
closed geodesic) $\gamma \subset S^*X,$ pick a
point $p \in \gamma$ and let $Z\subset S^*X$ be a small patch
of a hypersurface through $p$ transverse to $\gamma.$  Shrinking
$Z$ as necessary, we can consider the map $P_\gamma: Z \to
Z$ taking a point to its first intersection with $Z$ under
the bicharacteristic flow on $S^*X.$  This is called a \emph{Poincar\'e map}.  Since
$P_\gamma(p)=p,$ we can consider $dP_\gamma: T_pZ \to T_pZ.$  We say
that the closed geodesic is \emph{nondegenerate} if $\Id-d P_\gamma$ is
invertible.  Note that this condition is independent of our choices of
$p$ and $Z,$ as are the eigenvalues of $\Id-d P_\gamma.$

The following is due to Duistermaat-Guillemin \cite{Duistermaat-Guillemin}:
\begin{theorem}\label{theorem:DG}
Assume that all closed geodesics of length $L$ on $X$ are
nondegenerate.  Then
$$
\lim_{t\to L} (t-L) \Tr U(t) = \sum_{\gamma \text{ of length L}}
\frac{L}{2\pi} i^{\sigma_\gamma} \abs{\Id-dP_\gamma}^{-1/2},
$$
where $P_\gamma$ is the Poincar\'e map corresponding to the geodesic
$\gamma,$ and $\sigma_\gamma$ is the Morse index of
the variational problem for the energy with periodic boundary conditions, evaluated at $\gamma$.
\end{theorem}

A proof of this theorem requires understanding the symbol of the clean
composition $Tu$ (where $u$ is our parametrix for the half-wave
equation).  This lies beyond the scope of these notes.  We merely note
that we are in the setting of clean composition with excess $1,$ hence
locally near $t=L,$
$$
Tu \in I^{0-1/4+1/2}(\RR, \{t=L,\tau<0\}).
$$
This Lagrangian is easily seen to be parametrized, locally near $t=L,$ by the phase
function with one fiber variable\footnote{This phase function should
  of course be modified to make it smooth across $\theta=0,$ but
  making this modification will only add a term in $\CI(\RR)$ to the
  Lagrangian distribution we write down.}
$$
\phi(t,\theta) = \begin{cases} (t-L)\theta, \theta<0,\\ 0, \theta\geq 0;\end{cases}
$$
hence we may write
$$
Tu =(2\pi)^{-3/4} \int_0^\infty a(t,\theta) e^{-i(t-L)\theta} \, d\theta,
$$
where $a \in S^{0}(\RR\times \RR)$ has an asymptotic expansion $a \sim
a_0+ \abs{\theta}^{-1} a_{-1}+\dots.$ Our task is to find the
leading-order behavior of $Tu,$ and this is of course dictated by its
principal symbol.  To top order, $a$ is given by the constant function
$a_0(L,1),$ hence $Tu$ is (to leading order) a universal constant
times $a_0(L,1)$ times the Fourier transform of the Heaviside
function, evaluated at $t-L.$ Thus, the limit in the statement of the
theorem is, up to a constant factor, just the value of $a_0(L,1).$ The
whole problem, then, is to compute the principal symbol of this clean
composition, and we refer the interested reader to
\cite{Duistermaat-Guillemin} for the (rather tricky) computation.\footnote{We note
  that the factor $i^{\sigma_\gamma}$ is the contribution of the
  (in)famous Keller-Maslov index, and is in many ways the subtlest
  part of the answer.}

\section{A global calculus of pseudodifferential operators}

\subsection{The scattering calculus on $\RR^n$}
We now return to some of the problems discussed in
\S\ref{section:prequel}, involving operators on noncompact manifolds.
Recall that the Morawetz estimate on $\RR^n,$ for instance, hinged
upon a \emph{global} commutator argument, involving the commutator of
the Laplacian with $(1/2)(D_r +D_r^*)$ on $\RR^n.$ Generalizing this
estimate to noncompact manifolds will require some understanding of
differential and pseudodifferential operators that is uniform near
infinity.  Recall that thus far, we have focused on the calculus of
pseudodifferential operators on compact manifolds; in discussing
operators on $\RR^n,$ we have avoided as far as possible any
discussion of asymptotic behavior at spatial infinity.  Thus, our next
step is to discuss a calculus of operators---initially just on
$\RR^n$---that involves sensible bounds near infinity.

Thus, let us consider pseudodifferential symbols defined on all of
$T^*\RR^n$ with no restrictions on the support in the base variables,
with asymptotic expansions in \emph{both} the base and fiber
variables, both separately and jointly.  To this end, note that
changing to variables $\abs{x}^{-1},\hat{x},$ $\abs{\xi}^{-1},$ and
$\hat \xi$ amounts to \emph{compactifying} the base and fiber
variables of $T^*\RR^n$ radially, to make the space $B^n_x \times
B^n_\xi,$ with $B^n$ denoting the closed unit ball.  (Recall that we
defined a radial compactification map in \eqref{RC}, and that while
$\ang{\xi}^{-1}$ and $\ang{x}^{-1}$ are what we should really use as
defining functions for the spheres at infinity, $\abs{\xi}^{-1}$ and
$\abs{x}^{-1}$ are acceptable substitutes as long as we stay away from the
origin in the corresponding variables.)  The space $B^n
\times B^n$ is a \emph{manifold with codimension-two corners,} i.e.\ a
manifold locally modelled on $[0,1) \times [0,1) \times \RR^{2n-2};$
its boundary is the union of the two smooth hypersurfaces
$S^{n-1}_x\times B^n_\xi$ and $B^n_x \times S^{n-1}_\xi.$ In our local
coordinates, $\abs{x}^{-1}$ and $\abs{\xi}^{-1}$ are the defining
functions for the two boundary hypersurfaces, i.e.\ the variables
locally in $[0,1),$ while a choice of $n-1$ of each of the $\hat x$
and $\hat \xi$ variables gives the remaining $\RR^{n-2}.$

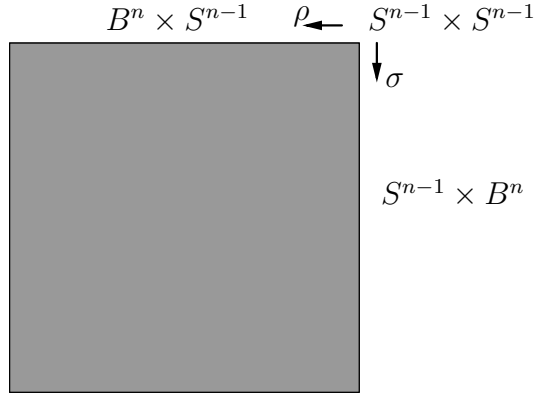
\begin{figure}[h]
\setlength{\unitlength}{0.0003in}
\begingroup\makeatletter\ifx\SetFigFont\undefined%
\gdef\SetFigFont#1#2#3#4#5{%
  \reset@font\fontsize{#1}{#2pt}%
  \fontfamily{#3}\fontseries{#4}\fontshape{#5}%
  \selectfont}%
\fi\endgroup%
{\renewcommand{\dashlinestretch}{30}
\begin{picture}(6892,6747)(0,-10)
\put(6462,5262){$\sigma$}
\thicklines
\path(5712,6312)(5112,6312)
\blacken\path(5352.000,6372.000)(5112.000,6312.000)(5352.000,6252.000)(5352.000,6372.000)
\path(6312,6012)(6312,5412)
\blacken\path(6252.000,5652.000)(6312.000,5412.000)(6372.000,5652.000)(6252.000,5652.000)
%\put(4887,6612){$\rho$}
\put(4887,6412){$\rho$}
\put(1662,6237){$B^n \times S^{n-1}$}
\put(6387,3237){$S^{n-1}\times B^n$}
\put(6162,6237){$S^{n-1}\times S^{n-1}$}
\thinlines
\texture{55888888 88555555 5522a222 a2555555 55888888 88555555 552a2a2a 2a555555 
	55888888 88555555 55a222a2 22555555 55888888 88555555 552a2a2a 2a555555 
	55888888 88555555 5522a222 a2555555 55888888 88555555 552a2a2a 2a555555 
	55888888 88555555 55a222a2 22555555 55888888 88555555 552a2a2a 2a555555 }
\shade\path(12,6012)(6012,6012)(6012,12)
	(12,12)(12,6012)
\path(12,6012)(6012,6012)(6012,12)
	(12,12)(12,6012)
\end{picture}
}
\caption{The manifold with corners $B^n\times B^n$ in the case $n=1.$
  At the top (and bottom) are the boundary faces from $B^{n}\times
  S^{n-1}$ arising from the compactification of the second
  factor---this is ``fiber infinity.''  At left (and right) are the
  faces from $S^{n-1}\times B^n,$ arising from compactification of the
  first factor---this is ``spatial infinity.''  The corner(s) at which
  these faces meet is $S^{n-1}\times S^{n-1}.$ The functions
  $\rho=\abs{x}^{-1}$ and $\sigma=\abs{\xi}^{-1}$ can be locally taken as
  defining functions for the spatial infinity resp.\ fiber infinity
  boundary faces.  The disconnectedness of $B^{n}\times S^{n-1}$ and
  $S^{n-1}\times B^n$ is of course a feature unique to dimension one.}
\end{figure}

We now let\footnote{This space should really be called $S^{m,l}_{\cl,\SC},$
  with the $\cl$ once again indicating ``classicality'' (as opposed to
  Kohn-Nirenberg type of estimates alone).  We omit the $\cl$ so as
  not to clutter up the notation.}
$$
S^{m,l}_{\SC}(T^*\RR^n)
$$
denote the space of $a \in \CI(T^*\RR^n)$ such that\footnote{We are
  abusing notation here by ignoring the diffeomorphism of radial
  compactification, thus identifying $\CI(B^n\times B^n)$ directly
  with a space of functions on $\RR^n \times \RR^n.$ }
\begin{equation}\label{scatteringsymbols}
\ang{\xi}^{-m}\ang{x}^{-l} a \in \CI(B^n\times B^n).
\end{equation}
This condition gives asymptotic expansions (i.e., Taylor series) in various regimes:
\begin{equation}
\begin{aligned}
a(x,\xi)&\sim \sum \abs{\xi}^{m-j}
a_{\bullet,j}(x,\hat{\xi}), \text{ as } \xi \to \infty,\ x \in U
\Subset \RR^n \cong (B^n)^\circ\\
a(x,\xi)&\sim \sum \abs{x}^{l-i}
a_{i,\bullet}(\hat x,\xi), \text{ as } x \to \infty,\ \xi \in V
\Subset \RR^n \cong (B^n)^\circ\\
a(x,\xi)&\sim \sum \abs{x}^{l-i} \abs{\xi}^{m-j} 
a_{ij}(\hat{x},\hat{\xi}), \text{ as } x,\xi \to \infty.
\end{aligned}
\end{equation}
Finally, let
$$
\Psisc^{m,l}(\RR^n)
$$
denote the space consisting of the (left) quantizations of these
symbols.  The ``sc''
stands for ``scattering.''\footnote{This is a space of operators
  considered by many authors; as we are following roughly the
  treatment of Melrose \cite{Melrose:spectral}, we have adopted his
  notation for the space.  Note, however, that we have reversed the
  sign from his convention for the order $l$.}

This is an algebra of pseudodifferential operators, containing all
ordinary pseudodifferential operators on $\RR^n$ with compactly
supported Schwartz kernels.  The algebra of scattering
pseudodifferential operators enjoys all the good properties of our
usual algebra, plus some more that derive from its good behavior at
infinity.  We can compose operators to get new operators, and if $A
\in \Psisc^{m,l}(\RR^n),$ $B \in \Psisc^{m',l'}(\RR^n),$ we have $AB
\in \Psisc^{m+m',l+l'}(\RR^n).$ Likewise, adjoints preserve orders.
What is novel here, however, is the principal symbol map.

As the symbols defined by \eqref{scatteringsymbols} are those that, up
to overall factors, are smooth functions on $B^n \times B^n,$ we can
define the \emph{principal symbol} of order $m,l$ of the operator
$\Op(a)$ as 
$$
\hsigma_{m,l}(A)= \ang{\xi}^{-m} \ang{x}^{-l} a |_{\pa(B^n \times B^n)};
$$
this can be further split into pieces corresponding to the
restrictions to the two boundary hypersurfaces:
$$
\hsigma_{m,l}(A) = (\hsigma_{m,l}^\xi(A), \hsigma_{m,l}^x(A))
$$
where
$$
\hsigma_{m,l}^{\xi}(A)(x,\hat\xi) \in \CI(B^n \times S^{n-1})
$$
is nothing but the ordinary principal symbol, rescaled by a power of
$\ang{x},$ and
$$
\hsigma_{m,l}^x(A)(\hat x,\xi) \in
\CI(S^{n-1}\times B^n)
$$
is the novel piece of the symbol, measuring the behavior of the
operator at spatial infinity.  Note that these two pieces of the
principal symbol are not independent: they must agree at the
\emph{corner,} $S^{n-1} \times S^{n-1}.$  We may also choose to think
of the principal symbol as
$$
\sigma^{m,l}(A)\in S^{m,l}_{\SC}(T^*\RR^n)/S^{m-1,l-1}_{\SC}(T^*\RR^n),
$$
and we will often confuse the symbol with its equivalence class; this
is usually less confusing than keeping track of the rescaling factor
$\ang{\xi}^m \ang{x}^l.$

The principal symbol short exact sequence thus reads:
$$
0 \to \Psisc^{m-1,l-l}(\RR^n)  \to \Psi^{m,l}(\RR^n)
\stackrel{\hsigma_{m,l}}{\to} \CI(\pa(B^n \times B^n)) \to 0.
$$
Thus, vanishing of this symbol yields improvement in both orders at
once; correspondingly, vanishing of one part of the symbol gives
improvement in just one order:
$$
0 \to \Psisc^{m-1,l}(\RR^n)  \to \Psi^{m,l}(\RR^n)
\stackrel{\hsigma_{m,l}^\xi}{\to} \CI(B^n \times S^{n-1}) \to 0,
$$
$$
0 \to \Psisc^{m,l-1}(\RR^n)  \to \Psi^{m,l}(\RR^n)
\stackrel{\hsigma_{m,l}^x}{\to} \CI(S^{n-1} \times B^n) \to 0.
$$

The symbol of the product of two scattering operators is indeed the
product of the symbols,\footnote{It is exactly this innocuous
  statement, which the reader might think routine, that separates the
  scattering calculus from many other choices of pseudodifferential
  calculus on noncompact manifolds: typically the ``symbol at
  infinity'' (here $\hsigma_{m,l}^x(\hat x, \xi)$) will compose under operator
  composition in a more complex, noncommutative way.}
 as (equivalence classes of) smooth functions on $\pa(B^n \times B^n).$

The symbol of the commutator of two scattering operators (which is of
lower order than the product in both filtrations) is, as one might
suspect, given by $i$ times the Poisson bracket of the symbols.

The residual calculus is particularly nice in this setting: instead of
merely consisting of smoothing operators, it consists of operators
that are ``Schwartzing''---they create decay as well as smoothness:
$$
R \in \Psisc^{-\infty,-\infty}(\RR^n) \Longleftrightarrow R:
\schwartz'(\RR^n) \to \schwartz(\RR^n).
$$

One problem with using the ordinary calculus for global matters is
that we can only conclude compactness of operators of negative order
for compactly supported operators.  Here, we have a much more precise
result:
\begin{proposition}
An operator in $\Psisc^{0,0}(\RR^n)$ is bounded on $L^2(\RR^n);$ an
operator of order $(m,l)$ with $m,l<0$ is compact on $L^2(\RR^n).$
\end{proposition}

Associated to the expanded notion of symbol, there is are associated
notions of ellipticity (nonvanishing of the principal symbol) and of
$\WF'$ (lack of infinite order vanishing of the total symbol).  We
have an associated family of Sobolev spaces:
$$
u \in \Hsc^{m,l}(\RR^n) \Longleftrightarrow \forall A
\in \Psisc^{m,l}(\RR^n),\ Au \in L^2(\RR^n).
$$
Operators in the calculus act on this scale of Sobolev spaces in the
obvious way.  Since smoothing operators are ``Schwartzing,'' it is not
hard to see that
$$
\Hsc^{-\infty,-\infty}(\RR^n)=\schwartz(\RR^n).
$$
(We will return to an explicit description of these
Sobolev spaces shortly.)

There is also an associated wavefront set:
$$
\WFsc u \subset \pa(B^n\times B^n)
$$
is defined by
$$
p \notin \WFsc u \Longleftrightarrow \text{ there exists } A \in
\Psisc^{0,0}(\RR^n), \text{ elliptic at } p, \text{ with } Au \in \schwartz.
$$
In $(B^n_x)^\circ \times S^{n-1}_\xi\subset\pa(B^n \times B^n),$
(i.e., in the
usual cotangent bundle of $\RR^n$) this definition just coincides with
ordinary wavefront set; but ``at infinity,'' i.e.\ in $S^{n-1}_x
\times B^n_\xi,$ it measures something new.  To see what, let us
consider some examples.

\begin{example}\
\begin{enumerate}
\item \emph{Constant coefficient vector fields on $\RR^n:$}  If $v \in \RR^n$
  and 
  $P=i^{-1} v\cdot \nabla,$ then, we can write
$$
P =\Opl(v \cdot \xi);
$$
the principal symbol is thus
$$
\sigma_{1,0}(P) = v\cdot \xi
$$
\item Likewise, the symbol of the Euclidean Laplacian $\Lap$ is
  $\sigma_{2,0}(\Lap)=\abs{\xi}^{2}.$ Note that the Laplacian is
  \emph{not elliptic} in the scattering calculus, as its principal
  symbol vanishes at $\xi=0$ on the boundary face $S^{n-1}_x \times
  B^n_\xi.$ This should come as no suprise, as $\Lap$ has nullspace in
  $\schwartz'(\RR^n)$ (given by harmonic polynomials) that does not lie in
  $L^2,$ hence is not consistent with elliptic regularity in the
  scattering calculus sense: if $Q$ is elliptic in the scattering
  calculus, $$Q u \in \schwartz(\RR^n)\Longrightarrow u \in
  \schwartz(\RR^n).$$

  On the other hand,
  consider $\Id +\Lap.$  We have $\Id \in \Psisc^{0,0}(\RR^n),$ hence
  adding it certainly does not alter the ``ordinary'' part of the
  symbol, living on $(B^n)^\circ \times S^{n-1}.$  But
  it \emph{does} affect the symbol in $S^{n-1}\times B^n:$ we have
$$
\sigma_{2,0}(\Id +\Lap) = 1+\abs{\xi}^2;
$$
$\Id+\Lap$ \emph{is} an elliptic operator in the scattering calculus,
and of course it is
the case that $(\Id+\Lap)u \in \schwartz(\RR^n)$ implies
that $u$ is likewise Schwartz.

\item If we vary the metric from the Euclidean metric to some other
  metric $g,$ we may or may not obtain a scattering differential
  operator; for example, if $g$ were periodic, we certainly would not,
  as the total symbol of $\Lap$ would clearly lack an asymptotic
  expansion as $\abs{x}\to \infty.$  Suppose, however, that we may write in spherical
  coordinates on $\RR^n$
$$
g= dr^2 +r^2 \sum h_{ij}(r^{-1}, \theta) d\theta^i d\theta^j \quad
\text{for } r>R_0 \gg 0.
$$
where $h_{ij}$ is a smooth function of its arguments, and
$$h_{ij}(0,\theta) d\theta^i d\theta^j$$ is the standard metric on the
``sphere at infinity.''
We will call such a metric \emph{asymptotically Euclidean}.  Then the corresponding Laplace
operator is in the scattering calculus.
\begin{exercise}
Check that this operator does lie in the scattering calculus.
\end{exercise}
Let $\Lap$ denote the Laplacian with respect to an asymptotically Euclidean
  metric.  Then $$(\Id+\Lap)^{-1} \in \Psisc^{-2,0}(\RR^n).$$
\item $\ang{x}^2 (\Id +\Lap) \in \Psisc^{2,2}(\RR^n)$ and has symbol
  $\ang{x}^2(1+\abs{\xi}^2).$  This is globally elliptic.
\end{enumerate}
\end{example}

By the last example, we find that
$$
u \in \Hsc^{2,2}(\RR^n) \Longleftrightarrow \ang{x}^2 (\Id+\Lap) u \in L^2(\RR^n);
$$
interpolation and duality arguments allow us to conclude more
generally that the scattering Sobolev spaces coincide with the usual
weighted Sobolev spaces:
$$
\Hsc^{m,l}(\RR^n) = \ang{x}^{-l} H^m(\RR^n).
$$

We now turn to some examples illustrating the scattering wavefront
set.  Consider the plane wave
$$
u(x) = e^{i\alpha \cdot x}.
$$
We have
$$
(D_{x^j}-\alpha_j) u =0 \text{ for all } j=1,\dots,n.
$$
The symbol of the operator $D_{x^j}-\alpha_j$ is $\xi_j-\alpha_j,$
hence the intersection of the characteristic sets of these operators
is just the points in $S^{n-1} \times B^n$ where $\xi=\alpha.$ As a
consequence, we have
$$
\WFsc (e^{i\alpha \cdot x}) \subseteq \{(\hat x, \xi) \in S^{n-1} \times \RR^n: \xi=\alpha\}
$$
(here we are as usual identifying $(B^n)^\circ \cong \RR^n$).  In fact this
containment turns out to be equality, as we see by the following
characterization of scattering wavefront set.
\begin{proposition}
Let $p=(\hat x_0, \xi_0) \in S^{n-1}\times \RR^n.$  We have
$$
p \notin \WFsc u
$$
if and only if there exist cutoff functions $\phi \in \CI_c(\RR^n)$
nonzero at $\xi_0$ and $\gamma \in \CI(\RR^n)$ nonzero in a conic
neighborhood of the direction $\hat x_0$ such that
$$
\phi \mathcal{F}(\gamma u) \in \schwartz(\RR^n).
$$
\end{proposition}
This is of course closely analogous to the characterization of
ordinary wavefront set in
Proposition~\ref{proposition:wfcharacterization}, and is proved in an
analogous manner.  Note that if $u$ is a Schwartz function in a set of
the form $$\big\{ \big\lvert \frac{x}{\abs{x}}-\hat x_0 \big\rvert<
\epsilon, \abs{x}>R_0\big\}$$
for any $\ep>0,$ $R_0\gg 0,$ then there is no scattering wavefront set at points of the
form $(\hat x_0,\xi)$ for any $\xi \in \RR^n.$  Thus, this new piece of the
wavefront set measures the asymptotics of $u$ in different directions
toward spatial infinity: $\hat x_0$ provides the direction, while
the value of $\xi_0$ records oscillatory behavior of a specific
frequency.

There is also, of course, a similar characterization of $\WFsc u$
inside $S^{n-1}\times S^{n-1}.$  We leave this as an exercise for the
reader.

\subsection{Applications of the scattering calculus}

As an example of how we might use the scattering calculus to obtain
global results on manifolds, let us return to the \emph{local
  smoothing estimate} from \S\ref{subsec:introschrodinger}.  Recall
that if $\psi$
satisfies the Schr\"odinger equation \eqref{scheqn} on $\RR^n$ with
initial data $\psi_0 \in H^{1/2},$ this estimate (or, at least, one
version of it) tells us that
\begin{equation}\label{locsmoothing}
\psi \in L^2_\loc(\RR_t; H^{1}_\loc (\RR^n)),
\end{equation}
hence the solution is (locally) half a derivative smoother than the data, on average.
How might we obtain this estimate on a manifold, with $\Lap$ replaced
by the Laplace-Beltrami operator (which we also denote $\Lap$)?  For a
start, note that \eqref{locsmoothing} fails badly on compact
manifolds; in particular, recall that since $[\Lap,\Lap^s]=0$ for all
$s \in \RR,$ the $H^s$ norms are conserved under the evolution, hence
if $\psi_0 \notin H^s,$ with $s>1/2,$ then we certainly do not
have\footnote{Note that this argument fails on $\RR^n$ exactly because
  of the distinction between local and global Sobolev regularity:
  there is nothing preventing a solution on $\RR^n$ with initial data
  in $H^{1/2}$ from being locally $H^1$---or even smooth on
  arbitrarily large compact sets---in return for having nasty behavior
  near infinity.} $\psi \in L^2_\loc(\RR_t; H^s).$ So if we seek a
broader geometric context for this estimate, we had better try
noncompact manifolds.

Recall that we initially obtained the estimate by a commutator
argument with the 
Morawetz commutant
$$
\pa_r +\frac{n-1}{2r},
$$
which actually gave more information; we noted that we could, instead,
have used a simpler commutant $f(r) D_r,$ with $f(r)=0$ near $r=0,$
nondecreasing, and equal to $1$ for $r\geq 2$ (say): this gives a
commutator with a term
$$
\chi'(r) D_r^2
$$
which, when paired with $\psi$ and integrated in time, tests for $H^1$
regularity in an annular neighborhood of the origin (which could have
been translated to be anywhere); other terms in the commutator are
positive also, modulo estimable error terms, and we thus obtain the
local smoothing estimate.  Generalizing this is tricky, as
the positivity of the symbol of the term
$$
i[\Lap,D_r]
$$
on $\RR^n$ is delicate: the symbol of this commutator is given by
the Poisson bracket
$$
\{\abs{\xi}^2, \xi\cdot \hat{x}\}= 2 \xi \cdot
\pa_x(\xi\cdot\hat{x})= \frac{2}{\abs{x}} \big(
\abs{\xi}^2-(\xi\cdot \hat{x})^2 \big)
$$
which is nonnegative but does actually vanish at $\xi \parallel x,$
i.e.\ in radial directions.  If we perturb the Euclidean metric a bit,
and replace $\abs{\xi}^2$ with $\abs{\xi}^2_g,$ the symbol of the
Laplace-Beltrami operator, but leave the inner product
$\ang{\xi,x}=\sum \xi_j x^j,$ then this computation fails to give
positivity.  So we have to be more careful.  We might try to adapt
$\sum \xi_j x^j$ to the new metric instead, but this is problematic,
as it doesn't really make much invariant sense.  Moreover, it seems
even more problematic upon interpretation: what positivity of
$\{\abs{\xi}^2_g, a\}$ \emph{means} is just that $a$ is increasing
along the bicharacteristic flow of $\abs{\xi}^2_g,$ i.e.\ is
increasing along (the lifts to the cosphere bundle of) geodesics.
This is clearly impossible if there are any closed (i.e., periodic) geodesics, or
indeed if there are geodesics that remain in a compact set for all
time, hence our difficulty in obtaining an estimate on compact
manifolds.
\begin{exercise}
  Suppose that a geodesic $\gamma$ remains in a compact subset of
  $\RR^n$ (equipped with a non-Euclidean metric) for all $t>0.$ Let
  $p=(\gamma(0),(\gamma'(0))^*) \in T^*\RR^n$ (with $*$ denoting dual
  under the metric).  Show that there cannot exist a smooth $a \in \CI
  (T^*\RR^n)$ with $\{\abs{\xi}^2_g, a\}\geq\ep>0$ and $a(p) \neq 0.$
\end{exercise}

\begin{definition}
Let $g$ be an asymptotically Euclidean metric on $\RR^n,$ and let
$\gamma$ be a geodesic.  We say that $\gamma$ is \emph{not trapped
  forward/backward} if
$$
\lim_{t\to \pm \infty} \abs{\gamma(t)}=\infty.
$$
We say that $\gamma$ is \emph{trapped} if it is trapped
both forward \emph{and} backward.  We also use the same notation for the
bicharacteristic projecting to $\gamma.$  Moreover, we say that a
point in $S^*\RR^n$ along a non-(forward/backward)-trapped geodesic is itself
non-(forward/backward)-trapped.
\end{definition}
It is a theorem of Doi \cite{Doi} that the local smoothing estimate
\eqref{locsmoothing} \emph{cannot hold} near a trapped geodesic.  (The
total failure of \eqref{locsmoothing} on compact manifolds should make
this plausible, but it turns out to be considerably more delicate to
show that it fails even if the only trapping is, for instance, a
single, highly unstable, closed geodesic.)  As a result we will require
some strong geometric hypotheses in in order to find a general context
in which \eqref{locsmoothing} holds.

The following is a result of Craig-Kappeler-Strauss \cite{CKS}:
\begin{theorem}
  Consider $\psi$ a solution to the Schr\"odinger equation on
  asymptotically Euclidean space, with $\psi_0 \in H^{1/2}(\RR^n).$
  The estimate \eqref{locsmoothing} holds microlocally at any
  $(x_0,\xi_0)$ that lies on a nontrapped bicharacteristic, i.e.\ for
  any $A\in \Psi^{1}(\RR^n)$ compactly supported and microsupported
  sufficiently near to $(x_0,\xi_0),$ we have for any
  $T>0,$\footnote{More generally, we can replace the Sobolev exponents
    $1/2$ and $1$ by $s$ and $s+1/2$ respectively; in particular,
    $L^2$ initial data gives an $L^2 H^{1/2}$ estimate.}
$$
\int_0^T \norm{A\psi}^2 \, dt \lesssim \norm{\psi_0}_{H^{1/2}}^2.
$$
\end{theorem}
\begin{proof}
  We will prove the theorem by using a commutator argument in the
  scattering calculus.  To begin, we recall from
  Exercise~\ref{exercise:schrodingerinvariance} that the set along
  which microlocal $L^2_\loc H^{1}$ regularity holds is invariant
  under the geodesic flow.  Hence it suffices just to obtain
  regularity of this form \emph{somewhere} along the geodesic
  $\gamma.$ The convenient place to do this is out near infinity.

  In order to make a commutator argument, note that it is very useful
  to have a quantity that behaves monotonically along the flow.  We
  refer to points in $T^*\RR^n$ near infinity (i.e.\ for $\abs{x}\gg
  0$) as \emph{incoming} if $\hat \xi \cdot \hat x<0$ and
  \emph{outgoing} if $\hat \xi \cdot \hat x>0$ (this corresponds to
  moving toward or away from the origin, respectively, under
  asymptotically Euclidean geodesic flow).  Heuristically, under the
  classical evolution, points move from being incoming to being
  outgoing.  More precisely, we observe that the Hamilton vector field
  of $p\equiv \sigma_{2,0}(\Lap)$ is given by
$$
\hamvf_p= -\sum \xi_i \xi_j \frac{\pa g^{ij}(x)}{\pa x^k} \pa_{\xi_k}+2\sum
\xi_i g^{ij}(x) \pa_{x^j}.
$$
Recalling that $g^{ij}$ has an asymptotic expansion with leading term
given by the identity metric, we can write this as
\begin{equation}\label{hamvfp}
\hamvf_p =2\xi \cdot \pa_x + O(\abs{x}^{-1}\abs{\xi})\pa_x
+O(\abs{x}^{-1}\abs{\xi}^2) \pa_\xi
\end{equation}
(where in fact the whole vector field is homogeneous of degree $1$ in
$\xi$).
\begin{exercise}
Verify \eqref{hamvfp}.
\end{exercise}
Thus,
$$
\hamvf_p (\hat\xi\cdot \hat x) = \frac{\abs\xi}{\abs{x}}
\left(1-(\hat\xi\cdot\hat x)^2\right)+ O (\abs{\xi} \abs{x}^{-1}).
$$
This is thus positive, as long as $\hat\xi \cdot \hat x$ is away from
$\pm 1$, and $\abs{x}$ is large,\footnote{Largeness of $\xi$ plays no role
  because of homogeneity of the Hamilton vector field of the principal
  symbol of $\Lap.$} i.e., as long as we stay away from precisely
incoming or outgoing points.  Thus, we
manufacture a scattering symbol for a commutant that has increase owing to the
increase in ``outgoingness:'' Let $\chi(s)$ denote a smooth function
that equals $0$ for $s<1/4$ and $1$ for $s>1/2,$ with $\chi'$ a square
of a smooth function, nonzero in the interior of its support.  Let $\chi_\delta(s)=\chi(\delta s).$ We choose
$$
a(x,\xi)= \abs{\xi}_g \chi(-\hat\xi\cdot \hat x) \chi_\delta(\abs{x})\chi(\abs{\xi}_g).
$$
Thus $a$ is supported at \emph{incoming} points at which $\abs{x} \geq
1/(4\delta)\gg 0;$ the first $\chi$ factor localizes near incoming
points, and the factor of $\chi_\delta$ keeps $\abs{x}$ large.  (The
factor $\chi(\abs{\xi}_g)$ simply cuts off near the origin in $\xi$ to
yield a smooth symbol.)  Under the flow on the support of $a,$ $x$ tends to
decrease and we become more outgoing, so the tendency is the
\emph{leave} the support of $a$ along the flow.  This is the
essential point in the following:
\begin{exercise}
Check that $a \in S^{1,0}_\SC(T^*\RR^n)$ and that if $\delta$ is
chosen sufficiently small, we may write
$$
\hamvf_p a = -b^2 -c^2
$$
where
\begin{enumerate}
\item
$b\in S^{1,-1/2}_\SC(T^*\RR^n)$ is supported in $\supp \chi'(-\hat\xi\cdot
\hat x) \chi_\delta(\abs{x})$ 
\item $c\in S^{1,-1/2}_\SC(T^*\RR^n)$ is supported in $\supp
\chi(-\hat\xi\cdot \hat x) \chi'_\delta(\abs{x})$ and nonzero on the
interior of that set.
\end{enumerate}
(Note that $\abs{\xi}_g$ is annihilated by $\hamvf_p,$ so the terms
containing $\abs{\xi}_g$ simply do not contribute.)
\end{exercise}

Now let $A \in \Psisc^{1,0}(\RR^n)$ have principal symbol $a.$ Then we
have
$$
i[\Lap, A] = -B^*B-C^*C+R
$$
with $B=\Op (b),\ C=\Op (c)\in \Psisc^{1,-1/2}(\RR^n),$ and $R \in
\Psisc^{1,-2}(\RR^n).$

Hence,
$$
\int_0^T \norm{C \psi}^2 \, dt \leq \Big\lvert\ang{A\psi,\psi}\big\rvert_0^T\Big\rvert
 + \abs{\int_0^T \ang{R\psi,\psi}\, dt}.
$$
As $\ang{A\psi,\psi}$ is bounded by the $L^\infty H^{1/2}$ norm
of $\psi$ and hence by $\norm{\psi_0}^2_{H^{1/2}},$ and the $R$ term
likewise,\footnote{In fact, the $R$ term is considerably better than
  necessary for this step, as it has weight $-2$ rather than just $0$
  (which would be all we need to obtain the estimate).  The astute
  reader may thus recognize that we are far from using the full power
  of the scattering calculus here.  A proof of the global estimate in
  Exercise~\ref{exercise:globalsmoothing} requires a more serious use
  of the symbol calculus, however, as do the estimates which are the
  focus of \cite{CKS}, which show that microlocal decay of the initial
  data yields higher regularity of the solution along bicharacteristics.}
 we obtain
\begin{equation}\label{yahoo}
\int_0^T \norm{C \psi}^2 \, dt \lesssim \norm{\psi_0}^2_{H^{1/2}}.
\end{equation}
\begin{Exercise}\label{exercise:incoming}
Show that for any $R_0 >0,$ there exists $\delta>0$ sufficiently small
that if $(x_0,\xi_0)\in T^*\RR^n \cap \{\abs{x}<R_0\}$ lies along a
non-backward trapped bicharacteristic, some point on that
bicharacteristic with $t\ll 0$ lies in $\liptic C,$ with $C=\Op(c)$
constructed as above.

Thus, rays starting close to the origin that pass through
$\abs{x} \sim \delta^{-1}$ for $t\ll 0$ are incoming when they do so.
This is an exercise in ODE.  You might begin by showing
that if a backward bicharacteristic starting in $\{\abs{x}<R_0\}$ passes
through the hypersurface $\abs{x} =R'$ with $R'\gg 0,$ then it must
have $\hat\xi\cdot \hat x<0$ there, and that $\hat\xi\cdot\hat x$ will
keep decreasing thereafter along the backward flow.
\end{Exercise}

Given a non-backward-trapped point $q \in S^*\RR^n,$
Exercise~\ref{exercise:incoming}
tells us that we may construct a commutant $A$ as above so that the
commutator term $C$ is elliptic somewhere along the bicharacteristic
through $q.$  Equation~\ref{yahoo} tells us that we have the desired
$L^2H^{1}$ estimate on $\liptic C,$ and the flow-invariance from
Exercise~\ref{exercise:schrodingerinvariance} yields the same
conclusion at $q.$  Thus, we have proved the desired result at
non-backward-trapped points.  It remains to consider
non-forward-trapped points.

Suppose, then, that $q=(x_0,\xi_0)\in T^*\RR^n$ is non-forward-trapped; then
note that $q'=(x_0,-\xi_0)$ is non-backward-trapped.  Consider then the
function $\overline{\psi}:$ if
$$
(D_t+\Lap)\psi=0
$$
then
$$
(-D_t+\Lap)\overline{\psi}=0,
$$
i.e.\
$$
\tpsi(t,x)=\overline{\psi}(T-t,x)
$$
again solves the Schr\"odinger equation.  Of course,
by unitarity, $$\big\|\tpsi(0,x)\big\|_{H^{1/2}}=\norm{\psi_0}_{H^{1/2}}.$$
Since $q'$ is non-backward trapped, we thus find that there exists $C \in
\Psisc^{1,-1/2}(\RR^n),$ elliptic at $q',$ with
$$
\int_0^T \norm{C\tpsi}^2 \,dt \lesssim \big\|\tpsi(0,x)\big\|_{H^{1/2}}^2=
\norm{\psi_0}^2_{H^{1/2}};
$$
on the other hand,
\begin{align*}
\norm{C\tpsi(t,\cdot)}^2 &= \norm{C\overline\psi(T-t,\cdot)}^2\\ &=
\norm{\overline{C} \psi(T-t,\cdot)}^2,
\end{align*}
where
$$
C=\Opl(c(x,\xi)),\ \text{and } \overline{C} = \Opl(\overline{c}(x,-\xi));
$$
thus, $\overline{C}$ tests for regularity at $q,$ and we have obtained
the desired estimate at $q.$
\end{proof}

\begin{corollary}
On an asymptotically Euclidean space with no trapped geodesics, the
local smoothing estimate holds everywhere.
\end{corollary}

\begin{Exercise}\emph{(Global (weighted) smoothing.)}\label{exercise:globalsmoothing}
Show that if there are no trapped geodesics, and $\psi_0 \in L^2,$ we have
$$
\int_0^T \norm{\ang{x}^{-1/2-\ep} \psi}_{H^{1/2}}^2 \,dt \lesssim  \norm{\psi_0}_{L^2}^2
$$
for every $\ep>0.$ (This is a bit involved; a solution can be found,
e.g., in Appendix II of \cite{htw1}.)
\end{Exercise}

\subsection{The scattering calculus on manifolds}\label{subsection:scmanifolds}

We can generalize the description of the scattering calculus to
manifolds quite easily, following the prescription of Melrose
\cite{Melrose:spectral}.  Let $X$ be a compact manifold
with boundary.  We will, in practice, think of the interior, $X^\circ,$ as a
noncompact manifold (with a complete metric) that just happens to
come pre-equipped with a compactification to $X.$  Our motivating
example will be $X=B^n,$ where $X^\circ$ is then diffeomorphically
identified with $\RR^n$ via the radial compactification map.  Recall
that on $\RR^n,$
radially compactified to the ball, we used coordinates near $S^{n-1},$
the ``boundary at infinity,'' given by
$$
\rho=\frac{1}{\abs{x}},\ \theta=\frac{x}{\abs{x}},
$$
where in fact $\rho$ together with an appropriate choice of $n-1$ of
the $\theta$'s furnish local coordinates near a point.  In these
coordinates, what do constant coefficient vector fields on $\RR^n$
look like?  We have
$$
\pa_{x^j} = \rho \pa_{\theta^j}-\rho \sum\theta^k \theta^j
\pa_{\theta^k}-\rho^2 \theta^j \pa_\rho.
$$
Recall moreover that functions in $\CI(B^n)$ correspond exactly, under
radial (un)compactification, to symbols of order zero on $\RR^n.$ So
in fact it is easy to check more generally that vector fields on $\RR^n$
with zero-symbol coefficients correspond exactly to vector fields on $B^n$
that, near $S^{n-1},$ take the form
$$
a(\rho,\theta) \rho^2\pa_\rho+\sum b_j(\rho,\theta) \rho \pa_{\theta^j},
$$
with $a,b_j \in \CI(B^n).$

We generalize this notion as follows.  Given our manifold $X,$ let
$\rho\in \CI(X)$ denote a \emph{boundary defining function}, i.e.\
$$
\rho\geq 0 \text{ on X, } \rho^{-1}(0)=\pa X,\ d\rho\neq 0 \text{ on }
\pa X.
$$
Let $\theta^j$ be local coordinates on $\pa X.$
We define \emph{scattering vector fields on $X$} to be those that can
be written locally, near $\pa X,$ in the form
$$
a(\rho,\theta) \rho^2\pa_\rho+\sum b_j(\rho,\theta) \rho \pa_{\theta^j},
$$
with $a,b_j \in \CI(X).$  Let
$$
\Vsc(X)=\{\text{scattering vector fields on }X\}
$$
\begin{exercise}\
\begin{enumerate}
\item Show that $\Vsc(X)$ is well-defined, independent of the choices
  of $\rho,\theta.$  
\item Let $\Vb(X)$ denote the space of smooth vector fields on $X$
  tangent to $\pa X.$  Show that
$$
\Vsc(X) = \rho \Vb(X)
$$
\item Show that both $\Vsc(X)$ and $\Vb(X)$ are Lie algebras.
\end{enumerate}
\end{exercise}

As we can locally describe the elements of $\Vsc(X)$ as the $\CI$-span
of $n$ vector fields, $\Vsc(X)$ is itself the space of sections of a
\emph{vector bundle}, denoted
$$
\Tsc X.
$$
There is also of course a dual bundle, denoted
$$
\Tscstar X,
$$
whose sections are the $\CI$-span of the one-forms
$$
\frac{d\rho}{\rho^2},\ \frac{d\theta^j}{\rho}.
$$
Over $X^\circ,$ we may of course canonically identify $\Tscstar X$
with $T^*X,$ and the canonical one-form on the latter pulls back to
give a canonical one-form
\begin{equation}\label{canoneform}
\xi \frac{d\rho}{\rho^2} + \eta \cdot \frac{d\theta}{\rho}
\end{equation}
defining coordinates $\xi,\eta$ on the fibers of $\Tscstar X.$

The scattering calculus on $\RR^n$ is concocted to contain scattering
vector fields:
\begin{exercise}
Show that $\Psisc^{1,0}(\RR^n) \supset \Vsc(B^n).$
\end{exercise}
We can, following Melrose, define the scattering calculus more
generally as follows.  Let $\Tbarscstar X$ denote the
\emph{fiber-compactification} of the bundle $\Tscstar X,$ i.e.\ we are
radially compactifying each fiber to a ball, just as we did globally
in compactifying $T^*\RR^n$ to $B^n \times B^n,$ only this time, the
base is already compact.  Now let
$$
S_{\SC}^{m,l}(\Tscstar X) = \sigma^{-m} \rho^{-l} \CI(\Tbarscstar X),
$$
where $\sigma$ is a boundary defining function for the fibers.  We can
(by dint of some work!) quantize these ``total'' symbols to a space of
operators, denoted
$$
\Psisc^{m,l}(X).
$$
(Note that in the case $X=B^n,$ we recover what we were previously
writing as $\Psisc^{m,l}(\RR^n);$ the latter usage, with $\RR^n$
instead of the more correct $B^n,$ was an abuse of
the usual notation.)
The principal symbol of a scattering operator is, in this invariant
picture, a smooth function on
$\pa(\Tbarscstar X);$ or equivalently, an equivalence class of smooth
functions on $\Tbarscstar X;$
or, in the partially uncompactified picture, an equivalence class of
smooth symbols on $\Tscstar X.$  (It is this last point of view that
we shall mostly adopt.)
In the coordinates defined by the canonical one-form
\eqref{canoneform}, we have
\begin{equation}\label{scsymbols}
\sigma_{1,0}(\rho^2 D_\rho)=\xi,\ \sigma_{1,0}(\rho D_{\theta^j}) =\eta_j.
\end{equation}

Recall that the Euclidean metric may be written in polar
coordinates as
$$
d(\rho^{-1})^2 +(\rho^{-1})^2 h(\theta,d\theta)
$$
with $h$ denoting the standard metric on $S^{n-1}.$  We can
generalize this to define a \emph{scattering metric} as one on a
manifold with boundary $X$ that can
be written in the form 
$$
\frac{d\rho^2}{\rho^4}+ \frac{h(\rho,\theta,d\theta)}{\rho^2}
$$
locally near $\pa X$, with $\rho$ a boundary defining function, and
$h$ now a smooth \emph{family} in $\rho$ of metrics on $\pa
X.$\footnote{The usual definition, as in \cite{Melrose:spectral}, is a
  little more general, allowing $d\rho$ terms in $h;$ however,
  it was shown by Joshi-S\'a Barreto that these terms can always be
  eliminated by appropriate choice of coordinates.}

\begin{exercise}\label{exercise:sclap}\
\begin{enumerate}\item
Show that if $g$ is a scattering metric on $X,$ then the Laplace
operator with respect to $g$ can be written
$$
\Lap =(\rho^2 D_\rho)^2 + O(\rho^3) D_\rho+ \rho^2 \Lap_\theta
$$
where $\Lap_\theta$ is the family of Laplacians on $\pa X$ associated
to the family of metrics $h(r,\theta,d\theta).$
\item Show that for $\lambda \in \CC,$
$$
\sigma_{2,0}(\Lap-\lambda^2) = \xi^2 +\abs{\eta}_h^2-\lambda^2.
$$
(Note that this entails noticing that you can drop the $O(\rho^3)
D_\rho$ terms for different reasons at the the two different boundary
faces of $\Tbarscstar X.$ The term $-\lambda^2$ is of course only
relevant at the $\rho=0$ face; it does not contribute to the part of the
symbol at fiber infinity, as it is a lower-order term there.)
\end{enumerate}
\end{exercise}
As a consequence of Exercise~\ref{exercise:sclap}, note as before that
for $\lambda \in \RR,$ the Helmholz operator $\Lap-\lambda^2$ is not
elliptic in the scattering sense: there are points in $\Tscstar_{\pa
  X} X$ where $\xi^2 +\abs{\eta}_h^2=\lambda^2$.

We now turn to scattering wavefront set $\WFsc$, which can, as one might
expect, be defined in the usual manner as a subset of
$$
\pa (\Tbarscstar X),
$$
hence is a subset of boundary faces at fiber infinity and at spatial
infinity (i.e., over $\pa X$). The scattering wavefront set is the obstruction to
a distribution lying in $\dot{\mathcal{C}}^\infty(X),$ where
$\dot{\mathcal{C}}^\infty(X)$ denotes the set of smooth functions on
$X$ decaying to infinite order at $\pa X.$ This space is the analogue
of the space of Schwartz functions in our compactified picture:
\begin{exercise}
Show that pullback under the radial compactification map sends
$\dot{\mathcal{C}}^\infty(B^n)$ to $\schwartz(\RR^n).$
\end{exercise}

By \eqref{scsymbols}, it is not hard to
see that
$$
(\rho^2 D_\rho-\alpha) u=0 \Longrightarrow \WFsc u\subset \{\rho=0, \xi=\alpha\},
$$
$$
(\rho D_{\theta^j}-\beta) u=0 \Longrightarrow \WFsc u\subset
\{\rho=0, \eta_j=\beta\}.
$$
The following variant provides a useful family of examples (and can be
proved with only a little more thought): if $a(\rho,\theta)$ and
$\phi(\rho,\theta) \in \CI(X),$ then\footnote{The distribution $a
  e^{i\phi}$ used here is a simple example of a \emph{Legendrian
    distribution.} The class of Legendrian distributions on manifolds with
  boundary, introduced by Melrose-Zworski \cite{Melrose-Zworski},
  stands in the same relationship to Lagrangian distributions as
  scattering wavefront set does to ordinary wavefront set.}
$$
\WFsc \big( a(\rho,\theta) e^{i\phi(\rho,\theta)/\rho}\big) =
\{(\rho=0,\theta,d(\phi(\rho,\theta)/\rho): (0,\theta) \in \esssupp a\},
$$
where $\esssupp a \subseteq \pa X$ denotes the ``essential support'' of
$a,$ i.e.\ the points near which $a$ is not $O(\rho^\infty).$

Of course, if 
\begin{equation}\label{helmholtz}
(\Lap-\lambda^2) u =f \in \dot{\mathcal{C}}^\infty(X),
\end{equation}
then we have, by microlocal elliptic regularity,
$$
\WFsc u \subset \{\rho=0,\ \xi^2 +\abs{\eta}_h^2=\lambda^2\}.
$$

In fact, there is a propagation of singularities theorem for
scattering operators of real principal type that further constrains
the scattering wavefront set of a solution to \eqref{helmholtz}: it
must be invariant under the (appropriately rescaled) Hamilton vector
field of the symbol of $\Lap-\lambda^2.$
\begin{Exercise}
Let $\omega=d(\xi\, d\rho/\rho^2+ \eta \cdot d\theta/\rho)$ and let
$$
p=\xi^2 +\abs{\eta}_h^2-\lambda^2;
$$
show that up to an overall scaling factor, the Hamilton vector field
of $p$ with respect to the symplectic form $\omega$ is, on the face, $\rho=0$ just
$$
\hamvf_p = 2\xi \eta\cdot \pa_\eta -2\abs{\eta}^2_{h_0}\pa_\xi +\hamvf_{h_0}
$$
where $h_0=h|_{\rho=0},$ and $\hamvf_{h_0}$ is the Hamilton vector field of
$h_0,$ i.e.\ (twice) geodesic flow on $\pa X.$

Show that maximally extended bicharacteristics of $\hamvf_p$ project to the
$\theta$ variables to be geodesics of length $\pi.$  (Hint:
reparametrize the flow.)

(For a careful treatment of the material in this exercise and indeed
in this section, see
\cite{Melrose:spectral}.)
\end{Exercise}

\section*{Appendix}

We give an extremely sketchy account of some background material
on Fourier transforms, distribution theory, and Sobolev spaces.
For further details, see, for instance, \cite{Taylor:PDE1} or
\cite{Hormander:v1}.

Let $\schwartz(\RR^n),$ the \emph{Schwartz space,} denote the space
$$
\{\phi \in \CI(\RR^n): \sup \abs{x^\alpha \pa_x^\beta \phi}<\infty\
\forall \alpha,\beta\},
$$
topologized by the seminorms given by the suprema.  The dual space to
$\schwartz(\RR^n),$ denoted $\schwartz'(\RR^n),$ is the space of
tempered distributions.

For $\phi\in \schwartz(\RR^n),$ let
$$
\F \phi (\xi) =(2\pi)^{-n/2} \int \phi(x) e^{-i\xi\cdot x} \, dx.
$$
Then $\F \phi \in \schwartz(\RR^n),$ too; indeed, $\F:
\schwartz(\RR^n)\to \schwartz(\RR^n)$ is an isomorphism, and its
inverse is closely related:
$$
\F^{-1} \psi (x) =(2\pi)^{-n/2} \int \psi(\xi) e^{+i\xi\cdot x} \, dx.
$$
We can, by duality, then define $\F$ on tempered distributions.

Let $\mathcal{E}'(\RR^n)$ denote the space of compactly supported
distributions on $\RR^n.$ When $X$ is a compact manifold without
boundary, we let $\mathcal{D}'(X)$ denote the dual space of $\CI(X).$

We define the ($L^2$-based) Sobolev spaces by
$$
H^s (\RR^n) = \{u \in \schwartz'(\RR^n): \ang{\xi}^s \F
u(\xi) \in L^2(\RR^n)\},
$$
where $\ang{\xi}=(1+\abs{\xi}^2)^{1/2}.$ If $s$ is a positive integer,
this definition coincides exactly with the space of $L^2$ functions
having $s$ distributional derivatives also lying in $L^2.$ We note
that the operation of multiplication by a Schwartz function is a
bounded map on each $H^s;$ this is most easily proved by interpolation
arguments similar to (but easier than) those alluded to in
Exercise~\ref{exercise:interpolation}---cf.\ \cite{Taylor:PDE1}.

Throughout these notes we will take for granted the \emph{Schwartz
  kernel theorem}, not so much as a result to be quoted but as a
world-view.  Recall that this result says \emph{any} continuous linear
operator
$$
\mathcal{S}(\RR^n)  \to \mathcal{S}'(\RR^n)
$$
is of the form
$$
u \mapsto \int k(x,y) u(y) \, dy
$$
%%%%%%%%%%% Think at all about the result on mflds?
for a unique $k \in \mathcal{S}' (\RR^n \times \RR^n);$ a corresponding
result also holds on all the manifolds that we will consider.  We thus
consistently take the liberty of confusing operators with their
Schwartz kernels, although we let $\kappa(A)$ denote the Schwartz
kernel of the operator $A$ when we wish to emphasize the difference.

Some results relating Schwartz kernels to traces are important for our
discussion of the wave trace.  Recall that an operator $T$ on a
separable Hilbert space is called \emph{Hilbert-Schmidt} if
$$
\sum_j \norm{Te_j }^2<\infty
$$
where $\{e_j\}$ is any orthonormal basis.  In the special case when
our Hilbert space is $L^2(X)$ with $X$ a manifold, the
condition to be Hilbert-Schmidt turns out to be easy to verify in
terms of the Schwartz kernel: $T$ is Hilbert-Schmidt if and only if
$\kappa(T),$ its Schwartz kernel,\footnote{It is probably best to
  think of $X$ as a \emph{Riemannian} manifold here, so that the
  Schwartz kernel is a function, which we can integrate against test
  functions via the metric density, and likewise integrate the
  kernel.} lies in $L^2(X\times X).$

A \emph{trace-class}
operator is one such that
$$
\sum_{i,j} \lvert \ang{T e_i, f_j} \rvert<\infty
$$
for every pair of orthormal bases $\{e_i\}, \{f_j\}.$  It turns out to
be the case that an operator $T$ is trace-class if and only if it can
be written
$$
T=PQ
$$
with $P,Q$ Hilbert-Schmidt.  The \emph{trace} of a trace-class
operator is given by
$$
\sum_i \ang{T e_i, e_i}
$$
over an orthonormal basis: this turns out to be well-defined.
We refer the reader to
\cite{Reed-Simon:v1} for further discussion of trace-class and
Hilbert-Schmidt operators.


\begin{thebibliography}{XX}
\bibitem{CKS}  Craig, W., Kappeler, T., Strauss, W. \emph{Microlocal dispersive smoothing for the Schr\"odinger equation,} Comm. Pure Appl. Math. 48 (1995), no. 8, 769--860. 
\bibitem{DimassiSjostrand}
Dimassi, Mouez; Sj\"ostrand, Johannes, \emph{Spectral asymptotics in the semi-classical limit,} London Mathematical Society Lecture Note Series, 268. Cambridge University Press, Cambridge, 1999.
\bibitem{Duistermaat-Guillemin} Duistermaat,
  J. J.; Guillemin, V. W. \emph{The spectrum of positive elliptic operators
  and periodic bicharacteristics,} Invent. Math. 29 (1975), no. 1,
  39--79.
\bibitem{Doi} Doi, S.-I. \emph{Smoothing effects of Schr\"odinger evolution groups on Riemannian manifolds,} Duke Math. J. 82 (1996), no. 3, 679--706.
\bibitem{Duistermaat-Hormander:FIO2} Duistermaat, J. J.; H\"ormander,
  L. \emph{Fourier integral operators. II,} Acta Math. 128 (1972),
  no. 3-4, 183--269.
\bibitem{FriedlanderJoshi} Friedlander, F. G.,
\emph{Introduction to the theory of distributions}
Second edition. With additional material by M. Joshi. Cambridge
University Press, Cambridge, 1998.
\bibitem{GrigisSjostrand} Grigis, A. and Sj\"ostrand, J.,
\emph{Microlocal analysis for differential operators. An
introduction.} London Mathematical Society Lecture Note Series,
196. Cambridge University Press, Cambridge, 1994.
\bibitem{htw1} \emph{A Strichartz inequality for the Schr\"odinger equation
on non-trapping asymptotically conic manifolds} (with Andrew Hassell and
  Terence Tao), Comm. PDE., \textbf{30} (2005), 157--205.
\bibitem{Hormander:Spectral} H\"ormander, L., \emph{The spectral function of an elliptic operator,}  Acta Math.  121  (1968), 193--218.
\bibitem{Hormander:FIO1} L. H\"ormander, \emph{Fourier Integral
    Operators I}, Acta Math. 127 (1971), 79--183.
\bibitem{Hormander:v1} H\"ormander, L. \emph{The analysis of linear partial
  differential operators. I. Distribution theory and Fourier
  analysis.} Second edition. Grundlehren der Mathematischen
  Wissenschaften,
  256. Springer-Verlag, Berlin, 1990.
\bibitem{Hormander:v2} H\"ormander, L. \emph{The analysis of linear
    partial differential operators. II. Differential operators with
    constant coefficients.} Grundlehren der Mathematischen
  Wissenschaften,
  257. Springer-Verlag, Berlin, 1983.
\bibitem{Hormander:v3} H\"ormander, L. \emph{The analysis of linear
    partial differential operators. III. Pseudodifferential
    operators.} Grundlehren der Mathematischen Wissenschaften,
  274. Springer-Verlag, Berlin, 1985.
\bibitem{Hormander:v4} H\"ormander, L. \emph{The analysis of linear
    partial differential operators. IV. Fourier integral
    operators.}  Grundlehren der
    Mathematischen Wissenschaften, 275. Springer-Verlag, Berlin, 1985.
\bibitem{Kac} Kac, M. \emph{Can one hear the shape of a drum?}
  Amer. Math. Monthly  73  1966 no. 4, part II, 1--23. 
\bibitem{Martinez}
Martinez, Andr\'e, \emph{An introduction to semiclassical and microlocal analysis,} Universitext. Springer-Verlag, New York, 2002.
\bibitem{Melrose:notes} R. Melrose {\sl Lecture notes on microlocal
    analysis,} available
  at \url{www-math.mit.edu/~rbm/Lecture_Notes.html}
\bibitem{Melrose:spectral} 
R.~B. Melrose, \emph{Spectral and scattering theory for the {L}aplacian on
  asymptotically {E}uclidian spaces}, Spectral and scattering theory (Sanda,
  1992), Dekker, New York, 1994, pp.~85--130.
\bibitem{Melrose-Zworski}
R. B. Melrose and M. Zworski, \emph{Scattering metrics and geodesic flow
  at infinity}, Invent. Math. \textbf{124} (1996), no.~1-3, 389--436.
\bibitem{Reed-Simon:v1}
Reed, Michael and Simon, Barry,  \emph{Methods of modern mathematical
  physics I: Functional analysis} Second edition, Academic Press, Inc., New York, 1980.
\bibitem{Seeley}
Seeley, R. T.,
\emph{Complex powers of an elliptic operator,} 1967 Singular Integrals (Proc. Sympos. Pure Math., Chicago, Ill., 1966) 288--307 Amer. Math. Soc., Providence, R.I. 
\bibitem{Shubin}
Shubin, M. A., \emph{Pseudodifferential operators and spectral
  theory,}  Second edition. Springer-Verlag, Berlin, 2001.
\bibitem{Stein:singular} Stein, E. M.
\emph{Singular integrals and differentiability properties of functions}
Princeton Mathematical Series, No. 30 Princeton University Press,
Princeton, N.J. 1970.
\bibitem{Taylor:pseudors} Taylor, M. E.
\emph{Pseudodifferential operators,}
Princeton Mathematical Series, 34. Princeton University Press, Princeton, N.J., 1981.
\bibitem{Taylor:PDE1} Taylor, M. E., \emph{Partial differential equations. I. Basic theory} Applied Mathematical Sciences, 115. Springer-Verlag, New York, 1996.
\bibitem{Taylor:PDE2} Taylor, M. E. \emph{Partial differential equations. II. Qualitative studies of linear equations} Applied Mathematical Sciences, 116. Springer-Verlag, New York, 1996.
\bibitem{Vasy:AdS} A. Vasy, \emph{The wave equation on asymptotically
    Anti-de Sitter spaces,} Anal. PDE, to appear.
\bibitem{Zworski} M. Zworski, \emph{Semiclassical
    analysis}, AMS Graduate Studies in Mathematics, American
  Mathematical Society, Providence, 2012.
\end{thebibliography}
\end{document}